\documentclass[11pt,reqno,a4paper,english]{amsart}

\usepackage[utf8]{inputenc}  
\usepackage{dsfont}
\usepackage{xcolor}

\usepackage{enumerate}          
\usepackage{amssymb,amsmath,amsfonts,amsthm}            
\usepackage{pdfsync}
\usepackage{url}
\usepackage{hyperref}
\usepackage{soul}
\usepackage{anysize}
\usepackage{mathrsfs}
\usepackage{graphicx,float,color}
\usepackage{epsfig}
\usepackage{tikz}
\usepackage{caption}
\usetikzlibrary{patterns, patterns.meta}

\begin{document}
\theoremstyle{plain}
\newtheorem{theorem}{Theorem}[section]
\newtheorem{definition}[theorem]{Definition}
\newtheorem{proposition}[theorem]{Proposition}
\newtheorem{lemma}[theorem]{Lemma}

\newtheorem{corollary}[theorem]{Corollary}
\newtheorem{conjecture}[theorem]{Conjecture}
\newtheorem{remark}[theorem]{Remark}
\newtheorem{open}[theorem]{Open question}
\numberwithin{equation}{section}
\newtheorem*{claim}{Claim}
\newtheorem{assumption}[theorem]{Assumption}
\errorcontextlines=0

\newcommand{\Id}{\text{Id}}
\newcommand{\mnum}{A}
\newcommand{\mrat}{B}
\newcommand{\R}{\mathbb{R}}
\newcommand{\Z}{\mathbb{Z}}
\newcommand{\N}{\mathbb{N}}
\newcommand{\X}{\mathcal{X}}
\newcommand{\supp}{\text{supp}}
\newcommand{\red}{\color{red}}
\newcommand{\black}{\color{black}}
\newcommand{\A}{A}
\newcommand{\ppow}{\beta}
\newcommand{\gau}{\ppow}
\newcommand{\Cor}{\mathcal{C}}
\renewcommand{\H}{\mathcal{H}}
\newcommand{\B}{B}
\newcommand{\Ball}{\mathcal{B}}
\newcommand{\C}{C}
\newcommand{\D}{D}
\newcommand{\E}{E}
\newcommand{\ct}{5\B\sqrt{d}}
\newcommand{\pow}{p}
\newcommand{\powun}{p_1}
\newcommand{\powde}{p_2}
\newcommand{\pds}{\delta}
\renewcommand{\P}{\mathbb{P}}
\renewcommand{\S}{\mathcal{S}}
\newcommand{\dist}{{\rm dist}}
\newcommand{\spt}{{\rm spt}}
\newcommand{\diam}{{\rm diam}}
\newcommand{\Var}{{\rm Var}}
\newcommand{\inte}{{\rm int}}
\newcommand{\Per}{{\rm Per}}
\newcommand{\vol}{{\rm vol}}
\newcommand{\CDM}{C_{\rm DM}}
\newcommand{\Res}{\mathcal{B}}
\newcommand{\Kant}{\mathcal K}
\newcommand{\YSp}{\mathcal Y}
\newcommand{\Rsp}{\mathbb R}
\newcommand{\eps}{\varepsilon}
\newcommand{\Class}{\mathcal C}
\newcommand{\Conv}{\mathrm{Conv}}
\newcommand{\sca}[2]{\langle #1\vert #2\rangle}
\newcommand{\nr}[1]{\Vert #1\Vert}
\newcommand{\restr}[1]{\left. #1\right\vert}

\setcounter{tocdepth}{1}

\title[Gluing methods for quantitative stability of optimal transport maps]{Gluing methods for quantitative stability\\ of optimal transport maps\\ \footnotesize Méthodes de recollement pour la stabilité quantitative des applications de transport optimal}

\author{Cyril Letrouit}\address{Cyril Letrouit. Université Paris-Saclay, CNRS, Laboratoire de mathématiques d’Orsay, 91405, Orsay, France} \email{cyril.letrouit@universite-paris-saclay.fr}

\author{Quentin Mérigot}\address{Quentin Mérigot. Université Paris-Saclay, CNRS, Inria, Laboratoire de mathématiques d’Orsay, 91405, Orsay, France /
DMA, École normale supérieure, Université PSL, CNRS, 75005 Paris, France  / Institut universitaire de France (IUF)} \email{quentin.merigot@universite-paris-saclay.fr}
\date{\today}
\maketitle

\begin{abstract}
We establish quantitative stability bounds for the quadratic optimal transport map $T_\mu$ between a fixed probability density $\rho$ and a probability measure $\mu$ on $\R^d$. Under general assumptions on $\rho$, we prove that the map $\mu\mapsto T_\mu$ is bi-Hölder continuous, with dimension-free Hölder exponents. The linearized optimal transport metric $W_{2,\rho}(\mu,\nu)=\|T_\mu-T_\nu\|_{L^2(\rho)}$ is therefore bi-Hölder equivalent to the $2$-Wasserstein distance, which justifies its use in applications. 

We show this property in the following cases: (i) for any log-concave density $\rho$ with full support in $\R^d$, and any log-bounded perturbation thereof; (ii) for $\rho$ bounded away from $0$ and $+\infty$ on a John domain (e.g., on a bounded Lipschitz domain), while the only previously known result of this type assumed convexity of the domain; (iii) for some important families of probability densities on bounded domains which decay or blow-up polynomially near the boundary. Concerning the sharpness of point (ii), we also provide examples of non-John domains for which the Brenier potentials do not satisfy any Hölder stability estimate. 

Our proofs rely on local variance inequalities for the Brenier potentials in small convex subsets of the support of $\rho$, which are glued together to deduce a global variance inequality. This gluing argument is based on two different strategies of independent interest: one of them leverages the properties of the Whitney decomposition in bounded domains, the other one relies on spectral graph theory.
\end{abstract}

\begin{abstract}
Nous établissons des bornes de stabilité quantitatives pour l’application de transport optimal quadratique $T_\mu$ entre une densité de probabilité fixée $\rho$ et une mesure de probabilité $\mu$ sur $\R^d$. Sous des hypothèses générales sur $\rho$, nous montrons que l’application $\mu\mapsto T_\mu$ est bi-höldérienne, avec des exposants de Hölder indépendants de la dimension. La distance de transport optimal linéarisée $W_{2,\rho}(\mu,\nu)=\|T_\mu-T_\nu\|_{L^2(\rho)}$
est ainsi équivalente de façon bi-höldérienne à la distance de Wasserstein $W_2(\mu,\nu)$, ce qui justifie son utilisation en pratique.

Nous établissons cette propriété dans les cas suivants : (i) pour toute densité log-concave $\rho$ à support plein dans $\R^d$, ainsi que pour toute perturbation log-bornée de celle-ci ; (ii) pour $\rho$ bornée inférieurement  et supérieurement sur un domaine de John (par exemple, sur un domaine lipschitzien borné), alors que le seul résultat antérieur de ce type supposait la convexité du domaine ; (iii) pour certaines familles importantes de densités de probabilité sur des domaines bornés, qui décroissent ou explosent de manière polynomiale au voisinage du bord. Concernant le caractère optimal du point (ii), nous fournissons également des exemples de domaines qui ne sont pas de John pour lesquels les potentiels de Brenier ne satisfont aucune estimation de stabilité höldérienne.

Nos preuves reposent sur des inégalités locales de variance pour les potentiels de Brenier dans de petits sous-ensembles convexes du support de $\rho$, que nous recollons afin d’obtenir une inégalité de variance globale. Cet argument de recollement s’appuie sur deux stratégies distinctes, présentant un intérêt propre : l’une exploite les propriétés de la décomposition de Whitney dans les domaines bornés, l’autre repose sur la théorie spectrale des graphes.
\end{abstract}

\tableofcontents

\section{Introduction}
Let $\mathcal{P}_2(\R^d)$ be the set of probability measures with finite second moment over $\R^d$. Given two probability measures $\rho,\mu$ in $\mathcal{P}_2(\R^d)$, the optimal transport problem for the quadratic cost in $\R^d$ consists in solving the minimization problem
\begin{equation}\label{e:opttransquad}
\inf_{\gamma\in \Gamma(\rho,\mu)} \int_{\R^d\times\R^d}\|x-y\|^2d\gamma(x,y)
\end{equation}
where $\Gamma(\rho,\mu)$ is the set of couplings between $\rho$ and $\mu$, i.e., the set of probability measures $\gamma$ over the product space $\R^d\times \R^d$ with first marginal $\rho$ and second marginal $\mu$. More generally, considering the $L^p$ cost, we define the $p$-Wasserstein distance between $\rho$ and $\mu$ as
$$
W_p(\rho,\mu)=\left(\inf_{\gamma\in \Gamma(\rho,\mu)} \int_{\R^d\times\R^d}\|x-y\|^pd\gamma(x,y)\right)^{1/p}.
$$
In this paper, we focus on the quadratic optimal transport problem in $\R^d$, i.e., the case $p=2$. A theorem of Brenier \cite{brenier} asserts that if $\rho$ is absolutely continuous with respect to the Lebesgue measure, then a unique solution $\gamma\in\Gamma(\rho,\mu)$ of \eqref{e:opttransquad} exists, and it is induced by a map $T=\nabla \phi$, where $\phi:\R^d\rightarrow \R$ is a convex function. 

\begin{definition}[Potentials and maps] 
We fix a probability measure $\rho\in\mathcal{P}_2(\R^d)$, which we assume to be absolutely continuous with respect to the Lebesgue measure and supported in $\mathcal{X}\subset\R^d$. Given $\mu\in\mathcal{P}_2(\R^d)$, we call
\begin{itemize}
\item \emph{Brenier map} and denote by $T_\mu$ the (unique) optimal transport map between $\rho$ and $\mu$;
\item \emph{Brenier potential} the unique lower semi-continuous convex function $\phi_\mu\in L^2(\rho)$ such that $T_\mu=\nabla \phi_\mu$ and  $\int_{\mathcal{X}}\phi_\mu d\rho=0$ (this function is always uniquely defined in the setting of this paper, see for instance Section  \ref{s:uniqueness}).
\end{itemize}
\end{definition}

The problem addressed in the present paper is the following: $\rho$ being fixed, and knowing that $\mu$ and $\nu$ are close to one another in some Wasserstein distance, how far can $T_\mu$ and $T_\nu$ be at most (e.g., in $L^2(\rho)$)? The stability of Brenier maps under variation of the measures is indeed fundamental from a mathematical viewpoint: it is closely related to the convergence of numerical approaches to solve optimal transport problems and justifies many of the applications of optimal transport.

Since we are dealing with the quadratic cost, the most natural distance on $\mathcal{P}_2(\R^d)$ to consider is the $2$-Wasserstein distance. To summarize, we are interested in quantitative stability bounds for the map $\mu \mapsto T_\mu$ from $(\mathcal{P}_2(\R^d),W_2)$ to $L^2(\rho)$,
i.e., bounds of the form
\begin{equation}\label{e:cequonveut}
\|T_\mu-T_\nu\|_{L^2(\rho)}\leq CW_2(\mu,\nu)^q
\end{equation}
where the constant $q$ is universal, and in particular independent of $d$. Although this is not essential (see Remark \ref{r:noncompacttarget}), we assume in our main results that the supports of $\mu,\nu$ are contained in a compact set $\mathcal{Y}$; the constant $C$ may depend on this set $\mathcal{Y}$ (and on $\rho$), but not on any other property of $\mu,\nu$.

Some comments are in order. Firstly, it is well-known that the map $\mu\mapsto T_\mu$ from $(\mathcal{P}_2(\R^d),W_2)$ to $L^2(\rho)$ is continuous, see \cite[Theorem 1.3]{brenierCPAM}. Secondly, since $(T_\mu,T_\nu)_{\#}\rho$ is a coupling between $\mu$ and $\nu$, this map is reverse Lipschitz, i.e., $\|T_\mu-T_\nu\|_{L^2(\rho)}\geq W_2(\mu,\nu)$.  Thirdly, simple examples show that one cannot expect in general $q>1/2$ in \eqref{e:cequonveut} (see \cite[Section 4]{gigli} or \cite[Lemma 5.1]{delmerchaz}). It is also natural to seek for stability bounds on the Brenier potentials, in the form
\begin{equation}\label{e:cequonveutpot}
\|\phi_\mu-\phi_\nu\|_{L^2(\rho)}\leq C'W_2(\mu,\nu)^{q'}.
\end{equation}
As we will see later, \eqref{e:cequonveutpot} implies  a bound of the form \eqref{e:cequonveut}.

Despite its important theoretical and practical interest, the quantitative stability of optimal transport maps has attracted attention only recently (\cite{gigli}, \cite{figdephil}, \cite{berman}, \cite{li}, \cite{delalande}, \cite{delmer}, \cite{delmerchaz}). Variants of this problem include quantitative stability for entropic optimal transport under perturbations of the measures (\cite{deligiannidis}, \cite{eckstein}, \cite{chizat}) or of the regularization parameter (\cite{delalande}, \cite{aram}), for semi-discrete optimal transport (\cite{bansil}, \cite{aram}), or for more general costs (\cite{mischlertrevisan}, \cite{thibert}). We also mention \cite{barycenters} which studies the quantitative stability of Wasserstein barycenters.

One of the motivations for studying the stability of optimal transport maps is that for a fixed source probability density $\rho$, the mapping $\mu \mapsto T_{\mu}$ provides an embedding of $(\mathcal{P}_2(\mathbb{R}^d),W_2)$ to the Hilbert space $L^2(\rho,\R^d)$. This embedding, often called \emph{linearized optimal transport}, allows one to apply the standard ``Hibertian" statistical toolbox to measure-valued data such as grayscale images \cite{slepcev}, \cite{kolouri}, \cite{basu}. This embedding is always distance-increasing, and stability estimates such as \eqref{e:cequonveut} show that it is bi-Hölder continuous. In other words, the distance $d(\mu,\nu) = \Vert{T_\mu - T_\nu}\Vert_{L^2(\rho)}$ preserves in a rough way the geometry associated to the Wasserstein distance. In dimension $d=1$, the embedding $\mu\mapsto T_\mu$ is an isometry, see for instance \cite[Theorem 2.18]{villani}.

Our starting point is the following result of \cite{delmer}. Here as in the rest of the paper, constants denoted by $C_{a_1,\ldots,a_n}$ are non-negative constants which depend on $a_1,\ldots,a_n$. Since the source measure $\rho$ lives in $\R^d$, constants depending on $\rho$ implicitly depend on $d$.
\begin{theorem}[Introduction of \cite{delmer}]\label{t:delmer} Let $\mathcal{X}\subset\R^d$ be a compact convex set and $\rho$ be a probability density on $\mathcal{X}$, bounded from above and below by positive constants. Let $p>d$ and $p\geq 4$. Assume that $\mu,\nu\in\mathcal{P}_2(\R^d)$ have bounded $p$-th moment, i.e. $\max(\int_{\R^d}\|x\|^pd\mu(x),\int_{\R^d}\|x\|^pd\nu(x))\leq M_p <+\infty$. Then
\begin{align}
\|T_\mu-T_\nu\|_{L^2(\rho)}&\leq C_{\rho,p,M_p}W_1(\mu,\nu)^{\frac{p}{6p+16d}}\nonumber\\
\|\phi_\mu-\phi_\nu\|_{L^2(\rho)}&\leq  C_{\rho,p,M_p}W_1(\mu,\nu)^{\frac12}.\label{e:delllmerrr}
\end{align}
If $\mu,\nu$ are supported on a compact set $\mathcal{Y}$, the Hölder exponent for the Brenier map is improved:
\begin{equation}\label{e:TmuTnustab}
\|T_\mu-T_\nu\|_{L^2(\rho)}\leq C_{\rho,\mathcal{Y}}W_1(\mu,\nu)^{\frac16}.
\end{equation}
\end{theorem}
\begin{remark}[Comparison between $W_1$ and $W_2$]
We note that since $W_1\leq W_2$, the estimates in the above theorems, as well as in all results of the present paper, imply a bi-Hölder behaviour of the map $\mu\mapsto T_\mu$ on subsets of $\mathcal{P}_2(\R^d)$ with respect to both Wasserstein distances $W_1$ and $W_2$. 
\end{remark}
Our purpose in this paper is to obtain similar stability results for much more general source probability measures $\rho$. Along the way, we unveil a strong analogy between optimal transport stability estimates and some proofs of Poincaré(-Wirtinger) inequalities. We also provide an example of $\rho$ for which stability of Brenier potentials fails. All together, these results give a fairly general picture of the source measures $\rho$ for which quantitative stability of optimal transport potentials and maps may be expected to hold. 

In all previous works on the subject (notably \cite{berman}, \cite{delmerchaz}, \cite{delmer}, \cite[Section 4]{aram}, \cite{mischlertrevisan}), convexity, compactness and upper and lower bounds on $\rho$ were used in a crucial way. Our results show that neither the compactness,  nor the convexity of $\mathcal{X}$ are actually needed for stability to hold, and that the density $\rho(x)$ may also be allowed to tend to $0$ or to $+\infty$ at some controlled rate, either for $x$ close to the boundary of $\mathcal{X}$ if $\mathcal{X}$ is bounded, or as $|x|\rightarrow +\infty$ if $\mathcal{X}$ is unbounded. 

We prove that the exponents we obtain for the stability of Brenier potentials are sharp in some cases, see for instance Theorem \ref{t:powlawdist} and Proposition \ref{p:sharpexpo}. This is the first time that stability exponents are proved to be sharp: for instance it is not known whether the exponents $1/2$ in \eqref{e:delllmerrr} and $1/6$ in \eqref{e:TmuTnustab} are sharp or not. 

The proofs of our main results leverage domain decomposition techniques and spectral graph theory. The starting point of our investigations was the proof of \cite[Proposition B.2]{barycenters} in which it was first noticed, although on the different problem of quantitative stability of Wasserstein barycenters, that convexity of the support of $\rho$ might not be necessary to prove quantitative stability bounds. 

\subsection{The log-concave case}

We first establish stability estimates in the case where the source measure $\rho$ is a log-concave probability density on $\R^d$, with support equal to $\R^d$. As in the Holley-Stroock principle (\cite[page 1185]{holley}, \cite[Proposition 4.2.7]{BGL}), we are also able to deal with log-bounded perturbations of log-concave measures. Our first main result is the following: 

\begin{theorem}\label{t:logconcave}
Let $\rho=e^{-U-F}dx$  be a probability density on $\R^d$, with $D^2U\geq \kappa\cdot {\rm Id}$ for some $\kappa>0$, and $F\in L^\infty(\R^d)$. Let $\mathcal{Y}\subset\R^d$ be a compact set. Then, there exists $C_{\rho,\mathcal{Y}}>0$ such that for any probability measures $\mu,\nu$ supported in $\mathcal{Y}$, 
\begin{equation}
\|\phi_\mu-\phi_\nu\|_{L^2(\rho)}\leq C_{\rho,\mathcal{Y}}W_1(\mu,\nu)^{\frac12}(1+\left|\log(W_1(\mu,\nu))\right|)^{\frac12}\label{e:stabpotlogconc}
\end{equation}
If additionally $D^2U\leq \kappa'\cdot{\rm Id}$ then there exists $C_{\rho,\mathcal{Y}}>0$ such that for any probability measures $\mu,\nu$ supported in $\mathcal{Y}$, 
\begin{equation}
\|T_\mu-T_\nu\|_{L^2(\rho)}\leq C_{\rho,\mathcal{Y}}W_1(\mu,\nu)^{\frac{\kappa}{2\kappa'+7\kappa}}.\label{e:stabmaplogcon}
\end{equation}
\end{theorem}
To our knowledge, Theorem \ref{t:logconcave} is the first result establishing stability estimates for probability densities with unbounded support in $\R^d$ (see for instance \cite{mischlertrevisan}
 where this is mentioned as an open problem in Section 1.4, and where the case of log-concave measures with bounded support is handled). The inequality \eqref{e:stabpotlogconc} is sharp, up to the log factor, see Proposition \ref{p:sharpexpo}.
Of course, the log-loss in \eqref{e:stabpotlogconc} may be replaced by $W_1(\mu,\nu)^{-\varepsilon}$ for arbitrarily small $\varepsilon>0$, at the price of introducing a dependence of $C_{\rho,\mathcal{Y}}$ in $\varepsilon$. Also, the exponent $\frac{\kappa}{2\kappa'+7\kappa}$ in \eqref{e:stabmaplogcon} can be replaced by $\frac{\kappa}{2\kappa'+6\kappa}-\varepsilon$ for any $\varepsilon>0$; and again, this introduces a dependence of the constant $C_{\rho,\mathcal{Y}}$ in $\varepsilon$. The constants in Theorem \ref{t:logconcave}, as well as in all other results of this paper, can be made explicit.

 \begin{remark}\label{r:noncompacttarget}
Theorem \ref{t:logconcave}, as well as the results that we describe below, are stated only for target measures which are supported in a compact set $\mathcal{Y}$. This choice is mainly made for simplicity and to keep the paper readable, since handling unbounded targets would create an additional layer of complexity in the proofs. We refer to \cite[Section 4]{delmer} where this extension is done for $\rho$ satisfying the assumptions of Theorem \ref{t:delmer}.
\end{remark}
 
\subsection{The case of John domains}

Our second result extends Theorem \ref{t:delmer} to the case where $\rho$ is a probability density on a John domain, whose density is assumed to be bounded above and below by positive constants. Roughly speaking, a domain is a John domain if it is possible to move from one point to another without passing too close to the boundary. This notion is encountered for instance when dealing with Sobolev-Poincaré inequalities. Let us recall the precise definition and then provide examples.


\begin{definition}\label{d:john}
A bounded domain $\X\subset\R^d$ is called a John domain if for some point $x_0\in\X$, some $\alpha>0$, and for any $x\in\X$, there exists a curve $\gamma:[0,\ell]\rightarrow \X$ parametrized by arc-length such that $\gamma(0)=x$, $\gamma(\ell)=x_0$, and
\begin{equation}\label{e:foralltinellgammat}
\forall t\in[0,\ell], \quad \dist(\gamma(t),\partial \X)\geq \alpha t
\end{equation}
where $\dist$ denotes the Euclidean distance. 
\end{definition}

John domains were introduced in \cite{john}, and the terminology was coined in \cite{martio}. Lipschitz domains, and more generally domains satisfying the cone condition, are John domains. John domains may have fractal (and nonrectifiable) boundaries or internal cusps, but external cusps are excluded. The snowflake domains, domains bounded by a Koch curve and bounded quasidiscs of two­dimensional quasiconformal theory are John domains. We refer to \cite{nakki} and \cite[Chapter 5]{pommerenke} for accounts on John domains.

The main result of this section is the following:
\begin{theorem}\label{t:mainjohn}
Let $\X\subset\R^d$ be a John domain and let $\rho$ be a probability density on $\X$, bounded from above and below by positive constants. Then, for any compact set $\mathcal{Y}$, there exists $C_{\rho,\mathcal{Y}}>0$ such that for any probability measures $\mu,\nu$ supported in $\mathcal{Y}$,
\begin{equation}\label{e:stabpotjohn}
\|\phi_\mu-\phi_\nu\|_{L^2(\rho)}\leq C_{\rho,\mathcal{Y}}W_1(\mu,\nu)^{\frac12}.
\end{equation}
If in addition $\X$ has a rectifiable boundary with finite $(d-1)$-dimensional Hausdorff measure, then
\begin{equation}\label{e:stabmapjohn2}
\|T_\mu-T_\nu\|_{L^2(\rho)}\leq C_{\rho,\mathcal{Y}}W_1(\mu,\nu)^{\frac16}.
\end{equation}
\end{theorem}
The constants $C_{\rho,\mathcal{Y}}$ may be computed explicitly, see Lemma \ref{l:gluevarjohn} and Remark \ref{r:constants} for an illustration. In Section \ref{s:moredomains} we show that \eqref{e:stabmapjohn2} still holds if $\X$ is a finite union of disjoint John domains with rectifiable boundary.

Proving a converse to Theorem \ref{t:mainjohn} is a subtle problem; we shall not attempt here to give a general characterization of the probability densities $\rho$ on compact domains for which an inequality like \eqref{e:stabpotjohn} (or \eqref{e:stabmapjohn2}) holds. Nevertheless, to illustrate the relevance of the ``John-type" condition in Theorem \ref{t:mainjohn}, we show that in typical examples of non-John domains considered in the literature, the stability of Brenier potentials \eqref{e:stabpotjohn} fails.
\begin{definition}\label{d:holder}
We say that an absolutely continuous probability measure $\rho$ on $\R^d$ has the H\"older potential stability property if there exist $C,q>0$ and $p\in [1,\infty)$ such that for any $\mu,\nu$ supported in the unit ball of $\R^d$,
$$
\|\phi_\mu-\phi_\nu\|_{L^2(\rho)}\leq CW_p(\mu,\nu)^q.
$$
\end{definition}
\begin{theorem}\label{t:counterexample}
For any $d\geq 2$, there exists a non-empty, bounded and path-connected domain $\X\subset\R^d$ such that any probability density $\rho$ which is bounded from above and below by positive constants on $\X$ does not have the H\"older potential stability property.
\end{theorem}
At least two kinds of non-John domains $\X$ can be used to prove Theorem \ref{t:counterexample}, and they are probably the most classical examples of non-John domains. The proof we provide is based on the construction of a ``room-and-passage" domain $\X$, a class of bounded domains with thin necks at all scales already used in the literature to provide counterexamples to various spectral and functional properties known to hold in John domains (see \cite{amick}, \cite[pp. 521-523]{couranthilbert} for instance). In Remark \ref{e:autrecontreex} we explain that domains $\X$ with outward cusps may also be used to prove Theorem \ref{t:counterexample}.

\subsection{The case of degenerating densities in bounded domains}

Theorem \ref{t:mainjohn} shows that the convexity assumption in Theorem \ref{t:delmer} may be dramatically relaxed. Our next statements demonstrate that the assumption that the probability density $\rho$ is bounded above and below on its support is not necessary either: we may allow $\rho$ to decay to $0$ or blow-up to $+\infty$ at some controlled rate when approaching the boundary of $\X$, and still obtain comparable stability bounds. 

Our first result in this direction is the following. This result is new even for convex $\X$. 
\begin{theorem}\label{t:explosebord}
Let $\rho$ be a probability density over a bounded Lipschitz domain $\X\subset\R^d$. Assume that there exist $c_1,c_2>0$ and $\delta>-1$ such that for any $x\in\X$,
\begin{equation}\label{e:encadrement}
c_1 \dist(x,\partial\X)^{\delta} \leq \rho(x)\leq c_2 \dist(x,\partial\X)^{\delta}.
\end{equation}
Let $\mathcal{Y}\subset\R^d$ be a compact set. Then there exists $C_{\rho,\mathcal{Y}}>0$ such that for any probability measures $\mu,\nu$ supported in $\mathcal{Y}$,
\begin{align}
\|\phi_\mu-\phi_\nu\|_{L^2(\rho)}&\leq C_{\rho,\mathcal{Y}}W_1(\mu,\nu)^{\frac12}\label{e:explstabpoto}\\
\|T_\mu-T_\nu\|_{L^2(\rho)}&\leq C_{\rho,\mathcal{Y}} W_1(\mu,\nu)^{\frac16 - \delta'} \label{e:1-8C}
\end{align}
where  $\delta'=\frac{|\delta|}{6}$ if $-1<\delta\leq 0$, and $\delta'=\frac{\delta}{12(1+\delta)}$ if $\delta\geq 0$.
\end{theorem}
The above condition $\delta>-1$ is necessary since $\rho$ is a probability measure. In contrast with Theorem \ref{t:explosebord}, the stability of optimal transport maps can be totally lost if the blow-up of the density $\rho$ is too strong: we exhibit in \cite{letrouit} an unbounded density $\rho$ on the unit ball of $\R^d$ such that for any $C,q>0$, any $p\geq 1$ and any bounded set $\mathcal{Y}\subset\R^d$, the bound 
$$
\forall \mu,\nu\in\mathcal{P}(\mathcal{Y}), \qquad \|T_\mu-T_\nu\|_{L^2(\rho)}\leq CW_p(\mu,\nu)^q
$$
fails.

Our second result concerning densities degenerating at the boundary of a bounded domain is motivated by works in statistics (starting with \cite{galichon}) which define notions of multivariate quantiles and statistical depth via optimal transport theory. These works aim at addressing the well-known problem of finding  a good analogue in dimension $d\geq 2$ of univariate distribution functions, despite the absence of a canonical ordering relation in $\R^d$. Compared to previous notions considered in the literature, the notion introduced in \cite{galichon} and further studied e.g. in \cite{hallin}, \cite{figalli}, enjoys  inferential properties expected from distribution and quantile functions in $\R^d$. Given a probability density $\mu$, the vector quantile map defined in \cite{galichon} is the quadratic optimal transport map from the spherical uniform distribution $\rho$ defined in Theorem \ref{t:hallin} to $\mu$: the quantiles of $\mu$ are defined by pushing forward the quantiles of $\rho$. The main theoretical result of \cite{galichon}, namely \cite[Theorem 3.1]{galichon}, establishes that for any compactly supported probability density $\mu$, the empirical vector quantile maps defined via empirical approximations of $\mu$ converge to the vector quantile map associated to $\mu$ (actually, \cite[Theorem 3.1]{galichon} shows that convergence holds even if $\rho$ itself is replaced by empirical approximations). Our next statement, when applied to an empirical approximation $\nu$ of $\mu$, quantifies the rate of convergence:

\begin{theorem}\label{t:hallin}
Let $\X=B(0,1)\setminus\{0\}$ where $B(0,1)$ is the closed unit ball of $\R^d$. Consider the probability density $\rho(x)=c_d|x|^{1-d}\chi_{\X}(x)$ where $\chi_{\X}$ is the characteristic function of $\X$ and $c_d$ is a normalizing constant. Let $\mathcal{Y}\subset\R^d$ be a compact set. Then there exists $C_{d,\mathcal{Y}}>0$ such that for any probability measures $\mu,\nu$ supported in $\mathcal{Y}$,
\begin{align}
\|\phi_\mu-\phi_\nu\|_{L^2(\rho)}&\leq C_{d,\mathcal{Y}}W_1(\mu,\nu)^{\frac12}\label{e:hallinpot}\\
\|T_\mu-T_\nu\|_{L^2(\rho)}&\leq C_{d,\mathcal{Y}} W_1(\mu,\nu)^{\frac{1}{6d}}. \label{e:hallinmap}
\end{align}
\end{theorem}
The exponent in the upper bound \eqref{e:hallinmap} depends on $d$, and we do not know if it possible to get rid of this dependence. We explain in Remark \ref{r:expl} that our proof techniques could potentially handle other families of probability measures $\rho$ whose density tends to $0$ or $+\infty$ than those considered in Theorems \ref{t:explosebord} and \ref{t:hallin}. We chose the above statements for their simplicity and their relevance regarding applications.

 
\subsection{Generalized Cauchy distributions}\label{s:exceptions}
We finally establish stability estimates for Brenier potentials and Brenier maps in the case where $\rho$ belongs to a family of probability measures on $\R^d$ which are not log-concave. Namely, we handle generalized Cauchy distributions of the form $\rho(x)=c(x)\langle x\rangle^{-\beta}$ where here and in the sequel
$$
\langle x\rangle = (1+|x|^2)^{1/2}.
$$
From now on, we assume that $\beta>d+2$. This ensures that $\rho$ has finite second moment, and is in particular necessary for the  left-hand side of \eqref{e:yeah} to be finite.

\begin{theorem}\label{t:powlawdist}
Assume that $\rho(x)=c(x)\langle x\rangle^{-\beta}$ on $\R^d$, with $\beta>d+2$ and $0<m\leq c(x)\leq M<+\infty$.  Let $\mathcal{Y}\subset\R^d$ be a compact set. Then, there exists $C_{\rho,\mathcal{Y}}>0$ such that for any probability measures $\mu,\nu$ supported in $\mathcal{Y}$,
\begin{align}
\|\phi_\mu-\phi_\nu\|_{L^2(\rho)}&\leq C_{\rho,\mathcal{Y}}W_1(\mu,\nu)^{\theta} \label{e:yeah}\\
\|T_\mu-T_\nu\|_{L^2(\rho)}&\leq C_{\rho,\mathcal{Y}}W_1(\mu,\nu)^{\theta'}\label{e:powlawmap}
\end{align}
where $\theta=\frac12 (1-\frac{2}{\beta-d})>0$ and $\theta'=\frac{\beta-d-2}{8\beta-2d-4}>0$. Moreover, the exponent $\theta=\frac12 (1-\frac{2}{\beta-d})$ in \eqref{e:yeah} is sharp.
\end{theorem}
With this result, we want in particular to challenge a natural idea that could arise from our results of the previous sections and their proofs, namely that there could exist a strict equivalence between stability estimates for Brenier potentials with source measure $\rho$ and the fact that the  Poincaré(-Wirtinger) inequality holds for $\rho$. But Theorem \ref{t:powlawdist} disproves this conjectural relation, since it is known that the Poincaré inequality does not hold for generalized Cauchy distributions, see e.g. \cite{bobkovledoux}. However, this does not rule out the possibility of a weaker relationship between the two properties: as emphasized in Section \ref{s:weightpoincare}, generalized Cauchy distributions satisfy a weighted Poincaré inequality, which may be recovered as a direct corollary of our proof. 

One more motivation for Theorem \ref{t:powlawdist} is that its proof requires to develop a different strategy compared to our other results, and we believe this strategy to be of independent interest. See Sections \ref{s:prooftechniques} and \ref{s:warmup} for details. This strategy may be used for instance to reprove weaker versions of Theorems \ref{t:logconcave} and \ref{t:mainjohn}; we shall not pursue this here. Refining this strategy, it is possible to prove analogous stability estimates for Brenier potentials and Brenier maps for other families of source measures which do not satisfy a Poincaré inequality, for instance $\rho(x)=c(x)e^{-\kappa |x|^{\alpha}}$ with $0<\alpha<1$, $\kappa>0$, and $0<m\leq c(x)\leq M<+\infty$. We do not pursue this here either, we only give a few more details in Remark \ref{r:exppowdist}.

\subsection{Proof strategies and organization of the paper}\label{s:prooftechniques}
The starting point of our proofs is to reinterpret the quantity $\|\phi_\mu-\phi_\nu\|^2_{L^2(\rho)}$ as a variance ${\rm Var}_\rho(\phi_\mu-\phi_\nu)$, due to the fact that both $\phi_\mu$ and $\phi_\nu$ have vanishing mean with respect to $\rho$. In Section \ref{s:kantofunct} we establish a variance inequality, i.e., an upper bound on this variance, in the case where $\rho$ is supported in a compact convex set and bounded above and below by positive constants. This variance inequality is a refinement of the results of \cite{delmer}. We provide a proof which simplifies the approach of \cite{delmer}. In Section \ref{s:vrailogconc} we deduce Theorem \ref{t:logconcave} via truncation arguments.

To prove the other results we follow a strategy which could as well be used to prove Poincaré-Wirtinger inequalities of the form ${\rm Var}_\rho(f)\leq C\int |\nabla f|^2 d\rho$ where here and in the sequel
$$
\Var_\rho(f)=\inf_{c\in \R}\int_{\R^d}|f-c|^2d\rho.
$$
Namely, we rely on two ingredients: firstly, the upper bound on the variance in convex sets established in Section \ref{s:kantofunct}; secondly, an argument to glue together ``local" variance inequalities in small convex sets and deduce an upper bound on the ``global" variance, i.e., the variance in $\X$. 

Our main contribution is related to the second ingredient: we extend already existing techniques and develop new methods to glue together ``local" variance inequalities. Overall, the leitmotiv of this work is that if a function $f$ does not vary much in each set $Q$ of a family $\mathcal{F}$, and if these sets intersect enough, then $f$ does not vary much in the union of the sets $Q\in\mathcal{F}$. 
To turn this rough idea into proofs, we use domain decompositions techniques which we describe below. Let us already mention that once the global variance inequality is proved, we deduce almost immediately the stability estimates on the difference of Brenier potentials $\|\phi_\mu-\phi_\nu\|_{L^2(\rho)}$ via Kantorovich-Rubinstein duality. The stability of Brenier maps requires more work, it essentially relies on reverse Poincaré inequalities for the difference of two convex functions proved in \cite{delmer} (i.e., $\|\nabla\phi_\mu-\nabla\phi_\nu\|_{L^2}\leq C\|\phi_\mu-\phi_\nu\|_{L^2}^\theta$).

A key point is thus to understand how to glue together variance inequalities. For this, we develop two strategies, which are compared in Section \ref{s:strategy}.

Our first method is set out in Section \ref{s:variancegen}, see notably Theorem \ref{t:varineqjohn}. To prove a global variance inequality, we take inspiration from the proof of Sobolev-Poincaré inequalities in John domains (\cite{bojarski}, \cite{sobolevmet}): we rely on the Whitney decomposition and its good properties in John domains, in particular the existence of a decomposition satisfying a Boman chain condition in the spirit of \cite{boman}, \cite{bojarski}. In Section \ref{s:proofjohn} we use this to prove Theorems \ref{t:mainjohn}, \ref{t:explosebord} and \ref{t:hallin}, which requires in each case to establish some doubling property for the probability measure $\rho$. The counterexample to Hölder potential stability in non-John domains used to prove Theorem \ref{t:counterexample} is provided in Section \ref{s:counterexample}. As already mentioned, it is based on the construction of an appropriate ``room-and-passage" domain.

Our second method for gluing variance inequalities is developed in Section \ref{s:proofexception} and allows us to prove Theorem \ref{t:powlawdist}. It consists in decomposing $\X$ into convex sets and constructing a weighted combinatorial graph from this decomposition, with one vertex for each convex set and one edge for each pair of overlapping convex sets. To be able to control the variance in $\X$ by the variances in the convex sets we need enough overlap between these sets, in other words that the weighted graph is sufficiently well connected. This connectedness is measured via the spectral gap of the graph Laplacian, which we prove to be strictly positive via a Cheeger inequality proved in Appendix \ref{a:cheeger}.

Apart from the families of $\rho$'s handled in Section \ref{s:exceptions}, all the $\rho$'s considered in the present work support an $L^2$-Poincaré-Wirtinger inequality. These Poincaré-Wirtinger inequalities are actually immediate corollaries of the techniques of this paper, see e.g. Section \ref{s:uniqueness}. It might seem natural to conjecture that if $\rho$ supports a Poincaré(-Wirtinger) inequality, then not only the Brenier potentials are unique (which is a consequence of connectedness of the support, see Section \ref{s:uniqueness}), but they are even stable. However, this is totally unclear at the present moment, and we know by Theorem \ref{t:powlawdist} that the converse is false.

Finally, we believe that the techniques of the present paper, which show that it is sufficient to establish local variance inequalities in order to prove global stability results, might foster progress in other directions, for instance regarding the stability of optimal transport maps for more general costs (and on Riemannian manifolds), the random matching problem and rates of convergence for the Sinkhorn algorithm.

\subsection{Notation} Throughout the paper, $\N$ denotes the set of non-negative integers. For any $n$ in the set of positive integers $\N^*$, the notation $[n]$ stands for $[n]=\{1,\ldots,n\}$. The characteristic function of a set $S\subset\R^d$ is denoted by $\chi_S$. The Euclidean distance in $\R^d$ is denoted by $\dist$. The Euclidean scalar product is denoted by  $\sca{\cdot}{\cdot}$, and this notation also stands for the duality pairing between continuous functions with compact support and real-valued Radon measures (the distinction between the two is clear from the context). The Euclidean (closed) ball of center $0$ and radius $r\geq 0$ is denoted by $\Ball_r:= B(0,r)$. The Lebesgue measure on $\R^d$ is denoted by $\lambda$. The support of a measure $\rho$ is denoted by ${\rm spt}(\rho)$. Absolutely continuous measures on $\R^d$ are sometimes identified with their density with respect to the Lebesgue measure.

\subsection{Acknowledgments} We thank an anonymous reviewer whose careful reading and detailed report helped us improve the manuscript in several places. We also thank Max Fathi, Piotr Haj\l{}asz and Radosław Wojciechowski for discussions related to this work. The authors acknowledge the support of the Agence nationale de la recherche, through the PEPR PDE-AI project (ANR-23-PEIA-0004). The first author would like to thank for its hospitality the Courant Institute of Mathematical Sciences in New York, where part of this work was done.

\section{Stability for log-concave sources}
\label{s:stablogconc}

\subsection{Variance inequality in compact convex sets} \label{s:kantofunct} 
In this section, we start with the case where $\rho$ is supported on a compact and convex set. We establish a variance inequality which is a key ingredient for most of the proofs in the present work. 

Given a function $\psi:\Rsp^d\to\Rsp\cup\{+\infty\},$ we recall that its convex conjugate is defined as
\begin{equation}\label{eq:conv-conjugate}
\psi^*(x) = \sup_{y\in\Rsp^d}\ \sca{x}{y} - \psi(y)
\end{equation}
If $\psi \in\Class^0(\YSp)$ where $\YSp\subset\R^d$ (thus $\psi$ is a priori not defined on $\Rsp^d$), we implicitly extend $\psi$ by $+\infty$ outside of $\YSp$ when computing $\psi^*$. This is equivalent to taking the supremum in \eqref{eq:conv-conjugate} over $y\in\YSp$.

The following statement is a minor modification of \cite{delmer}, \cite{mischlertrevisan}. Due to its importance in what follows, we provide here a complete proof.
\begin{theorem}
\label{t:stability-compact}
Let $Q\subset\R^d$ be a compact convex set with non-empty interior, let $\sigma$ be a logarithmically-concave probability density over $Q$ and let $\rho$ be another probability density over $Q$ satisfying $m_\rho \sigma \leq \rho \leq M_\rho \sigma$ for some constants $M_\rho\geq m_\rho>0$.
Let $\mathcal{Y}\subset\R^d$ be a compact set and set $R_\YSp = \max_{y\in\YSp} \nr{y}$.  Then, for all $\psi_0,\psi_1\in\Class^0(\YSp),$
\begin{equation} \label{eq:dm}
\sca{\psi_1 - \psi_0}{\nabla\psi_{0\#}^*\rho - \nabla\psi_{1\#}^*\rho}
\geq  e^{-1} \frac{m_\rho}{M_\rho}  \frac{1}{R_{\YSp} \diam(Q)} \Var_\rho(\psi_1^* - \psi_0^*).
\end{equation} 
\end{theorem}
An example to keep in mind is when $\sigma$ is the characteristic function of $Q$, normalized to be a probability density. Another important example is when $\rho$ itself is log-concave, in which case we may take $\sigma=\rho$ and $m_\rho/M_\rho=1$.

Before giving the proof of Theorem \ref{t:stability-compact} we recall the following result from \cite{delmer}, which we shall use several times in this work.
\begin{proposition}\cite[Proposition 4.1]{delmer}\label{p:GNdelmer}
Let $K$ be a compact domain of $\R^d$ with rectifiable boundary and let $u,v:K\rightarrow\R$ be two $L$-Lipschitz functions on $K$ that are convex on any segment included in $K$. Then there exists a constant $C_d$ depending only on $d$ such that
\begin{equation}\label{e:GNdelmer}
\|\nabla u-\nabla v\|_{L^2(\lambda,K)}^2\leq C_d\mathcal{H}^{d-1}(\partial K)^{2/3}L^{4/3}\|u-v\|_{L^2(\lambda,K)}^{2/3}
\end{equation}
where $\mathcal{H}^{d-1}$ denotes the $(d-1)$-dimensional Hausdorff measure and the norms in \eqref{e:GNdelmer} are taken with respect to the $d$-dimensional Lebesgue measure $\lambda$.
\end{proposition}
We turn to the proof of Theorem \ref{s:kantofunct}.
\begin{proof}[Proof of Theorem \ref{s:kantofunct}]
As in  \cite{delmer}, the proof mainly consists in proving a strong convexity property for the Kantorovich functional $\Kant_{\rho}:\mathcal{C}^0(\mathcal{Y})\rightarrow \R$,
$$\Kant_{\rho}(\psi) = \int_Q \psi^* d\rho.$$

\noindent{\emph{Step 1: Restriction to $Q$-convex functions.}}
In the following, we fix a closed ball $B$ containing $Q$ in its interior. 
A function $\psi: \Rsp^d \to \Rsp$ is called \emph{$Q$-convex} if it is equal to the convex conjugate of a convex and Lipschitz function $\phi$ defined on $B$, i.e. $\psi(y) = \max_{x\in B}\sca{x}{y} - \phi(x)$. In this case, Fenchel-Rockafellar's theorem implies that $\phi = \psi^*$ on $B$.  Note that if $\psi\in\Class^0(\YSp)$, then 
\begin{equation} \label{eq:app:extension}
    \hat{\psi}(y) = \left(\restr{\psi^*}_{B}\right)^*(y) = \max_{x\in B} \sca{x}{y} - \psi^*(x)
\end{equation} 
is $Q$-convex, since $\restr{\psi^*}_{B}$ is $\mathcal{R}_\YSp$-Lipschitz.

Our first step is to prove that the inequality \eqref{eq:dm} for any pair of $Q$-convex functions implies the same inequality for any pair of continuous functions on $\YSp$. To do so, we consider $\psi_0,\psi_1 \in \Class^0(\YSp)$ and we define $\hat{\psi}_0$ and $\hat{\psi}_1$ as in \eqref{eq:app:extension}. Since the convex functions $\hat{\psi}_i^*$ and $\psi_i^*$ agree on $Q$,  the right-hand side of the inequality \eqref{eq:dm} is unchanged under the transformation $\psi_i \to \hat{\psi}_i$:
\begin{equation} \label{eq:app:rhs-equal}
\begin{aligned}
    &\Var_{\rho}(\psi_1^* -\psi_0^*) = \Var_{\rho}(\hat{\psi}_1^* - \hat{\psi}_0^*).
\end{aligned}
\end{equation} 
We now deal with the duality pairing in the left-hand side of \eqref{eq:dm}. For any $x\in Q$ where $\psi_i^*$ is differentiable (that is, $\rho$-a.e. $x$), the maximum in the definition of the  convex conjugate is attained at $y = \nabla \psi_i^*(x)$. For such $x\in Q$, we have $\hat{\psi}_i(\nabla \psi_i^*(x)) = \psi_i(\nabla \psi_i^*(x))$. This implies that
$$ \sca{\psi_i}{(\nabla\psi_i^*)_{\#}\rho} 
= \int_{Q} \psi_i(\nabla\psi_i^*(x)) d \rho(x) 
=   \sca{\hat{\psi}_i}{(\nabla \hat{\psi}_i^*)_{\#}\rho}.$$
Moreover, using $\hat{\psi}_{1-i} \leq \psi_{1-i}$ and the fact that $\psi_i^*=\hat{\psi}_i^*$ in $Q$, we have 
$$ \sca{\psi_{1-i}}{(\nabla\psi_i^*)_{\#}\rho} \geq \sca{\hat{\psi}_{1-i}}{(\nabla\psi_i^*)_{\#}\rho} = \sca{\hat{\psi}_{1-i}}{(\nabla \hat{\psi}_i^*)_{\#}\rho}.$$
Summing the two previous inequalities with $i\in\{0,1\}$ we obtain
\begin{equation}\label{eq:app:lhs-equal}
\sca{\psi_1 - \psi_0}{\nabla \psi_{0\#}^*\rho - \nabla \psi_{1\#}^*\rho} 
\geq 
\sca{\hat{\psi}_1 - \hat{\psi}_0}{\nabla \hat{\psi}_{0\#}^*\rho - \nabla\hat{\psi}_{1\#}^*\rho}.
\end{equation}
Putting \eqref{eq:app:rhs-equal} and \eqref{eq:app:lhs-equal} together, we obtain that if  \eqref{eq:dm}  holds for the $Q$-convex functions $(\hat{\psi}_0,\hat{\psi}_1)$, then it also holds for the merely continuous functions $(\psi_0,\psi_1)$. 
\bigskip

\noindent{\emph{Step 2: Restriction to smooth and strongly convex functions.}}
We now consider two $Q$-convex functions $\psi_0,\psi_1: \Rsp^d\to\Rsp$. These functions are convex, Lipschitz, and their convex conjugates $\psi_0^*,\psi_1^*$ are Lipschitz on the ball $B$. Let $\chi:\R^d\rightarrow\R$ be a smooth and non-negative function supported in $B(0,1)$, with $\int_{\R^d}\chi\; d\lambda=1$. For $\varepsilon>0$ let $\chi_\varepsilon=\varepsilon^{-d}\chi(\cdot/\varepsilon)$. For $i=1,2$, setting
$$
\psi_{i,\varepsilon}=\chi_\varepsilon*\psi_i + \varepsilon |\cdot|^2
$$
we obtain a family $(\psi_{i,\eps})_{\eps\to 0}$ of $\Class^2$-smooth and strongly convex functions on $\Rsp^d$ such that $(\psi_{i,\eps})$ converges pointwise to $\psi_{i}$ on $\Rsp^d$ and $(\psi_{i,\eps}^*)$ converges pointwise to $\psi_{i}^*$ on $\mathrm{int}(B)$ as $\eps \to 0$. By \cite[Theorem 3.1.4]{fundamentals}, we get
\begin{enumerate}
\item $(\psi_{i,\eps})$ converges uniformly to $\psi_i$  on all compact subset of $\Rsp^d$ ;
\item $(\psi_{i,\eps}^*)$ converges uniformly to $\psi_{i}^*$ on all compact subsets of $\mathrm{int}(B)$ (in particular $Q$) ;
\end{enumerate}
The functions $\psi_{i,\eps}^*$ are equi-Lipschitz in every compact subset of $\mathrm{int}(B)$: this follows from the monotonicity of the slope of convex functions and the uniform boundedness of $(\psi_{i,\varepsilon}^*)$ in compact subsets of $\mathrm{int}(B)$ (see \cite[Lemma 3.1.1]{fundamentals} for a very similar argument).  Combining this fact with Proposition~\ref{p:GNdelmer} and with the inequality $W_1(\nabla \psi_{i,\eps\#}^*\rho, \nabla \psi_{i\#}^*\rho) \leq \nr{\nabla \psi_{i,\eps}^* - \nabla \psi_i^*}_{L^1(\rho)}$, we get
\begin{enumerate}
\item[(3)] $\lim_{\eps\to 0} W_1(\nabla \psi_{i,\eps\#}^*\rho, \nabla \psi_{i\#}^*\rho) = 0.$
\end{enumerate} 
Thanks to (1)-(3), all the terms of the inequality \eqref{eq:dm} for $\eps>0$ converge to their non-regularized counterparts as $\eps\to 0$. In other words, \eqref{eq:dm} for $Q$-convex functions is a consequence of the same inequality for $\Class^2$ and strongly convex functions on $\Rsp^d$. 
\bigskip

\noindent\emph{Step 3: Proof of the inequality in the smooth and strongly convex case.}
From now on, we therefore assume that $\psi_0,\psi_1$ are strongly convex and $\Class^2$. This allows us to apply \cite[Proposition 2.2]{delmer} to $\phi_i = \psi_i^*$, which gives formulas for the first and second derivative of  $\Kant_{\rho}$ along the curve $\psi_{t} = \psi_{0} + t v$ where $v = \psi_{1}-\psi_{0}$. We obtain: 
\begin{equation} \label{eq:formula-diff-K}
\begin{aligned}
\sca{\psi_1 - \psi_0}{\nabla\psi_{0\#}^*\rho - \nabla\psi_{1\#}^*\rho} &=
    \restr{\frac{d}{d t} \Kant_{\rho}(\psi_{t})}_{t=1} - 
    \restr{\frac{d}{d t} \Kant_{\rho}(\psi_{t})}_{t=0} \\
    &= \int_0^1 \frac{d^2}{dt^2} \Kant_\rho(\psi_{t}) d t   \\
    &= \int_0^1 \int_Q \sca{\nabla v(\nabla \psi_{t}^*)}{D^2 \psi_{t}^* \cdot \nabla v(\nabla \psi_{t}^*)} d \rho d t,
\end{aligned}
\end{equation}
Introducing $w_t(x) = v(\nabla \psi_t^*)$, whose gradient is given by
$\nabla w_t = D^2\psi_t^* \cdot \nabla v(\nabla \psi_t^*),$ we get
\begin{equation}\label{e:anepasoublier2}
    \int_Q \sca{\nabla v(\nabla \psi_t^*)}{D^2 \psi_t^* \cdot \nabla v(\nabla \psi_t^*)} d\rho
    = \int_Q \sca{(\D^2 \psi_t^*)^{-1} \nabla w_t}{\nabla w_t} d \rho. 
\end{equation}
Note here that $\psi_t$ is $C^2$, thus $\psi_t^*$ is strongly convex and its Hessian is invertible.
We now use the assumptions on $\sigma$ and $\rho$. We write $\sigma = e^{-V}$ where $V$ is a convex potential on $Q$. We first note that
\begin{equation}\label{e:anepasoublier}
\int_Q \sca{(\D^2 \psi_t^*)^{-1} \nabla w_t}{\nabla w_t} d \rho \geq  \int_Q \sca{(\D^2 \psi_t^* + \D^2 V)^{-1} \nabla w_t}{\nabla w_t} d \rho,
\end{equation} 
To apply the Brascamp-Lieb inequality, we  compare the probability density $\rho_t =Z_t^{-1} e^{-(V+\psi_t^*)}$ over $Q$, where $Z_t$ is a normalizing constant, to the probability density $\rho$. We denote by $m_t$
and $M_t$ the minimum and the maximum of $\psi_t^*(x)$ over $x\in Q$, and let $r=\sup_{t\in[0,1]} M_t-m_t$. We have
\begin{equation*}
\begin{aligned}
 &\rho(x) \geq m_\rho e^{-V} e^{m_t - \psi_t^*(x)} = m_\rho Z_t e^{m_t} \rho_t(x), \\
 &\rho_t(x) = \frac{1}{Z_t} e^{-M_t} e^{M_t -\psi_t^*-V} \geq \frac{1}{M_\rho Z_t}e^{-M_t} \rho(x).
 \end{aligned}
\end{equation*}
The Brascamp-Lieb inequality \cite{brascamplieb} applies to log-concave probability measures supported on the compact and convex set $Q$ (as a special case of \cite[Corollary 1.3]{lepeutrec}, see also \cite[Theorem 2.3]{delmer}). Thanks to the above comparison between $\rho$ and $\rho_t$ and this version of the Brascamp-Lieb inequality (in the second line) we obtain 
\begin{align} 
\int_Q \sca{(\D^2 \psi_t^* + \D^2 V)^{-1} \nabla w_t}{\nabla w_t} d\rho &\geq
m_\rho Z_t e^{m_t} \int_Q \sca{(\D^2 \psi_t^* + \D^2 V)^{-1} \nabla w_t}{\nabla w_t} d\rho_t \notag \\
&\geq m_\rho Z_t e^{m_t} \Var_{\rho_t}(w_t) \notag \\
&\geq \frac{ m_\rho}{M_\rho} e^{-r} \Var_{\rho}(w_t). \label{eq:d2-lower-bound-var} 
\end{align}
Using the convexity of the variance, and using that $\frac{d}{dt} \psi_t^*(x) = -v(\nabla \psi_t^*(x)) = - w_t(x)$ (a consequence of Fenchel-Young's equality, see proof of Proposition 2.2 in \cite{delmer}),
\begin{equation}\label{eq:var-dual-to-primal}
\begin{aligned}
    \Var_\rho(\psi_1^* - \psi_0^*)  
    = \Var_\rho\left(\int_0^1 \frac{d}{dt} \psi_t^* d t\right)
    &\leq \int_0^1\Var_{\rho}\left(\frac{d}{dt} \psi_t^* \right)d t
    &= \int_0^1 \Var_{\rho}(w_t) d t.
\end{aligned}
\end{equation}
Combining Equations~\eqref{eq:formula-diff-K}, \eqref{e:anepasoublier2}, \eqref{e:anepasoublier}, \eqref{eq:d2-lower-bound-var} and \eqref{eq:var-dual-to-primal}, we obtain 
\begin{equation}\label{eq:var:exp-constant}
    \sca{\psi_1 - \psi_0}{\nabla \psi_{0\#}^*\rho - \nabla \psi_{1\#}^* \rho} \geq \frac{m_\rho}{M_\rho} e^{-r} \Var_\rho(\psi_1^* - \psi_0^*)
\end{equation} 
\bigskip

\noindent\emph{Step 4. Computing an explicit constant.}
We first improve the constant in \eqref{eq:var:exp-constant} by a scaling argument. Given $\lambda>0$, we recall that $(\lambda\psi)^* = \lambda \psi^*(\cdot/\lambda)$. Applying the previous inequality to the functions $\psi_i^\lambda = \lambda\psi_i$ and to the dilated probability density $\rho_\lambda = (x\mapsto \lambda^{-1} x)_{\#}\rho$, and remarking that $(M^\lambda_t, m^\lambda_t) = (\lambda M_t, \lambda m_t)$, we get 
$$ \lambda\sca{\psi_1 - \psi_0}{\nabla \psi_{0\#}^*\rho - \nabla \psi_{1\#}^* \rho} \geq \frac{m_\rho}{M_\rho}  e^{-\lambda r} \lambda^2 \Var_\rho(\psi_1^* - \psi_0^*). $$
Choosing $\lambda = r^{-1}$, we see that we can replace the constant $e^{-r}$ in \eqref{eq:var:exp-constant} by $1/er$. 
Finally, we need to control the oscillation $r$. For this, we fix $x\in Q$ and consider $y \in \YSp$ such that $\psi_t^*(x)= \sca{x}{y} - \psi_t(y)$. Then, for any other point $x'\in Q$, we have
$$\psi_t^*(x') \geq \sca{x'}{y} - \psi_t(y) = \sca{x'-x}{y} + \psi_t^*(x) \geq -\diam(Q)R_\YSp + \psi_t^*(x).$$
This implies that $r\leq\diam(Q)R_\YSp$, which concludes the proof of \eqref{eq:dm}.
\end{proof}

\subsection{Proof of Theorem \ref{t:logconcave}}\label{s:vrailogconc}
Let us assume that $\rho = e^{-U-F}$ where $D^2U\geq \kappa {\rm Id}$ with $\kappa>0$, and $F\in L^\infty(\R^d)$. Without loss of generality we assume that $U$ attains its minimum at the origin $0\in\R^d$. For $r\geq 0$, we consider $\Ball_r=B(0,r) \subset \R^d$, the closed Euclidean ball of center $0$ and radius $r$. We denote by $\phi_{\mu,r}$ (resp. $\phi_{\nu,r}$) the restriction of $\phi_\mu$ (resp. $\phi_\nu$) to $\Ball_r$, extended by $+\infty$ outside $\Ball_r$, and we consider the probability measures 
$$
\rho_r=\frac{\rho_{|\Ball_r}}{\rho(\Ball_r)}, \qquad \mu_r=(\nabla \phi_{\mu,r})_{\#}\rho_r, \qquad \nu_r=(\nabla \phi_{\nu,r})_{\#}\rho_r.
$$
Finally we set $\psi_{\mu,r}=\phi_{\mu,r}^*$ and $\psi_{\nu,r}=\phi_{\nu,r}^*$. We observe that $\psi_{\mu,r}^*=\phi_{\mu,r}$ and $\psi_{\nu,r}^*=\phi_{\nu,r}$ by the Fenchel-Moreau theorem.

We apply Theorem \ref{t:stability-compact} to $\rho_r$, with $\sigma$ the (unique) probability density on $\Ball_r$ whose density is proportional to $e^{-U}$. This gives
\begin{equation}\label{e:varineqrhosgau}
\Var_{\rho_r}(\phi_{\mu,r}-\phi_{\nu,r})\leq C_{\rho,\mathcal{Y}}r\sca{\psi_{\mu,r}-\psi_{\nu,r}}{\nu_r-\mu_r}
\end{equation}
where $C_{\rho,\mathcal{Y}}$ does not depend on $r$. 
We notice that 
$\psi_{\mu,r}$ and $\psi_{\nu,r}$ are $r$-Lipschitz. By Kantorovich-Rubinstein duality, we deduce from \eqref{e:varineqrhosgau} the upper bound
\begin{equation}\label{e:varineqsgau}
\Var_{\rho_r}(\phi_{\mu,r}-\phi_{\nu,r})\leq C_{\rho,\mathcal{Y}}r^2W_1(\mu_r,\nu_r).
\end{equation}

\begin{claim} There exists $r_0\geq 1$ (depending only on $\rho$, not on $\mu,\nu$) such that 
\begin{equation}\label{e:infpluspetitque1gau}
\inf_{x\in \Res_{r_0}}|(\phi_\mu-\phi_\nu)(x)|\leq 1.
\end{equation}
\end{claim}
\begin{proof}[Proof of the claim]
Set $f=\phi_\mu-\phi_\nu$. Let $r_0$ be large enough so that
$$\rho(\Res_{r_0})>2R_{\mathcal{Y}}\int_{\R^d\setminus \Res_{r_0}}{\rm dist}(x,\Res_{r_0})d\rho(x).
$$ Notice that such $r_0$ exists since the left-hand side tends to $1$ as $r_0\rightarrow +\infty$ and the right-hand side tends to $0$ by dominated convergence (indeed, all moments of $\rho$ are finite, since $\rho$ is a log-bounded perturbation of a log-concave density). Assume for the sake of a contradiction that \eqref{e:infpluspetitque1gau} does not hold, and without loss of generality that $f>1$ on $\Res_{r_0}$. Together with the fact that 
\begin{equation}\label{e:bddgradphigau} 
\phi_\mu \text{ and } \phi_\nu \text{ are } R_{\mathcal{Y}}\text{-Lipschitz},
\end{equation}  
this implies $f(x)\geq -2R_{\mathcal{Y}}{\rm dist}(x,\Res_{r_0})$ for $x\notin \Res_{r_0}$, therefore 
$$
\int_{\R^d}fd\rho=\int_{\Res_{r_0}}fd\rho+\int_{\R^d\setminus \Res_{r_0}}fd\rho> \rho(\Res_{r_0})-2R_{\mathcal{Y}}\int_{\R^d\setminus \Res_{r_0}}{\rm dist}(x,\Res_{r_0})d\rho(x)>0.
$$ 
But recall that 
\begin{equation}\label{e:nullaveragegau}
\int_{\R^d} \phi_\mu d\rho = \int_{\R^d} \phi_\nu d\rho =0 
\end{equation}  
therefore $\int_{\R^d}fd\rho=0$ and we get a contradiction. 
\end{proof}
For $r\geq 0$ and $\ell=0,1,2$ we set
\begin{equation}\label{e:truncmoments}
m_\ell(r)=\int_{\R^d\setminus \Res_r}|x|^\ell d\rho(x)
\end{equation}
the $\ell$-th moment of the tail of $\rho$ outside $\Res_r$.
Using \eqref{e:bddgradphigau} and the claim, we deduce that there exists $C_\rho>0$ such that for any $x\in\R^d$, $|(\phi_\mu-\phi_\nu)(x)|\leq C_\rho+R_{\mathcal{Y}}|x|$.  This implies that for large values of $r$, 
\begin{equation}\label{e:complementballgau}
\|\phi_\mu-\phi_\nu\|^2_{L^2(\rho,\R^d\setminus \Res_r)}\leq \int_{\R^d \setminus \Res_r} (C_\rho+R_{\mathcal{Y}}|x|)^2 d\rho(x)\leq C_{\rho,\mathcal{Y}}m_2(r).
\end{equation}
Now we estimate $\eta_r=\int_{\R^d} (\phi_{\mu,r}-\phi_{\nu,r})d\rho_r$: using \eqref{e:nullaveragegau}, we get
$$
\rho(\Res_r)\left|\eta_r\right|=\Bigl|\int_{\R^d \setminus \Res_r} (\phi_\mu-\phi_\nu)d\rho\Bigr|\leq \int_{\R^d \setminus \Res_r} (C_\rho+R_{\mathcal{Y}}|x|)d\rho(x)\leq C_{\rho,\mathcal{Y}}m_1(r)
$$
again for large $r$.
For $r$ large enough, $\rho(\Res_r)\geq 1/2$, we deduce
\begin{align}
\|\phi_{\mu,r}-\phi_{\nu,r}\|_{L^2(\rho)}^2\leq \|\phi_{\mu,r}-\phi_{\nu,r}\|_{L^2(\rho_r)}^2&={\rm Var}_{\rho_r}(\phi_{\mu,r}-\phi_{\nu,r})+\eta_r^2\nonumber\\
&\leq C_{\rho,\mathcal{Y}}(r^2W_1(\mu_r,\nu_r)+m_1(r)^2)\label{e:estballgau}
\end{align}
where in the last inequality we used \eqref{e:varineqsgau}.
Together with \eqref{e:complementballgau} we get
\begin{align}
{\rm Var}_\rho(\phi_\mu-\phi_\nu)=\|\phi_\mu-\phi_\nu\|^2_{L^2(\rho)}&=\|\phi_{\mu,r}-\phi_{\nu,r}\|_{L^2(\rho,\Res_r)}^2+\|\phi_\mu-\phi_\nu\|^2_{L^2(\rho,\R^d\setminus \Res_r)}\nonumber\\
&\leq C_{\rho,\mathcal{Y}}(r^2W_1(\mu_r,\nu_r)+m_1(r)^2+m_2(r))\label{e:varwasssgau}
\end{align}

\begin{claim} There holds 
\begin{equation}\label{e:W1W1gau}
W_1(\mu_r,\nu_r)\leq W_1(\mu,\nu)+C_{\rho,\mathcal{Y}}m_0(r).
\end{equation}
\end{claim}
\begin{proof}[Proof of the claim]
There holds 
$$
W_1(\mu_r,\nu_r)-W_1(\mu,\nu)\leq W_1(\mu_r,\mu)+W_1(\nu_r,\nu).
$$
Let us prove an upper bound on $W_1(\mu_r,\mu)$ for instance. We write $\mu=\rho(\Res_r)\mu_r+\mu_r'$. Thanks to the Kantorovich-Rubinstein duality formula we know there exists a $1$-Lipschitz function $g$ such that $W_1(\mu_r,\mu)=\int_{B(0,R_{\mathcal{Y}})} g\  d(\mu_r-\mu)$. We may assume that $g(0)=0$. Hence $|g(x)|\leq R_{\mathcal{Y}}$ for any $x\in \mathcal{Y}$. Since $\mu_r'(\mathcal{Y})=1-\rho(\Res_r)$ we deduce
$$
W_1(\mu_r,\mu)= (1-\rho(\Res_r))\int_{\mathcal{Y}} g d\mu_r-\int_{\mathcal{Y}}gd\mu_r'\leq 2R_{\mathcal{Y}}(1-\rho(\Res_r))=2R_{\mathcal{Y}}m_0(r).
$$
The claim follows.
\end{proof}

Plugging the claim into \eqref{e:varwasssgau} we get
\begin{equation}\label{e:phimuphinugau}
\|\phi_\mu-\phi_\nu\|_{L^2(\rho)}^2\leq C_{\rho,\mathcal{Y}}\left(r^2W_1(\mu,\nu)+r^2m_0(r)+m_1(r)^2+m_2(r)\right).
\end{equation}
Then due to the assumption that $\rho=e^{-U-F}$ where $U$ is minimal at $0$, $D^2U\geq \kappa{\rm Id}$ with $\kappa>0$ and $F\in L^\infty(\R^d)$, 
\begin{equation}\label{e:upperboundmoment2gau}
m_\ell(r)\leq C_{\rho,\ell} r^{d+\ell-2}e^{-\frac12 \kappa r^2}
\end{equation}
for some constant $C_{\rho,\ell}$ independent of $r$. Now, observe that $\|\phi_\mu-\phi_\nu\|_{L^2(\rho)}$ is also upper bounded by a constant which depends only on $\rho$, as shown by the same argument as in \eqref{e:complementballgau} (with integration over $\R^d$ instead of $\R^d\setminus\mathcal{B}_r$). Hence, to establish \eqref{e:stabpotlogconc}, we may fix $\varepsilon>0$ and assume in the sequel that $W_1(\mu,\nu)\leq \varepsilon$.
We plug \eqref{e:upperboundmoment2gau} into \eqref{e:phimuphinugau} and we let $r=(4\kappa^{-1}|\log W_1(\mu,\nu)|)^{1/2}$. Note that choosing $\varepsilon>0$ small enough guarantees that $\rho(\mathcal{B}_r)\geq 1/2$. This choice of $r$ yields \eqref{e:stabpotlogconc}.

As in \cite{delmer}, the stability of Brenier maps \eqref{e:stabmaplogcon} now follows from Proposition \ref{p:GNdelmer}. We apply it with $K=\Res_r$ (where $r\geq 0$ is arbitrary), to the functions $\phi_\mu,\phi_\nu$ which are $R_{\mathcal{Y}}$-Lipschitz. We get
\begin{equation}\label{e:balchgau}
\|T_\mu-T_\nu\|^2_{L^2(\lambda,\Ball_r)}\leq C_\rho  r^{2(d-1)/3}R_{\mathcal{Y}}^{4/3}\|\phi_\mu-\phi_\nu\|^{2/3}_{L^2(\lambda,\Ball_r)}.
\end{equation} 
This is written for the Lebesgue measure $\lambda$, but since $me^{-\frac12 \kappa' |x|^2}\leq \rho(x)\leq M$ we deduce immediately that
$$
\|T_\mu-T_\nu\|^2_{L^2(\rho,\Ball_r)}\leq C_{\rho,\mathcal{Y}}  r^{2(d-1)/3}e^{\frac16 \kappa' r^2}\|\phi_\mu-\phi_\nu\|^{2/3}_{L^2(\rho)}.
$$
Due to \eqref{e:bddgradphigau} and \eqref{e:upperboundmoment2gau} (for $\ell=0$) we get
\begin{equation}\label{e:onyestpresque}
\|T_\mu-T_\nu\|^2_{L^2(\rho)}\leq  C_{\rho,\mathcal{Y}}  r^{2(d-1)/3}e^{\frac16 \kappa' r^2}\|\phi_\mu-\phi_\nu\|^{2/3}_{L^2(\rho)}+C_{\rho,\mathcal{Y}} r^{d-2}e^{-\frac12 \kappa r^2}.
\end{equation}
Because of the uniform upper bound $\|T_\mu-T_\nu\|_{L^2(\rho)}\leq 2R_{\mathcal{Y}}$, we may assume that $W_1(\mu,\nu)$ is small to establish \eqref{e:stabmaplogcon}; here we assume that it is small enough so that $\|\phi_\mu-\phi_\nu\|_{L^2(\rho)}\leq 1/2$ (recall \eqref{e:stabpotlogconc}). We optimize \eqref{e:onyestpresque} over $r$, namely we choose $r$ in a way that $e^{\frac16 \kappa' r^2}\|\phi_\mu-\phi_\nu\|^{2/3}_{L^2(\rho)}=e^{-\frac12 \kappa r^2}$, namely $r=\left(\frac{4}{\kappa'+3\kappa}|\log \|\phi_\mu-\phi_\nu\|_{L^2(\rho)}|\right)^{1/2}$, we obtain
$$
\|T_\mu-T_\nu\|_{L^2(\rho)}\leq C_{\rho,\mathcal{Y}}\|\phi_\mu-\phi_\nu\|^{\frac{\kappa}{\kappa'+3\kappa}}_{L^2(\rho)}|\log \|\phi_\mu-\phi_\nu\|_{L^2(\rho)}|^{\max(\frac{d-2}{2},\frac{d-1}{3})}
$$
Combining with \eqref{e:stabpotlogconc} and using that $\frac{\kappa}{2\kappa'+7\kappa}<\frac{\kappa}{2\kappa'+6\kappa}$ to compensate the log-factor, we get \eqref{e:stabmaplogcon}.

\section{Variance inequalities}\label{s:variancegen}

Our proofs of Theorems \ref{t:mainjohn}, \ref{t:explosebord} and \ref{t:hallin} mainly rely on a variance inequality for the dual potentials (Theorem \ref{t:varineqjohn}), which extends Theorem \ref{t:stability-compact} to all probability densities $\rho$ having the following three properties: (i) the support of $\rho$ can be decomposed into a family $\mathcal{F}$ of cubes satisfying the so-called Boman chain condition (see Definition \ref{d:boman}); (ii) a doubling property \eqref{e:ctdoubling}; (iii) a uniform upper bound \eqref{e:supratio} on the maximum and minimum value of $\rho$ in each cube of this decomposition. 

In Section \ref{s:whitneyboman}, we provide reminders on the Whitney decomposition and the Boman chain condition. Then in Section \ref{s:glv} we prove a lemma showing how to glue variance inequalities together when items (i) and (ii) described above hold. The proof of this lemma relies on techniques developed to prove Poincaré-Sobolev inequalities in John domains (see for instance \cite{bojarski}). Finally in Section \ref{s:johnpotentials} we deduce the variance inequality (Theorem \ref{t:varineqjohn}).

\subsection{Whitney decomposition and the Boman chain condition}\label{s:whitneyboman}

Let us first recall briefly the Whitney decomposition (see \cite[Appendix J.1]{grafakos} for a proof). A dyadic cube in $\R^d$ is a cube of the form 
$$
\left\{(x_1,\ldots,x_d)\in\R^d \mid m_j2^{-\ell}\leq x_j  < (m_j+1)2^{-\ell} \ \text{for any } j\in[d]\right\}
$$
where $\ell\in\Z$ and $m_j\in\Z$ for any $j\in[d]$.
Roughly speaking, the Whitney decomposition is a decomposition of any open set into dyadic cubes in a way that each of these cubes has a sidelength which is comparable to its distance from the boundary of the open set.
\begin{lemma}[Whitney decomposition]\label{l:whitney}
Let $\X$ be an open non-empty proper subset of $\R^d$. Then there exists a family of closed dyadic cubes $P_j$ such that
\begin{itemize}
\item  The $P_j$'s have disjoint interiors and 
$$
\bigcup _{j}P_{j}=\X.
$$
\item If $\ell (P)$ denotes the sidelength of a cube $P$ and $\X^c=\R^d\setminus\X$, 
\begin{equation}\label{e:distwhitney}
{\sqrt {d}}\ell (P_{j})\leq \dist(P_{j},\X ^{c})\leq 4{\sqrt {d}}\ell (P_{j}).
\end{equation}
\item If the boundaries of two cubes $P_j$ and $P_k$ touch then 
\begin{equation}\label{e:whitneylength}
{\frac {1}{4}}\leq {\frac {\ell (P_{j})}{\ell (P_{k})}}\leq 4.
\end{equation}
\item For a given $P_j$ there exist at most $12^{d}$ cubes $P_{k}$'s that touch it.
\item Let $1<\sigma<5/4$. If $Q_j$ denotes the cube with same center as $P_j$ and $\ell(Q_j)=\sigma\ell(P_j)$, then
$$
\sum_j \chi_{Q_j}\leq 12^d
$$
where $\chi_{Q_j}$ is the characteristic function of $Q_j$.
\end{itemize}
\end{lemma}

Under some assumptions on $\X$ and $\rho$ that are studied in Section \ref{s:proofjohn}, the cubes $Q_j$ of the last item of the above lemma (for some fixed $\sigma\in (1,5/4)$) provide another decomposition of $\X$ into cubes which now overlap, and which satisfy the properties listed below (\cite[Lemma 2.1]{boman}, \cite[Section 3]{bojarski}, \cite[Section 4]{iwaniec}):
\begin{definition}[Boman chain condition]\label{d:boman}
Let $\sigma\geq 1$ and $\A,\B,\C>1$. A probability measure $\rho$ on an open set $\X\subset\R^d$ satisfies the Boman chain condition with parameters $\A,\B,\C$ if there exists a covering $\mathcal{F}$ of $\X$ by open cubes $Q\in\mathcal{F}$ such that
\begin{itemize}
\item for $\sigma=\min(10/9,(d+1)/d)$ and for any $x\in\R^d$,
\begin{equation}\label{e:A}
\sum_{Q\in\mathcal{F}}\chi_{\sigma Q}(x)\leq \A\chi_\X(x).
\end{equation}
\item For some fixed cube $Q_0$ in $\mathcal{F}$, called the central cube, and for every $Q\in\mathcal{F}$, there exists a chain $Q_0,Q_1,\ldots,Q_N=Q$ of distinct cubes from $\mathcal{F}$ such that for any $j\in \{0,\ldots,N-1\}$,
\begin{equation}\label{e:BJohn}
Q\subset \B Q_j
\end{equation}
where $BQ_j$ is the cube with same center as $Q_j$, and sidelength multiplied by $B$.
\item Consecutive cubes of the above chain overlap quantitatively: for any $j\in \{0,\ldots,N-1\}$,
\begin{equation}\label{e:conseccubeoverlap}
\rho(Q_j\cap Q_{j+1})\geq \C^{-1}\max(\rho(Q_j),\rho(Q_{j+1})).
\end{equation}
\end{itemize}
\end{definition}
Note that in the above definition, the length $N$ of the connecting chain depends on $Q\in\mathcal{F}$ and is not assumed to be uniformly bounded in $Q\in\mathcal{F}$. 
Also we could have replaced $\sigma Q$ in \eqref{e:A} simply by $Q$ for the purpose of this paper, but we keep the same convention as in \cite{boman}, \cite{bojarski}, \cite{iwaniec} to remain coherent with the literature. Finally, let us observe that \eqref{e:conseccubeoverlap} implies that consecutive cubes of the above chain are comparable in size: for any $j\in \{0,\ldots,N-1\}$,
\begin{equation}\label{e:conseccubessize}
\C^{-1}\leq \frac{\rho(Q_j)}{\rho(Q_{j+1})}\leq \C.
\end{equation}
Indeed, if \eqref{e:conseccubeoverlap} holds, then 
$$
\rho(Q_j)\geq \rho(Q_j \cap Q_{j+1}) \geq C^{-1}\max(\rho(Q_j), \rho(Q_{j+1})) \geq C^{-1}\rho (Q_{j+1}),
$$
and the other bound in \eqref{e:conseccubessize} is proved in the same way by reversing $j$ and $j+1$.

\subsection{Gluing variance inequalities}\label{s:glv}

Let $\rho$ be a probability density over a domain $\X$. We assume that $\rho$ satisfies the Boman chain condition for some $\A,\B,\C>1$. Let $\mathcal{F}$ be a covering as in Definition \ref{d:boman}. For any cube $Q\in\mathcal{F}$ we consider
\begin{equation*}\label{e:BrenierpotQ}
\tilde{\rho}_Q=\frac{\rho_{|Q}}{\rho(Q)}.
\end{equation*}
A key step in the proof of Theorem \ref{t:varineqjohn} is the following lemma. 
\begin{lemma}\label{l:gluevarjohn}
Let $\rho$ be a probability density over a domain $\X\subset\R^d$ satisfying the Boman chain condition for some covering $\mathcal{F}$ and some $\A,\B,\C>1$. Assume moreover that there exists $\D>0$ such that
\begin{equation}\label{e:ctdoubling}
\forall Q\in\mathcal{F}, \qquad \rho(\ct Q)\leq \D \rho(Q).
\end{equation}
Then for any continuous function $f$ on $\X$, there holds 
$$
{\rm Var}_\rho(f)\leq 200\A^2\C \D^3 \sum_{Q\in\mathcal{F}}\rho(Q)\Var_{\tilde{\rho}_Q}(f).
$$
\end{lemma}

\begin{proof}[Proof of Lemma \ref{l:gluevarjohn}]
We denote by $f_Q=\int_Q f d\tilde{\rho}_Q$ the mean of $f$ over $Q$.
Let $Q_0$ be a central cube as in Definition \ref{d:boman}. By definition of the variance,
\begin{equation}\label{e:varQ0}
{\rm Var}_\rho(f)\leq \int_\X|f(x)-f_{Q_0}|^2d\rho(x).
\end{equation}
We estimate the right-hand side of \eqref{e:varQ0}. For $Q\in\mathcal{F}$, we set $a_{Q}=(\Var_{\tilde{\rho}_Q}(f))^{1/2}$. Since $\mathcal{F}$ is a covering of $\X$,
\begin{align}
\int_\X |f-f_{Q_0}|^2d\rho(x)&\leq \sum_{Q\in \mathcal{F}}\int_Q |f(x)-f_{Q_0}|^2d\rho(x)\nonumber \\
&\leq 2 \sum_{Q\in\mathcal{F}}\left(\int_Q |f(x)-f_Q|^2d\rho(x)+\int_Q |f_Q-f_{Q_0}|^2d\rho(x)\right)\nonumber\\
&= 2\sum_{Q\in\mathcal{F}} \rho(Q)a_Q^2 + 2\sum_{Q\in\mathcal{F}}\int_Q |f_Q-f_{Q_0}|^2d\rho(x)\label{e:2terms}
\end{align}
Fix $Q\in\mathcal{F}$ and take a chain $Q_0,Q_1,\ldots,Q_N=Q$ of distinct cubes from $\mathcal{F}$ satisfying the assumptions of Definition \ref{d:boman}. By the triangle inequality, we have
\begin{equation}\label{e:boj1}
\left(\int_Q |f_Q-f_{Q_0}|^2d\rho(x)\right)^{\frac12}\leq \sum_{j=0}^{N-1}\left(\int_Q |f_{Q_j}-f_{Q_{j+1}}|^2d\rho(x)\right)^{\frac12}
\end{equation}
For fixed $j$, using the last condition in Definition \ref{d:boman} and viewing $|f_{Q_j}-f_{Q_{j+1}}|$ as a constant function, we have
\begin{align*}
&\int_Q |f_{Q_j}-f_{Q_{j+1}}|^2d\rho(x)\\
&\qquad\qquad=\frac{\rho(Q)}{\rho(Q_j \cap Q_{j+1})}\int_{Q_j\cap Q_{j+1}} |f_{Q_j}-f_{Q_{j+1}}|^2d\rho(x)\\
&\qquad\qquad\leq 2\frac{\rho(Q)}{\rho(Q_j \cap Q_{j+1})}\left(\int_{Q_j\cap Q_{j+1}} |f_{Q_j}-f(x)|^2d\rho(x)+\int_{Q_j\cap Q_{j+1}} |f_{Q_{j+1}}-f(x)|^2d\rho(x)\right)\\
&\qquad\qquad\leq 2\C \frac{\rho(Q)}{\rho(Q_j)}\int_{Q_j}|f_{Q_j}-f(x)|^2d\rho(x)+2\C \frac{\rho(Q)}{\rho(Q_{j+1})}\int_{Q_{j+1}}|f_{Q_{j+1}}-f(x)|^2d\rho(x)\\
&\qquad\qquad\leq 2\C \rho(Q)(a_{Q_j}^2+a_{Q_{j+1}}^2).
\end{align*}
Taking the square root, we obtain
$$
\left(\int_Q |f_{Q_j}-f_{Q_{j+1}}|^2d\rho(x)\right)^{\frac12}\leq (2\C)^{\frac12} \rho(Q)^{\frac12}(a_{Q_j}+a_{Q_{j+1}}).
$$
Plugging into \eqref{e:boj1}, and using that $Q\subset BQ_j$ and $Q\subset BQ_{j+1}$, this yields
$$
\left(\int_Q |f_Q-f_{Q_0}|^2d\rho(x)\right)^{\frac12}\leq (8\C)^{\frac12}\rho(Q)^{\frac12}\sum_{Q\subset B\tilde{Q}} a_{\tilde{Q}}
$$
where the symbol $\displaystyle\sum_{Q\subset B\tilde{Q}}$ means that the sum runs over all cubes $\tilde{Q}\in\mathcal{F}$ such that $Q\subset B\tilde{Q}$ (recall that $Q$ is fixed). Taking the square, we get
\begin{align}
\int_Q |f_Q-f_{Q_0}|^2d\rho(x)\leq 8 \C\rho(Q)\Bigl(\sum_{Q\subset B\tilde{Q}} a_{\tilde{Q}}\Bigr)^2&=8\C  \int_Q \Bigl(\sum_{Q\subset B\tilde{Q}} a_{\tilde{Q}}\Bigr)^2 d\rho(x)\nonumber\\
&\leq 8\C \int_Q\Bigl(\sum_{\tilde{Q}\in \mathcal{F}}a_{\tilde{Q}}\chi_{B\tilde{Q}}(x)\Bigr)^2d\rho(x)\label{e:bqtilde}
\end{align}
where in the last line we used that for any $x\in Q$, 
$$
\sum_{Q\subset B\tilde{Q}} a_{\tilde{Q}}\leq \sum_{\tilde{Q}\in \mathcal{F}}a_{\tilde{Q}}\chi_{B\tilde{Q}}(x).
$$
The following lemma, which is a generalization of \cite[Lemma 4.2]{bojarski}, is proved in Appendix \ref{a:maximal}:
\begin{lemma} \label{l:maximal}
Under the assumption \eqref{e:ctdoubling} there holds
\begin{equation}\label{e:maximalforrho}
\Bigl\|\sum_{\tilde{Q}\in \mathcal{F}}a_{\tilde{Q}}\chi_{B\tilde{Q}}\Bigr\|_{L^2(\rho)} \leq (2\D)^{3/2}\Bigl\|\sum_{\tilde{Q}\in \mathcal{F}}a_{\tilde{Q}}\chi_{\tilde{Q}}\Bigr\|_{L^2(\rho)}.
\end{equation}
\end{lemma}

Summing \eqref{e:bqtilde} over $Q\in\mathcal{F}$, we obtain 
\begin{align}
\sum_{Q\in \mathcal{F}}\int_Q|f_Q - f_{Q_0}|^2\rho(x)&\leq 8\A \C\int_\X\Bigl(\sum_{\tilde{Q}\in \mathcal{F}}a_{\tilde{Q}}\chi_{B\tilde{Q}}(x)\Bigr)^2d\rho(x)\nonumber\\
&\leq 64\A\C\D^3\int_\X\Bigl(\sum_{\tilde{Q}\in \mathcal{F}}a_{\tilde{Q}}\chi_{\tilde{Q}}(x)\Bigr)^2d\rho(x)\nonumber\\
&\leq 64\A^2\C\D^3\int_\X \sum_{\tilde{Q}\in\mathcal{F}} a_{\tilde{Q}}^2\chi_{\tilde{Q}}(x)d\rho(x)\nonumber\\
&= 64\A^2\C\D^3\sum_{\tilde{Q}\in\mathcal{F}}\rho(\tilde{Q})a_{\tilde{Q}}^2\label{e:II}
\end{align}
where for the first inequality we use the first point in Definition \ref{d:boman}, and for the third inequality we use Cauchy-Schwarz together with the first point in Definition \ref{d:boman}. Finally, putting together \eqref{e:varQ0}, \eqref{e:2terms} and \eqref{e:II}, we get the lemma. 
\end{proof}

\subsection{Remarks on uniqueness of Brenier potentials and Poincaré-Wirtinger inequalities}
\label{s:uniqueness}
Before stating and proving the main result of Section \ref{s:variancegen}, namely Theorem \ref{t:varineqjohn}, let us comment on uniqueness of Brenier potentials and link Lemma \ref{l:gluevarjohn} with Poincaré-Wirtinger inequalities. In the settings considered in the present paper, Brenier potentials are always unique: this is a consequence of the connectedness of the support of $\rho$, see for instance \cite[Proposition 7.18]{santambrogio} or \cite[Remark 10.30]{villani2}. However, if the support of $\rho$ is not connected, then in general the Brenier potentials are not unique:
\begin{proposition}
Let $\X_1$ and $\X_2$ be two non-empty bounded subsets of $\R^d$ with $\dist(\X_1,\X_2)>0$. Let $\rho$ be a probability measure on $\X=\X_1\cup \X_2$ such that $\min(\rho(\X_1),\rho(\X_2))>0$. Then there exists a $1$-parameter family of Brenier potentials corresponding to the identity mapping on $\X$ (which is the optimal transport map between $\rho$ and $\rho$), and any two of them are not $\rho$-almost everywhere equal.
\end{proposition}
\begin{proof}
Consider $\phi_\alpha(x)=\frac12|x|^2+\alpha \chi_{\X_1}(x)+c_\alpha$, where $\chi_{\X_1}$ is the characteristic function of $\X_1$ and $c_\alpha$ is chosen so that $\int_{\X}\phi_\alpha(x)d\rho(x)=0$. Then $\nabla\phi_\alpha$ transports $\rho$ to itself for any $\alpha$, and it is easy to check that for $0\leq \alpha<\dist(\X_1,\X_2)^2/2$, there holds $(\phi_\alpha^*)^*=\phi_\alpha$ on $\X$. Thus $\nabla\phi_\alpha$ is actually the optimal transport map.   
\end{proof}

Let us point out that in the setting of Lemma \ref{l:gluevarjohn}, not only the support of $\rho$ is connected, but $\rho$ even satisfies an $L^2$-Poincaré-Wirtinger inequality. Therefore, uniqueness of Brenier potentials may be checked also in the following way: the equality $\nabla\phi=\nabla\phi'$ $\rho$-almost everywhere together with $\int_{\X}\phi d\rho=\int_{\X}\phi'd\rho=0$ implies that $\phi=\phi'$.
Let us show that if $\rho$ satisfies the assumptions of Lemma~\ref{l:gluevarjohn} and also satisfies $\sup_{Q} \frac{M_{\tilde{\rho}_Q}}{m_{\tilde{\rho}_Q}}<+\infty$, then it supports a $L^2$-Poincaré-Wirtinger inequality:
\begin{equation}\label{e:poincare}
{\rm Var}_\rho(f)\leq C_P\int_{\X}|\nabla f|^2 d\rho.
\end{equation}
To prove \eqref{e:poincare}, we first observe that there exists $C_P'>0$ such that for any $Q\in\mathcal{F}$ there holds
\begin{equation}\label{e:poincareQ}
{\rm Var}_{\tilde{\rho}_Q}(f)\leq C_P'\int_{Q}|\nabla f|^2d\tilde{\rho}_Q.
\end{equation}
Indeed, since $\sup_{Q} \frac{M_{\tilde{\rho}_Q}}{m_{\tilde{\rho}_Q}}<+\infty$, it is sufficient to prove that \eqref{e:poincareQ} holds when $\tilde{\rho}_Q$ is replaced by the normalized Lebesgue measure on $Q$, with a $Q$-independent constant $C_P'$. This latter fact holds because all $Q$ are cubes, with uniformly bounded diameter. Summing \eqref{e:poincareQ} over $Q\in\mathcal{F}$ with weights $\rho(Q)$, and using \eqref{e:A} we get \eqref{e:poincare}.

\subsection{The variance inequality}\label{s:johnpotentials}
In the rest of the paper, for any probability density $\rho_0$ supported on $\X_0\subset \R^d$ we set 
\begin{equation}\label{e:supetinf}
M_{\rho_0}=\sup_{x\in \X_0} \rho_0(x), \qquad m_{\rho_0}=\inf_{x\in \X_0}\rho_0(x).
\end{equation}
Given a compact set $\mathcal{Y}\subset\R^d$, we set
\begin{equation}\label{e:RY}
R_{\mathcal{Y}}=\max_{y\in \mathcal{Y}} |y|.
\end{equation}
Recall from \eqref{eq:conv-conjugate} that the Legendre-Fenchel transform $\psi^*$ of a function $\psi\in\mathcal{C}^0(\mathcal{Y})$ is defined as $\psi^*(x)=\max_{y\in\mathcal{Y}}\ \sca{x}{y}-\psi(y)$. In this section we prove the following variance inequality:
\begin{theorem}\label{t:varineqjohn}
Let $\rho$ be a probability density over a domain $\X\subset\R^d$ satisfying the Boman chain condition for some covering $\mathcal{F}$ and some $\A,\B,\C>1$. Assume moreover that \eqref{e:ctdoubling} holds for some $\D>0$ and that there exists $\E>0$ such that
\begin{equation}\label{e:supratio}
\sup_{Q\in\mathcal{F}}\frac{M_{\rho_Q}}{m_{\rho_Q}}\leq \E<+\infty.
\end{equation}
Let $\mathcal{Y}\subset\R^d$ be a compact set. Set $C'= 200e\A^3\C\D^3\E R_{\mathcal{Y}}{\rm diam}(\X)>0$. Then, for any $\psi_0, \psi_1\in \mathcal{C}^0(\mathcal{Y})$, there holds
\begin{equation} \label{e:varineqgeneral}
{\rm Var}_\rho\left(\psi_1^*-\psi_0^*\right) \leq C' \sca{\psi_0 - \psi_1}{(\nabla \psi_1^*)_{\#}\rho-(\nabla \psi_0^*)_{\#}\rho}.
\end{equation}
\end{theorem}

\begin{proof}[Proof of Theorem \ref{t:varineqjohn}]
We apply Theorem \ref{t:stability-compact} to $\tilde{\rho}_Q$ on the compact convex set $Q$, with $\sigma$ the normalized Lebesgue measure on $Q$. We obtain
\begin{equation}\label{e:varconvex}
{\rm Var}_{\tilde{\rho}_Q}\left(\psi_1^*-\psi_0^*\right) \leq \CDM \sca{\psi_0 - \psi_1}{(\nabla \psi_1^*)_{\#}\tilde{\rho}_Q-(\nabla \psi_0^*)_{\#}\tilde{\rho}_Q}
\end{equation}
where $\CDM=e \E  R_{\mathcal{Y}} \diam(\X)$ is a constant which does not depend on $Q$.
Putting together Lemma \ref{l:gluevarjohn} with \eqref{e:varconvex} we get
\begin{align}
{\rm Var}_\rho(\psi_1^*-\psi_0^*)&\leq 200\A^2\C\D^3 \sum_{Q\in\mathcal{F}}\rho(Q)\Var_{\tilde{\rho}_Q}(\psi_1^*-\psi_0^*)\nonumber\\
&\leq 200\A^2\C\D^3 \CDM \sum_{Q\in\mathcal{F}}\sca{\psi_0-\psi_1}{(\nabla \psi_1^*)_\#\rho_{|Q}- (\nabla \psi_0^*)_\#\rho_{|Q}}\label{e:fatal}
\end{align}
Each point $x\in\X$ belongs to at most $\A$ cubes $Q\in\mathcal{F}$ due to the first point of Definition \ref{d:boman}. We define a partition $\mathcal{F}'$ of $\X$ as follows: $x,x'$ belong to the same element $P\in\mathcal{F}'$ if and only if they belong to exactly the same elements of $\mathcal{F}$. Hence it follows from \eqref{e:fatal} and the lower bound $\sca{\psi_0-\psi_1}{(\nabla \psi_1^*)_\#\rho_{|P}- (\nabla \psi_0^*)_\#\rho_{|P}}\geq 0$ valid for any $P\in\mathcal{F}'$ that
\begin{align}
{\rm Var}_\rho(\psi_1^*-\psi_0^*)&\leq 200\A^3\C\D^3 \CDM\sum_{P\in\mathcal{F}'}\sca{\psi_0-\psi_1}{(\nabla \psi_1^*)_\#\rho_{|P}- (\nabla \psi_0^*)_\#\rho_{|P}}\label{e:fatal2}\\
&=200\A^3\C\D^3 \CDM\sca{\psi_0-\psi_1}{(\nabla \psi_1^*)_\#\rho- (\nabla \psi_0^*)_\#\rho}
\end{align}
(the last equality holds because $\mathcal{F}'$ is a partition). This proves Theorem \ref{t:varineqjohn}.
\end{proof}

\section{Proof of Theorems \ref{t:mainjohn}, \ref{t:explosebord} and \ref{t:hallin}: bounded domains} \label{s:proofjohn}

We prove successively Theorems \ref{t:mainjohn}, \ref{t:explosebord} and \ref{t:hallin} in Sections \ref{s:johnconclusion}, \ref{s:stabpotexpl} and \ref{s:hallinmap} respectively. These three proofs are based on the variance inequality proved in Theorem \ref{t:varineqjohn}. We apply it to the dual Brenier potentials: recalling that the Brenier potential associated to the optimal transport from $\rho$ to $\mu$ is denoted by $\phi_\mu$ (and is extended by $+\infty$ outside $\X$), its Legendre transform defined by
\begin{equation}\label{e:dualbrenier}
\forall y\in\mathcal{Y}, \qquad \psi_\mu(y)=\sup_{x\in\mathcal{X}}\ \sca{x}{y}-\phi_\mu(x)
\end{equation}
is called the dual Brenier potential. Since $\X$ is bounded, $\psi_\mu$ is a (convex) ${\rm diam}(\X)$-Lipschitz function on $\mathcal{Y}$. Moreover, since $\phi_\mu$ is convex, $\psi_\mu^*=\phi_\mu$ where $\psi_\mu^*$ is defined via \eqref{eq:conv-conjugate}, or equivalently 
$$
\forall x\in\X, \qquad \psi_\mu^*(x)=\sup_{y\in\mathcal{Y}} \ \sca{x}{y}-\psi_\mu(y).
$$

We start in Section \ref{s:bomancubes} with the construction of a family $\mathcal{F}$ of cubes which will be used in all three proofs as the covering of the Boman chain condition.

\subsection{Construction of the Boman cubes} \label{s:bomancubes}
Let $\X$ be a non-empty proper John domain of $\R^d$, and consider its Whitney decomposition into closed dyadic cubes $P_j$ recalled in Lemma \ref{l:whitney}. Let us fix $\sigma =\min(10/9,(d+1)/d)$. For any $j$, we set $Q_j=\sigma P_j$; then $\mathcal{F}$ is defined as the set of all $Q_j$. In particular, due to \eqref{e:distwhitney}, there holds
\begin{equation}\label{e:bomanloin}
    \frac{1}{2}{\sqrt {d}}\ell (Q_{j})\leq \dist(Q_{j},\X ^{c})\leq 5{\sqrt {d}}\ell (Q_{j}).
\end{equation}

In the next sections, we will use this decomposition $(Q_j)_{j\in\mathcal{F}}$ to verify the Boman chain condition. Of course,  condition \eqref{e:conseccubeoverlap} depends on the density $\rho$ that we put on $\X$; but it is already possible to verify \eqref{e:A} and \eqref{e:BJohn}, which are properties of $\mathcal{F}$ (and $\X$), independent of $\rho$. 

Let us first check \eqref{e:A}. Denote by $\mathcal{W}$ the set of cubes in the Whitney decomposition of $\X$. For $x\in\X$,
$$
\sum_{Q_j\in\mathcal{F}} \chi_{\sigma Q_j}(x)=\sum_{P_j\in \mathcal{W}} \chi_{\sigma^2 P_j}(x)\leq 12^d
$$
where we used $\sigma^2<5/4$ and the last property of the Whitney decomposition listed in Lemma \ref{l:whitney}. For $x\notin \X$, notice that $x$ does not belong to any cube $\sigma Q_j$ due to \eqref{e:bomanloin}, and therefore \eqref{e:A} also holds in this case. Then, \eqref{e:BJohn} is a consequence of $\X$ being a John domain, according to the following result:

\begin{lemma}(see \cite[Lemma 2.1]{boman}) \label{l:boman}
Let $\X\subset\R^d$ be a John domain, and let $m$ denote the Lebesgue measure on $\X$. Then $m$ satisfies the Boman chain condition for the above covering $(Q_j)_{j\in\mathcal{F}}$ and some $\A,\B,\C>0$. In particular, \eqref{e:BJohn} holds (since \eqref{e:BJohn} is a property of $\mathcal{F}$).
\end{lemma}
In Remark \ref{r:constants} we compute $\A,\B,\C$ in terms of $\alpha,\beta$. Let us mention that the converse of Lemma \ref{l:boman} has been investigated in \cite[Theorem 3.1]{bomanjohn}. Let us mention here that the result \cite[Lemma 2.1]{boman} is proved with a slightly different definition of John domains, but it was shown to be equivalent to Definition \ref{d:john} in \cite[Lemma 2.10]{nakki}. Actually, the initial definition of John domains by John \cite[Page 402]{john} was even another one, but all definitions are equivalent according to the results of \cite[Lemma 2.7]{martio} and \cite[Lemma 2.10]{nakki}.

\subsection{Proof of Theorem \ref{t:mainjohn}}\label{s:johnconclusion}
We assume that $\X$ is a John domain and that $0<m_\rho\leq \rho\leq M_\rho<+\infty$. In order to apply Theorem \ref{t:varineqjohn}, we first verify that its assumptions are verified. Concerning the Boman chain condition, we already checked \eqref{e:A} and \eqref{e:BJohn} in Section \ref{s:bomancubes}. The property \eqref{e:conseccubeoverlap} follows from \eqref{e:whitneylength} and the fact that $Q_j=\sigma P_j$, together with the fact that $0<m_\rho\leq \rho\leq M_\rho<+\infty$. These upper and lower bounds on $\rho$ also imply conditions \eqref{e:ctdoubling} and \eqref{e:supratio}. Therefore we may apply Theorem \ref{t:varineqjohn} in $\X$, i.e., \eqref{e:varineqgeneral} holds. We consider this inequality for $\psi_0=\psi_\mu$ and $\psi_1=\psi_\nu$, which are both continuous in $\mathcal{Y}$.
Since $(\nabla\psi_\mu^*)_{\#}\rho=\mu$ and $(\nabla\psi_\nu^*)_{\#}\rho=\nu$ we obtain
\begin{equation}\label{e:jaiappl}
{\rm Var}_\rho\left(\phi_\mu-\phi_\nu\right) \leq C' \sca{\psi_\mu-\psi_\nu}{\nu-\mu}.
\end{equation}
Then, one can check that $\psi_\mu-\psi_\nu$ is ${\rm diam}(\X)$-Lipschitz continuous, hence by the Kantorovich-Rubinstein duality formula for the $W_1$ distance (see \cite[Theorem 1.14]{villani}) the following upper bound for the right-hand side in \eqref{e:jaiappl} holds:
\begin{equation}\label{e:kr}
\sca{\psi_\mu-\psi_\nu}{\nu- \mu}\leq {\rm diam}(\X) W_1(\mu,\nu).
\end{equation}
Finally, we observe that $\Var_\rho(\phi_\mu-\phi_\nu)=\|\phi_\nu-\phi_\nu\|_{L^2(\rho)}^2$ since $\int_\X\phi_\mu d\rho = \int_\X\phi_\nu d\rho=0$. All in all, putting \eqref{e:jaiappl} and \eqref{e:kr} together, we have obtained \eqref{e:stabpotjohn}, i.e., the stability of Brenier potentials. 

Like for the log-concave case, the stability of Brenier maps \eqref{e:stabmapjohn2} now follows from Proposition \ref{p:GNdelmer}. 
Since $u=\phi_\mu$ and $v=\phi_\nu$ are $R_{\mathcal{Y}}$-Lipschitz and convex in $\R^d$, and the boundary of $\partial \X$ is assumed rectifiable, \eqref{e:stabpotjohn} and Proposition \ref{p:GNdelmer} imply \eqref{e:stabmapjohn2}  since
\begin{align}
M_\rho^{-1/2}\|T_\mu-T_\nu\|_{L^2(\rho)}&\leq\|T_\mu-T_\nu\|_{L^2(\lambda,\X)}\leq C_{\rho,\mathcal{Y}}\|\phi_\mu-\phi_{\nu}\|_{L^2(\lambda,\X)}^{1/3}\nonumber\\
&\leq C_{\rho,\mathcal{Y}}m_{\rho}^{-1/6}\|\phi_\mu-\phi_{\nu}\|_{L^2(\rho)}^{1/3}\leq C_{\rho,\mathcal{Y}}m_{\rho}^{-1/6}W_1(\mu,\nu)^{1/6}.\label{e:chainGN}
\end{align}

\begin{remark}\label{r:poincaresob}
As already mentioned, our proof of Lemma \ref{l:gluevarjohn} is inspired by techniques developed in \cite{bojarski} (building upon anterior work  \cite{boman}) to prove Poincaré-Sobolev inequalities in John domains. In both cases, a global variance is controlled by a sum of local variances.

For Sobolev-Poincaré inequalities, it is known \cite[Theorem 1.1]{buckley} that the property for a domain $\X\subset \R^d$ with finite volume to support Sobolev-Poincaré inequalities is almost equivalent to $\X$ being a John domain. The two properties are truly equivalent if $\X$ is assumed in addition to satisfy a separation condition described in \cite{buckley}, which is verified for instance for simply connected planar domains. Analogously, Theorem \ref{t:counterexample} hints towards an equivalence for a bounded domain $\X$ between the property that it is John (or some related property) and that the uniform probability density on $\X$ has the potential stability property. However, like for Sobolev-Poincaré inequalities, there are obvious caveats, which explain the need of a separation property for the equivalence to hold: for instance in dimension $d\geq 2$, it is possible to remove a countable number of points to a ball to make it a non-John domain, while the optimal transport potentials and maps with uniform source measure are unchanged (and therefore the stability property is unchanged under this operation).
\end{remark}

\begin{remark}\label{r:constants}
We proved that $\rho$ satisfies the Boman chain condition for some parameters $\A,\B,\C$. Let us briefly explain how to express these parameters in terms of the parameter $\alpha$ of the John domain $\X$ (see Definition \ref{d:john}). For this, we need to dig into Boman's proof of Lemma \ref{l:boman} given in \cite[Lemma 2.1]{boman}. Due to properties of the Whitney decomposition (Lemma \ref{l:whitney}) and the construction of the family of cubes in the Boman chain condition (see Section \ref{s:bomancubes}), we may take $\A=12^d$ and
\begin{equation}\label{e:explicitC}
\C=\left(\frac{4}{\sigma-1}\right)^d \sup_{Q}\frac{M_{\rho_Q}}{m_{\rho_Q}}
\end{equation}
which is finite because bounded above by $\left(\frac{4}{\sigma-1}\right)^d\frac{M_\rho}{m_\rho}$. Both $\A$ and $\C$ are thus independent of $\alpha$. The proof of \cite[Lemma 2.1]{boman} gives $\B=\frac{40d}{\alpha'}$ if we take the same definition of John domains as Boman, i.e., $\dist(\gamma(t),\partial \X)\geq \alpha' |\gamma(t)-x|$ instead of \eqref{e:foralltinellgammat} (and the ratio of $\alpha$ and $\alpha'$ is bounded above and below according to \cite[Lemma 2.10]{nakki}). Also, recall that due to Theorem \ref{t:stability-compact},
\begin{equation}\label{e:explicitCDM}
\CDM=e \frac{M_{\rho_Q}}{m_{\rho_Q}}  R_{\mathcal{Y}} \diam(Q),
\end{equation}
which is uniformly bounded above by $e \frac{M_{\rho}}{m_{\rho}}  R_{\mathcal{Y}} \diam(\X)$. Therefore the constants $C_{\rho,\mathcal{Y}}$ in \eqref{e:stabpotjohn} and \eqref{e:stabmapjohn2} are totally explicit.
\end{remark}


\subsection{Extension to finite unions of John domains}
\label{s:moredomains}
In this section, we show that optimal transport maps are still stable if one replaces in Theorem \ref{t:mainjohn} the assumption that $\X$ is a John domain by the assumption that $\X$ is a finite union of John domains. On the side of Brenier potentials, we observed in Section \ref{s:uniqueness} that they are not unique as soon as $\X$ has at least two connected components separated by a positive distance, but the proof of the proposition below shows that stability of Brenier potentials can be recovered by shifting them by an appropriate constant on each connected component, see \eqref{e:rugby} below.
\begin{proposition}
Let $\X=J_1\cup \ldots \cup J_k\subset\R^d$ be a finite union of John domains $J_i$ whose closures pairwise do not intersect, and with rectifiable boundaries. Let $\rho$ be a probability density on $\X$, bounded from above and below by positive constants. Then, for any compact set $\mathcal{Y}$, there exists $C_{\rho,\mathcal{Y}}>0$ such that for any probability measures $\mu,\nu$ supported in $\mathcal{Y}$,
\begin{equation}\label{e:stabmapjohn2finiteunion}
\|T_\mu-T_\nu\|_{L^2(\rho)}\leq C_{\rho,\mathcal{Y}}W_1(\mu,\nu)^{\frac16}.
\end{equation} 
\end{proposition}
\begin{proof}
Denote by $T_\mu=\nabla \phi_\mu$ and $T_\nu=\nabla\phi_\nu$ the optimal transports from $\rho$ to $\mu,\nu$. The functions $\phi_\mu$ and $\phi_\nu$ are convex according to Brenier's theorem, but in this proof (and only here) we do not assume $\int_\X\phi_\mu d\rho=\int_\X\phi_\nu d\rho=0$.

We consider $\rho_i$ the restriction of $\rho$ to $J_i$. We denote by $\tilde{\rho}_i=\rho_i/\rho_i(J_i)$ the associated probability density. According to Theorem \ref{t:varineqjohn}, for any $\psi_0,\psi_1\in\mathcal{C}^0(\mathcal{Y})$, 
$$
{\rm Var}_{\tilde{\rho}_i} (\psi_1^*-\psi_0^*)\leq C_{\rho,\mathcal{Y}}\langle \psi_1-\psi_0\mid \nabla \psi_{0\#}^*\tilde{\rho}_i- \nabla \psi_{1\#}^*\tilde{\rho}_i\rangle
$$
Applying this inequality to $\psi_0=\phi_\mu^*$ and $\psi_1=\phi_\nu^*$, and multiplying by $\rho_i(J_i)$ on both sides, we obtain
$$
\rho_i(J_i){\rm Var}_{\tilde{\rho}_i} (\phi_\mu-\phi_\nu)\leq C_{\rho,\mathcal{Y}}\langle \phi_\nu^*-\phi_\mu^*\mid \nabla \phi_{\mu\#}\rho_i- \nabla \phi_{\nu\#}\rho_i\rangle.
$$
Let $c_i=\int_\X (\phi_\mu - \phi_\nu)d\tilde{\rho}_i$. Then summing the above inequality over $i$ we deduce 
\begin{equation}\label{e:rugby}
\sum_{i=1}^k \|\phi_\mu-\phi_\nu-c_i\|_{L^2(\rho_i)}^2  \leq C_{\rho,\mathcal{Y}}\langle \phi_\nu^*-\phi_\mu^*\mid \nabla \phi_{\mu\#}\rho- \nabla \phi_{\nu\#}\rho\rangle \leq  C_{\rho,\mathcal{Y}}W_1(\mu,\nu).
\end{equation}
where in the last inequality we used Kantorovich-Rubinstein duality.
To get stability of maps we apply Proposition \ref{p:GNdelmer} in $J_i$, which gives
$$
\|T_\mu-T_\nu\|_{L^2(\rho_i)}^2 \leq C_{\rho,\mathcal{Y}}\|\phi_\mu-\phi_\nu-c_i\|_{L^2(\rho_i)}^{2/3}.
$$
Summing over $i$ we get
\begin{equation}\label{e:nwzel}
\|T_\mu-T_\nu\|_{L^2(\rho)}^2 \leq C_{\rho,\mathcal{Y}}\sum_{i=1}^k\|\phi_\mu-\phi_\nu-c_i\|_{L^2(\rho_i)}^{2/3}\leq C_{\rho,\mathcal{Y},k}\Bigl(\sum_{i=1}^k\|\phi_\mu-\phi_\nu-c_i\|_{L^2(\rho_i)}^{2}\Bigr)^{1/3}
\end{equation}
where  we used $\sum_{i=1}^k a_i \leq k^{2/3}\Bigl(\sum_{i=1}^k a_i^3\Bigr)^{1/3}$ for any $a_i\geq 0$, which is a consequence of Hölder's inequality. Combining \eqref{e:rugby} and \eqref{e:nwzel} we obtain \eqref{e:stabmapjohn2finiteunion}.
\end{proof}

\begin{remark} 
It is explained in \cite[Section 3.5]{letrouit} that similar stability bounds still hold if the source $\rho$ is bounded above and below on a domain with an \emph{infinite} number of connected components, as soon as the measure of the connected components decays sufficiently quickly (when they are non-increasingly ordered).
\end{remark}

\subsection{Proof of Theorem \ref{t:explosebord}}\label{s:stabpotexpl}

The proof of Theorem \ref{t:explosebord} follows broadly the same lines as that of Theorem \ref{t:mainjohn}. Again we first need to check that the assumptions of Theorem \ref{t:varineqjohn} are verified. 

Let us observe that due to \eqref{e:encadrement} and \eqref{e:bomanloin}, there exists $\E>0$ such that if $Q\in\mathcal{F}$ and $x,y\in Q$, then $\E^{-1}\rho(y)\leq \rho(x)\leq \E\rho(y)$. This proves \eqref{e:supratio}. Now let us check the Boman chain condition.
The above observation also implies that if $Q,Q'$ intersect, then  $\E^{-2}\rho(y)\leq \rho(x) \leq \E^2\rho(y)$ for any $x\in Q$ and $y\in Q'$. Therefore, denoting by $\lambda$ the Lebesgue measure in $\R^d$,
\begin{equation*}
\E^{-2}\frac{\lambda(Q)}{\lambda(Q')}\leq \frac{\rho(Q)}{\rho(Q')}\leq  \E^2\frac{\lambda(Q)}{\lambda(Q')}
\end{equation*}
and
\begin{equation*}
\frac{\rho(Q\cap Q')}{\max(\rho(Q),\rho(Q'))}\geq \E^{-2}\frac{\lambda(Q\cap Q')}{\max(\lambda(Q),\lambda(Q'))}.
\end{equation*}
Now \eqref{e:conseccubeoverlap} can be deduced as follows. Recall that each Boman cube is obtained as $\sigma P$ where $P$ is a cube of the Whitney decomposition (see Section \ref{s:bomancubes}). Moreover, any two consecutive cubes in the Boman chains used to prove Lemma \ref{l:boman} are scalings of Whitney cubes sharing part of their boundary. One can see this directly from the proof in \cite[Lemma 2.1]{boman}, but this property also follows from the fact that if two Boman cubes $Q=\sigma P$ and $Q'=\sigma P'$ have non-empty intersection, then the original cubes $P$ and $P'$ have touching boundaries.  The latter fact can be proved by contradiction: if the boundaries of $P$ and $P'$ do not touch, assume without loss of generality $\ell(P)\geq \ell(P')$; then all Whitney cubes sharing part of their boundary with $P$ are of sidelength $\geq \ell(P)/4$ due to \eqref{e:whitneylength} (and they are different from $P'$), hence there exists at least one coordinate axis along which $P$ and $P'$ are separated by at least $\ell(P)/4$. The projections of the scalings $Q$ and $Q'$ along this coordinate axis are also at strictly positive distance (since $\sigma \leq 10/9$), which shows that $Q\cap Q'=\emptyset$. Using again \eqref{e:whitneylength}, we deduce that \eqref{e:conseccubeoverlap}  holds.  The other two conditions \eqref{e:A} and \eqref{e:BJohn} of the Boman chain condition have already been checked in Section \ref{s:bomancubes}. There remains to prove that \eqref{e:ctdoubling} holds:
\begin{lemma}\label{l:dddd}
The doubling condition  \eqref{e:ctdoubling} holds for some $\D<+\infty$.
\end{lemma}
\begin{proof} Let us set $r=\ct$.
We first assume $\delta\geq 0$. For any $x\in Q$ and $y\in rQ\cap\X$, there holds $\dist(y,\partial\X)\leq \dist(x,\partial\X)+r\sqrt{d}\ell(Q)\leq (1+2r)\dist(x,\partial\X)$ according to \eqref{e:bomanloin}. Hence \eqref{e:encadrement} yields
$$
\rho(y)\leq c_2 d(y,\partial \X)^{\delta}\leq c_2(1+2r)^{\delta}d(x,\partial\X)^{\delta}\leq c_1^{-1}c_2(1+2r)^{\delta}\rho(x).
$$
We deduce that $\rho(rQ)\leq c_1^{-1}c_2(1+2r)^{\delta}r^d \rho(Q)$.

We turn to the case $\delta\in (-1,0)$. Recall that $\X$ is a bounded Lipschitz domain, hence some neighborhood of its boundary may be covered by a finite number of bi-Lipschitz charts. For $\varepsilon>0$ we set
$$
\X_\varepsilon=\{x\in\X \mid d(x,\partial\X)\leq \varepsilon\}
$$ 
and we denote by $\mathcal{F}_\varepsilon$ the set of elements $Q$ of $\mathcal{F}$ such that $rQ\subset \X_\varepsilon\cup \mathcal{X}^c$. By \eqref{e:bomanloin}, for any $\varepsilon>0$ there exists $\varepsilon'$ such that all cubes $Q\in\mathcal{F}$ whose center is at distance $\leq \varepsilon'$ from $\partial \X$ are in $\mathcal{F}_\varepsilon$. According to \eqref{e:bomanloin}, there is a uniform lower bound on the size of all cubes of the Whitney decomposition whose center is at distance $>\varepsilon'$ from $\partial\X$. Since the Whitney cubes have disjoint interiors, we deduce that the set $\mathcal{F}\setminus\mathcal{F}_\varepsilon$ is finite for any $\varepsilon>0$. Therefore we need to check \eqref{e:ctdoubling} only for elements of $\mathcal{F}_\varepsilon$.

We choose $\varepsilon$ sufficiently small so that $\mathcal{X}_\varepsilon$ is covered by bi-Lipschitz charts. Let us fix such a bi-Lipschitz chart $\Phi:U\rightarrow \R^+\times\R^{d-1}$, $\Phi(\partial \X\cap U)\subset \{0\}\times\R^{d-1}$ where $U\subset \overline{\X}$. Due to \eqref{e:bomanloin}, we deduce that there exists $C_{\rho,r}>0$ (depending only on $\rho$ and $r$, not on $Q$) such that for any cube $Q\in\mathcal{F}_\varepsilon$, if $\eta=\dist(Q,\partial \X)$, then 
$$
\Phi(rQ\cap \X)\subset \underbrace{[0,b_1]\times [a_2,b_2]\times\ldots\times [a_d,b_d]}_{:=R_Q}\subset \R^+\times\R^{d-1}
$$
for some $a_1=0,a_2,\ldots,a_d,b_1,\ldots,b_d$ verifying $|b_i-a_i|\leq C_{\rho,r}\eta$ for any $i\in[d]$. 
There holds
\begin{equation}\label{e:rrrQ}
    \rho(rQ)=(\Phi_{\#}\rho)(\Phi(rQ\cap\X)) \leq (\Phi_{\#}\rho)(R_Q).
\end{equation} 
Now since $\Phi$ is bi-Lipschitz, it follows from \eqref{e:encadrement} that there exist $c_1',c_2'>0$ such that for any $y=(y_1,y')\in \Phi(rQ\cap\X)\subset \R^+\times\R^{d-1}$,
\begin{equation}\label{e:encadrementchart}
c_1'y_1^\delta\leq (\Phi_{\#}\rho)(y)\leq c_2'y_1^\delta.
\end{equation}
Integrating over $R_Q$ we obtain that the right-hand side in \eqref{e:rrrQ} is $\leq C_{\rho,r}\eta^{d+\delta}$. Finally we observe that $\rho(Q)\geq C_{\rho,c_1}\eta^{d+\delta}$, again due to \eqref{e:bomanloin} and \eqref{e:encadrementchart}. This concludes the proof for $\delta\in (-1,0)$.
\end{proof}

All in all, Theorem \ref{t:varineqjohn} applies, and the stability of Brenier potentials \eqref{e:explstabpoto} follows as in Section \ref{s:johnconclusion} by Kantorovich-Rubinstein duality \eqref{e:kr}.

Concerning the stability of Brenier maps \eqref{e:1-8C}, the argument provided in \eqref{e:chainGN} must be modified in the present setting since $M_\rho=0$ when $\delta>0$, and $m_\rho=+\infty$ when $-1<\delta<0$. Let us show that for some $C_\rho>0$,
\begin{equation}\label{e:erosionblowup}
\forall \varepsilon>0, \qquad \rho(\X_\varepsilon)\leq C_\rho\varepsilon^{1+\delta}.
\end{equation}
Up to changing the constant $C_\rho$, it is sufficient to prove \eqref{e:erosionblowup} for $\varepsilon>0$ small enough.
Since $\X$ is assumed to be a Lipschitz domain, we may take as in Lemma \ref{l:dddd} a finite number of bi-Lipschitz local charts covering $\X_\varepsilon$. For each such bi-Lipschitz local chart $\Phi:U\rightarrow \R^+\times \R^{d-1}$, \eqref{e:encadrementchart} holds for any $y=(y_1,y')\in \Phi(U)$, and therefore integrating over $(y_1,y')\in \Phi(U\cap \X_\varepsilon)$  and using that the number of charts is finite, we obtain \eqref{e:erosionblowup}.
Let us now consider the density $\rho_\varepsilon=\rho_{|\X\setminus\X_\varepsilon}+\rho(x_0) \chi_{\X_{\varepsilon}}$ where $x_0\in\X\setminus \partial \X$ is fixed (and is independent of $\varepsilon$). Then $\rho_\varepsilon$ is not a probability measure, but we will use that its support is the closure of $\X$, independently of $\varepsilon$. Using $\|T_\mu\|_{L^\infty}\leq R_{\mathcal{Y}}$ and $\|T_\nu\|_{L^\infty}\leq R_{\mathcal{Y}}$ (recall that $R_{\mathcal{Y}}$ has been introduced in \eqref{e:RY}), \eqref{e:erosionblowup} implies
\begin{equation}\label{e:decoupe}
\|T_\mu-T_\nu\|_{L^2(\rho)}^2\leq C_{\rho,\mathcal{Y}}\varepsilon^{1+\delta}+\|T_\mu-T_\nu\|_{L^2(\rho_\varepsilon)}^2
\end{equation} 

We now prove an upper bound on $\|T_\mu-T_\nu\|_{L^2(\rho_\varepsilon)}^2$. We first consider the case $\delta\geq 0$. By the same computation as in \eqref{e:chainGN},
\begin{align}
\|T_{\mu}-T_{\nu}\|^3_{L^2(\rho_\varepsilon)}&\leq C_{\rho,\mathcal{Y}}\mathcal{H}^{d-1}(\partial \X)m_{\rho_\varepsilon}^{-1/2}\|\phi_\mu-\phi_\nu\|_{L^2(\rho_\varepsilon)}\nonumber\\
&\leq C_{\rho,\mathcal{Y}}\mathcal{H}^{d-1}(\partial \X)m_{\rho_\varepsilon}^{-1/2}(\|\phi_\mu-\phi_\nu\|_{L^2(\rho)}+\varepsilon^{1+\delta})\nonumber\\
&\leq C_{\rho,\mathcal{Y}}\mathcal{H}^{d-1}(\partial \X)m_{\rho_\varepsilon}^{-1/2}(W_1(\mu,\nu)^{1/2}+ \varepsilon^{1+\delta})\label{e:lecube}
\end{align}
To go from first to second line we used \eqref{e:erosionblowup} together with the observation that $\|\phi_\mu-\phi_\nu\|_{L^\infty(\X)}\leq 2R_{\mathcal{Y}}{\rm diam}(\X)$ since $\phi_\mu,\phi_\nu$ are $R_{\mathcal{Y}}$-Lipschitz and $\phi_\mu-\phi_\nu$ vanishes on $\X$ at least at one point due to the fact that $\int_\X \phi_\mu d\rho=\int_\X\phi_\nu d\rho=0$. We estimate the right-hand side of \eqref{e:lecube}: due to \eqref{e:encadrement} there holds 
$$
m_{\rho_\varepsilon}=\min(\rho_{|\X\setminus\X_\varepsilon},\rho(x_0))\geq \min(c_1\varepsilon^\delta, \rho(x_0)) \geq C_\rho \varepsilon^\delta,
$$
for some constant $C_\rho$ which depends on $\rho$ but not on $\varepsilon$ (for the last inequality we use that $\X$ is bounded, so $\varepsilon$ is bounded above). Hence
$m_{\rho_\varepsilon}^{-1/2}\leq C_\rho\varepsilon^{- \delta/2}$.
Plugging into \eqref{e:lecube} and then \eqref{e:decoupe}, we get
$$
\|T_{\mu}-T_{\nu}\|_{L^2(\rho)}^2\leq C_{\rho,\mathcal{Y}}\varepsilon^{1+\delta}+C_{\rho,\mathcal{Y}}\varepsilon^{-\delta/3}(W_1(\mu,\nu)^{1/2}+\varepsilon^{1+\delta})^{2/3}
$$
Choosing $\varepsilon=W_1(\mu,\nu)^{\frac{1}{2(1+\delta)}}$, we obtain \eqref{e:1-8C}. 

Finally, we consider the case $\delta\leq 0$. Then 
\begin{align}
\|T_{\mu}-T_{\nu}\|^3_{L^2(\rho_\varepsilon)}&\leq C_{\rho,\mathcal{Y}}\mathcal{H}^{d-1}(\partial \X)M_{\rho_\varepsilon}^{3/2}\|\phi_\mu-\phi_\nu\|_{L^2(\rho_\varepsilon)}\nonumber\\
&\leq C_{\rho,\mathcal{Y}}\mathcal{H}^{d-1}(\partial \X)M_{\rho_\varepsilon}^{3/2}\|\phi_\mu-\phi_\nu\|_{L^2(\rho)}\nonumber\\
&\leq C_{\rho,\mathcal{Y}}\mathcal{H}^{d-1}(\partial \X)M_{\rho_\varepsilon}^{3/2}W_1(\mu,\nu)^{1/2}\label{e:lecube2}
\end{align}
Proceeding as above with $M_{\rho_\varepsilon}\leq C_\rho \varepsilon^\delta$ and choosing $\varepsilon=W_1(\mu,\nu)^{1/3}$, we obtain \eqref{e:1-8C}.

\subsection{Proof of Theorem \ref{t:hallin}}\label{s:hallinmap}
The proof of Theorem \ref{t:hallin} follows the same lines as that of Theorem \ref{t:explosebord}. Again we check that the assumptions of Theorem \ref{t:varineqjohn} are satisfied. For \eqref{e:supratio} and the Boman chain condition, the argument is exactly the same as at the beginning of Section \ref{s:stabpotexpl}, we do not repeat it here. We only need to verify the doubling condition \eqref{e:ctdoubling}. 
\begin{lemma}
The doubling condition \eqref{e:ctdoubling} is verified for some $D<+\infty$.
\end{lemma}
\begin{proof}
Let $r=\ct$. 
Let $Q\in\mathcal{F}$ and let $\eta=\dist(Q,0)$ denote its distance to the origin. 
Then $Q\subset B(0,3\eta)$ since ${\rm diam}(Q)\leq \sqrt{d}\ell(Q)\leq 2\eta$ according to \eqref{e:bomanloin}. Hence $rQ\cap \X \subset B(0,4r\eta)\cap \X$ where $\X=B(0,1)$. Therefore $\rho(rQ\cap \X)\leq \rho(B(0,4r\eta)\cap\X)=\min(4r\eta,1)$. On the other hand, 
$$
\rho(Q)\geq C_d\ell(Q)^d(3\eta)^{1-d}\geq C_d\eta^d \eta^{1-d}= C_d \eta
$$
again due to the fact that $Q\subset B(0,3\eta)$ and then \eqref{e:bomanloin}. This proves that the doubling condition is verified. 
\end{proof}

Therefore Theorem \ref{t:varineqjohn} applies, and the stability of Brenier potentials \eqref{e:hallinpot} follows via Kantorovich-Rubinstein duality \eqref{e:kr} as in Section \ref{s:johnconclusion}.

For the stability of maps \eqref{e:hallinmap}, we set $K_\varepsilon=B(0,1)\setminus B(0,\varepsilon)$ and $\rho_{\varepsilon}=\rho_{|K_\varepsilon}$. We notice that $M_{\rho_\varepsilon}=c_d\varepsilon^{1-d}$ and that $\mathcal{H}^{d-1}(\partial K_\varepsilon)\leq C_d$ uniformly in $\varepsilon$. Applying Proposition \ref{p:GNdelmer} as in \eqref{e:chainGN} we obtain
\begin{align*}
\|T_\mu-T_\nu\|_{L^2(\rho_\varepsilon)}^2\leq M_{\rho_\varepsilon}\|T_\mu-T_\nu\|_{L^2(\lambda,K_\varepsilon)}^2 
&\leq C_{d,\mathcal{Y}}\mathcal{H}^{d-1}(\partial K_\varepsilon)^{2/3}M_{\rho_\varepsilon}\|\phi_\mu-\phi_\nu\|_{L^2(\lambda,K_\varepsilon)}^{2/3}\\
&\leq C_{d,\mathcal{Y}}M_{\rho_\varepsilon}\|\phi_\mu-\phi_\nu\|_{L^2(\rho)}^{2/3}\\
&\leq C_{d,\mathcal{Y}} \varepsilon^{1-d} W_1(\mu,\nu)^{1/3}.
\end{align*}
Finally, since $\rho(B(0,\varepsilon))=\varepsilon$, $\|T_\mu\|_{L^\infty}\leq R_{\mathcal{Y}}$ and $\|T_\nu\|_{L^\infty}\leq R_{\mathcal{Y}}$, we obtain
$$
\|T_\mu-T_\nu\|_{L^2(\rho)}^2\leq \|T_\mu-T_\nu\|_{L^2(\rho_\varepsilon)}^2+C_{\mathcal{Y}}\varepsilon\leq C_{d,\mathcal{Y}}(\varepsilon^{1-d} W_1(\mu,\nu)^{1/3}+\varepsilon).
$$
Choosing $\varepsilon=W_1(\mu,\nu)^{1/3d}$, we get \eqref{e:hallinmap}.

\begin{remark}\label{r:expl}
The assumptions of Theorems \ref{t:explosebord} and \ref{t:hallin} are the simplest ones which guarantee: (i) the Boman chain property \eqref{e:conseccubeoverlap}; (ii) the boundedness of the maximal operator $M$ defined in Appendix \ref{a:maximal}; (iii) property \eqref{e:erosionblowup}; (iv) an estimate of $m_{\rho_\varepsilon}$ and/or $M_{\rho_\varepsilon}$ as a power of $\varepsilon$. Theorem \ref{t:explosebord} would hold for any other family of $\rho$'s for which the above four ingredients can be proved.
\end{remark}

\section{Proof of Theorem \ref{t:counterexample}: counterexample to H\"older potential stability}\label{s:counterexample}

We first prove the result for $d=2$. The coordinates in $\R^2$ are denoted by $x_1,x_2$. We recall the concept of room and passage domain (see \cite{amick}, \cite[pp. 521-523]{couranthilbert} for instance). Figure \ref{fig:room} illustrates our description. Take a sequence of rooms (i.e., rectangles) $R_j$, contained in the unit ball of $\R^2$, $j\in\N^*$, $R_j$ symmetric with respect to the $x_1$-axis, and such that $\overline{R_j}\cap \overline{R_k}=\emptyset$ for $j\neq k$. These rooms are labelled as $R_1,R_2,\ldots$ along the increasing $x_1$-axis. They are joined together by passages (i.e., rectangles) $P_n$, $n\in\N^*$, $P_n$ symmetric with respect to the $x_1$-axis, of height $h_n$ much smaller than the height of the adjoining rooms $R_n$ and $R_{n+1}$. In particular the whole room and passage domain is contained in the unit ball. The passages $P_n$ are also ordered as $P_1,P_2,\ldots$ along the increasing $x_1$-axis. The sidelengths and centers of all these rectangles will be specified later.

\begin{figure}[h]
\centering

\begin{tikzpicture}

    \draw[dashed, gray, ->] (-1, 0) -- (10, 0) node[anchor=west] {$x_1$};
    \draw[dashed, gray, ->] (0.75, -3) -- (0.75, 3) node[anchor=south] {$x_2$};

    \draw[thick] (1.5,0.5) -- (1.5,2) -- (0,2) -- (0,-2) -- (1.5,-2) -- (1.5,-0.5) ;

    \draw[thick] (1.5, 0.5) -- (3, 0.5);
    \draw[thick] (1.5, -0.5) -- (3, -0.5);

    \draw[thick] (3,-0.5) -- (3, -1.33) -- (4, -1.33) -- (4, -0.25);
    \draw[thick] (4, 0.25) -- (4,1.33) -- (3,1.33) -- (3,0.5);

    \draw[thick] (4, 0.25) -- (5, 0.25);
    \draw[thick] (4, -0.25) -- (5, -0.25);

    \draw[thick] (5,-0.25) -- (5, -0.88) -- (5.67, -0.88) -- (5.67,-0.125);
    \draw[thick] (5,0.25) -- (5, 0.88) -- (5.67, 0.88) -- (5.67,0.125);

    \draw[thick] (5.67, 0.125) -- (6.33, 0.125);
    \draw[thick] (5.67, -0.125) -- (6.33, -0.125);

    \draw[thick] (6.33, -0.125) -- (6.33, -0.59) -- (6.78,-0.59) -- (6.78, -0.0625);
    \draw[thick] (6.33, 0.125) -- (6.33, 0.59) -- (6.78,0.59) -- (6.78, 0.0625);

    \draw[thick] (6.78, 0.0625) -- (7.22, 0.0625);
    \draw[thick] (6.78, -0.0625) -- (7.22, -0.0625);

    \draw[thick] (7.22, -0.0625) -- (7.22, -0.39) --  (7.52, -0.39) -- (7.52,-0.03);
    \draw[thick] (7.22, 0.0625) -- (7.22, 0.39) --  (7.52, 0.39) -- (7.52,0.03);

    \draw[thick] (7.52, 0.03) rectangle (7.84, 0.03);
    \draw[thick] (7.52, -0.03) rectangle (7.84, -0.03);

    \draw[thick] (7.84, -0.03) -- (7.84, -0.26) -- (8.01,-0.26) -- (8.01, -0.0156);
    \draw[thick] (7.84, 0.03) -- (7.84, 0.26) -- (8.01,0.26) -- (8.01, 0.0156);

    \draw[thick] (8.01, 0.0156) rectangle (8.21, 0.0156);
    \draw[thick] (8.01, -0.0156) rectangle (8.21, -0.0156);

    \draw[thick] (8.21,-0.0156) -- (8.21, -0.175) -- (8.34, -0.175)-- (8.34, -0.0078);
    \draw[thick] (8.21,0.0156) -- (8.21, 0.175) -- (8.34, 0.175)-- (8.34, 0.0078);

    \draw[thick] (8.34, 0.0078) rectangle (8.47, 0.0078);
    \draw[thick] (8.34, -0.0078) rectangle (8.47, -0.0078);

    \draw[thick] (8.47, -0.0078) -- (8.47, -0.117) -- (8.56, -0.117) -- (8.56, -0.0039);
    \draw[thick] (8.47, 0.0078) -- (8.47, 0.117) -- (8.56, 0.117) -- (8.56, 0.0039);

    \draw[thick] (8.56, -0.0039) rectangle (8.65, 0.0039);

    \node at (0.7, 1) {$R_1$};
    
    \node at (2.6, -0.25) {$P_1$};

    \node at (3.45, 1) {$R_2$};
    
    \node at (4.45, 0) {$P_2$};

     \draw (1.5, 0.1) -- (1.5, -0.1);  
    \node at (1.5, 0.3) {$t_1$};
    \draw (3, 0.1) -- (3, -0.1);  
    \node at (3, 0.3) {$t_1'$};
     \draw (4, 0.1) -- (4, -0.1);  
    \node at (3.8, 0.3) {$t_2$};
    \draw (5, 0.1) -- (5, -0.1);  
    \node at (5.2, 0.3) {$t_2'$};

    \draw[<->] (2.2, -0.5) -- (2.2, 0.5) node[midway, left, xshift=+3pt, yshift=-4pt] {$h_1$};

\end{tikzpicture}
\caption{Room and passage domain} \label{fig:room}
\end{figure}
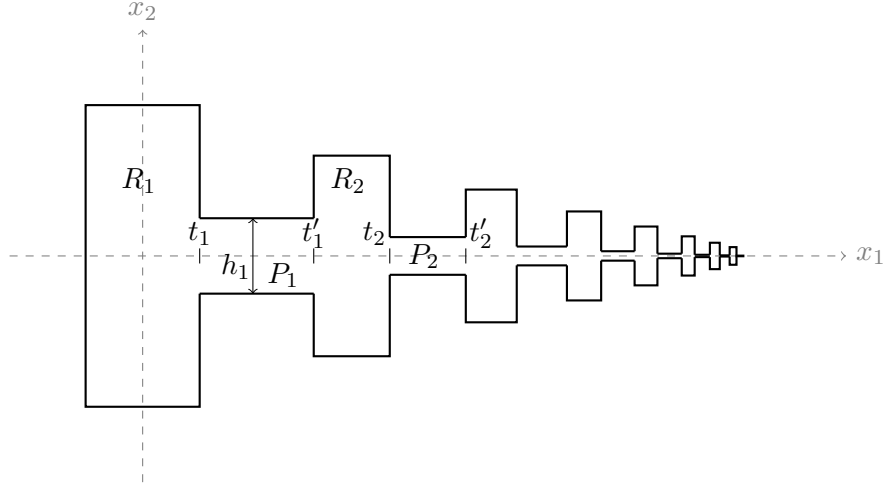

We denote by $t_n$ (resp. $t_n'$) the infimum (resp. the supremum) of the abscissa $x_1$ in $P_n$. In particular, $t_n<t_n'<t_{n+1}$ for any $n\in\N^*$. We fix the parameters $t_n,t_n'$ and the rectangles $R_n$ for $n\in\N^*$. Then, the heights $h_n$ are chosen sufficiently small, in order to satisfy the following property: for any $\delta>0$, 
\begin{equation}\label{e:tre}
\lim_{n\rightarrow +\infty} \frac{h_n^\delta}{|t_n-t_n'|^2\lambda(R_{n+1})}=0
\end{equation}
where $\lambda$ denotes the Lebesgue measure in $\R^d$. We denote by $\X$ the room and passage domain given by the union of all $R_n$'s and $P_n$'s, for $n\geq 1$. Finally, let $\rho$ denote a probability density on $\X$ bounded away from $0$ and $+\infty$, i.e., there exist $m_\rho>0$ and $M_\rho<+\infty$ such that $m_\rho\leq \rho\leq M_\rho$ in $\X$. For instance, one can take $\rho$ to be the normalized Lebesgue measure on the room and passage domain.

Consider the convex functions $\phi_{n}(x)=|x_1-t_n|$ and $\phi_{n}'(x)=|x_1-t_n'|$, which are the respective Brenier potentials for the quadratic transport from $\rho$ to $(\nabla \phi_n)_{\#}\rho$, and from $\rho$ to $(\nabla \phi_n')_{\#}\rho$. For $x=(x_1,x_2)\in\R^2$,
\begin{equation}\label{e:dormeur}
\phi_{n}'(x)-\phi_{n}(x)=\begin{cases} t_n'-t_n \text{ if } x_1\leq t_n \\ t_n-t_n' \text{ if } x_1\geq t_n'\end{cases}
\end{equation}
and 
\begin{equation}\label{e:dormeur2}
|\phi_{n}'(x)-\phi_{n}(x)|\leq |t_n-t_n'| \text{ if } x\in P_n.
\end{equation}
Consequently,
$$
\|\phi_{n}-\phi_{n}'\|_{L^2(\rho)}^2\geq |t_n-t_n'|^2(1-\rho(P_n)).
$$
We set 
$$
v_n=\rho(\{x\in\X \mid x_1\leq t_n\}) \qquad \text{and} \qquad w_n=\rho(\{x\in\X \mid x_1\geq t_n'\}).
$$
We have $v_n\rightarrow 1$, $w_n\rightarrow 0$ and $\rho(P_n)\rightarrow0$ as $n\rightarrow +\infty$. Besides, there holds for any $n\in\N^*$
\begin{equation}\label{e:=1}
v_n+w_n+\rho(P_n)=1.
\end{equation}
We bound the mean of $\phi_{n}'-\phi_{n}$ as follows: for $n$ large enough, according to \eqref{e:dormeur} and \eqref{e:dormeur2},
\begin{align*}
0\leq \int_\X (\phi_{n}'-\phi_{n})d\rho&\leq v_n(t_n'-t_n)+w_n(t_n-t_n')+|t_n'-t_n|\rho(P_n)\\
&=(1-2w_n)(t_n'-t_n)
\end{align*}
where the last equality comes from \eqref{e:=1}. Hence, still for $n$ large enough,
\begin{align}
\Var_\rho(\phi_{n}-\phi_{n}')&=\|\phi_{n}-\phi_{n}'\|_{L^2(\rho)}^2-\left(\int_\X (\phi_{n}-\phi_{n}')d\rho\right)^2\nonumber\\
&\geq (t_n'-t_n)^2(1-\rho(P_n)-(1-2w_n)^2)\nonumber\\
&=(t_n'-t_n)^2(4w_n-\rho(P_n)-4w_n^2)\nonumber\\
&\geq (t_n'-t_n)^2(\rho(R_{n+1})-\rho(P_n))\label{e:varpard}
\end{align}
since $w_n\rightarrow 0$ and $w_n\geq \rho(R_{n+1})$.

Fix $p\in[1,+\infty)$. Let us compute $W_p(\mu_n,\mu_n')$ for $n\in\N^*$. For any $n\in\N^*$, the measures $\mu_n=(\nabla\phi_n)_{\#}\rho$ and $\mu_n'=(\nabla\phi_n')_{\#}\rho$ are both supported on the union of two singletons, namely $A=(x_1,x_2)=(-1,0)$ and $B=(x_1,x_2)=(1,0)$. The subset of $\X$ given by the points $x\in\X$ such that $\nabla\phi_n(x)\neq \nabla\phi_n'(x)$ is exactly $P_n$. Since $\dist(A,B)=2$ we deduce $W_p(\mu_n,\mu_n')=2\rho(P_n)^{1/p}$.

Now we observe that $\rho(P_n)\leq M_\rho\lambda(\X)^{-1}h_n|t_n'-t_n|\leq h_n$ for $n$ large enough where $h_n$ denotes the height of $P_n$ and $\lambda$ denotes the Lebesgue measure in $\R^d$. 
Due to \eqref{e:tre}, $\rho(P_n)\leq h_n \leq \rho(R_{n+1})/2$ for $n$ large enough. Plugging into \eqref{e:varpard} we get that for any $q>0$ and $p\in[1,+\infty)$, for $n$ large enough,
$$
\frac{W_p(\mu_n,\mu_n')^q}{\Var_\rho(\phi_{n}-\phi_{n}')}\leq 2^{q+1}m_\rho\frac{h_n^{q/p}}{|t_n-t_n'|^2\lambda(R_{n+1})}\underset{n\rightarrow+\infty}{\longrightarrow} 0
$$
which concludes the proof when $d=2$. The generalization to any $d\geq 2$ is straightforward, considering hyperrectangles rooms and passages, instead of rectangles. Notice that in our example the target measures $\mu_n, \mu_n'$ are supported in the closure of the unit ball of $\R^d$, but the result actually holds for any fixed non-empty ball $B(0,r)\subset\R^d$ (supporting the target measures), by considering $\phi_n=\frac{1}{2}r|x-t_n|$ and $\phi_n'=\frac{1}{2}r|x-t_n'|$.

\begin{remark}
From the above construction it is straightforward to construct a 1d probability density $\rho$ on $\R$, whose support is a compact interval, and for which the Hölder stability property of Definition \ref{d:holder} does not hold. Let $\pi:\R^2\rightarrow\R$, $\pi(x_1,x_2)=x_1$ be the projection onto the first coordinate. Consider the 1d probability measure $\rho'=\pi_{\#}\rho$ on $\R$, where $\rho$ is the normalized Lebesgue measure on the domain constructed above (and represented in Figure \ref{fig:room}). Then the above computations show that the Brenier potentials associated to $\rho'$ do not satisfy the Hölder potential stability property. However, the Brenier maps with source measure $\rho'$ are stable: even better, they satisfy $\|T_\mu-T_\nu\|_{L^2(\rho')}=W_2(\mu,\nu)$, this equality being a property of all 1d probability measure.
\end{remark}

\begin{remark}\label{e:autrecontreex}
There exist other counterexamples to Hölder potential stability than room and passage domains, for instance domains with a sharp outward cusp: in $\R^2$, consider a domain contained in the half-plane $x_1\geq 0$ of the form $|x_2|\leq e^{-1/x_1^2}$ near $0$, and a probability density $\rho$ on this domain, bounded above and below by positive constants. Then the same family of potentials $|x_1-t|$ as above shows that Hölder potential stability does not hold in this case either. 
\end{remark}


\section{Proof of Theorem \ref{t:powlawdist}: generalized Cauchy distributions}\label{s:proofexception}

We describe in Section \ref{s:warmup} a natural attempt to prove Theorem \ref{t:powlawdist} which, although too rough to lead to a complete proof, gives clear insights about a key difficulty of the problem. Our final proof is an improvement over this attempt. Its broad lines are described at the end of Section \ref{s:warmup}.

\subsection{Warm-up and proof strategy}\label{s:warmup} 
In this section, which is merely illustrative, we assume for simplicity that $\rho(x)=c_\beta \langle x\rangle^{-\beta}$ with $\beta>d+2$, and $c_\beta>0$ is a normalizing constant.
We explain one of the key difficulties of the proof in this simple case. A natural approach would be to apply Theorem \ref{t:stability-compact}, but since it works only in bounded sets, we are naturally led to truncate the source measure $\rho$, as in the log-concave case handled in Section \ref{s:vrailogconc}. For $r\geq 0$, recall that $\Ball_r=B(0,r) \subset \R^d$ denotes the closed Euclidean ball of center $0$ and radius $r$. We denote by $\phi_{\mu,r}$ (resp. $\phi_{\nu,r}$) the restriction of $\phi_\mu$ (resp. $\phi_\nu$) to $\Ball_r$, and we consider the probability measures 
\begin{equation}\label{e:rhormurnur}
\rho_r=\frac{\rho_{|\Ball_r}}{\rho(\Ball_r)}, \qquad \mu_r=(\nabla \phi_{\mu,r})_{\#}\rho_r, \qquad \nu_r=(\nabla \phi_{\nu,r})_{\#}\rho_r.
\end{equation}
Finally we set $\psi_{\mu,r}=\phi_{\mu,r}^*$ and $\psi_{\nu,r}=\phi_{\nu,r}^*$. Simple arguments similar to \eqref{e:varwasssgau} based on the fact that $\phi_{\mu,r}$ and $\phi_{\nu,r}$ are $R_{\mathcal{Y}}$-Lipschitz imply that for large $r$,
$$
{\rm Var}_\rho(\phi_\mu-\phi_\nu)\leq {\rm Var}_{\rho_r}(\phi_{\mu,r}-\phi_{\nu,r})+C_{d,\mathcal{Y}}r^{d+2-\beta}.
$$
Since $\Ball_r$ is a bounded convex domain on which $\rho$ is bounded from above and below, Theorem \ref{t:stability-compact} yields
\begin{equation}\label{e:toonaive}
{\rm Var}_{\rho_r}(\phi_{\mu,r}-\phi_{\nu,r})\leq C_{d,\mathcal{Y}}\ r^{\beta+1}  \sca{\psi_{\mu,r}-\psi_{\nu,r}}{\nu_r- \mu_r}.
\end{equation}
Now we observe that $\psi_{\mu,r}-\psi_{\nu,r}$ is $2r$-Lipschitz (because ${\rm diam}(\Ball_r)=2r$) hence using Kantorovich-Rubinstein duality in the right-hand side of \eqref{e:toonaive} we obtain
\begin{equation}\label{e:CdYW1}
{\rm Var}_\rho(\phi_\mu-\phi_\nu)\leq C_{d,\mathcal{Y}}\ r^{d+2-\beta}+C_{d,\mathcal{Y}} \ r^{\beta+2} W_1(\mu_r,\nu_r).
\end{equation}
There remains to find an upper bound for $W_1(\mu_r,\nu_r)$ in terms of $W_1(\mu,\nu)$. The best we can hope for is that for large $r$,
$$
W_1(\mu_r,\nu_r)\leq W_1(\mu,\nu)+C_{\mathcal{Y}} \ r^{d-\beta}
$$
(due to \eqref{e:W1W1gau}). Plugging into \eqref{e:CdYW1} we see that the remainder term is too large to obtain \eqref{e:stabpotlogconc}, the method needs to be refined.

However, this strategy tells us that a key point is to be able to establish in $\Ball_r$ a variance inequality like \eqref{e:toonaive} with an improved dependence in $r$. This is precisely what we do in our proof of Theorem \ref{t:powlawdist}. It relies on the same basic idea as Theorem \ref{t:mainjohn}, that of decomposing $\X$ (here, $\R^d$) into convex sets $Q\in\mathcal{F}$, and we do this in a way that the ratio $M_{\rho_Q}/m_{\rho_Q}$ of the maximum density over the minimum density in each convex set $Q$ is bounded above uniformly in $Q\in\mathcal{F}$. Then we  apply the variance inequality in each $Q$, and finally we glue together all these local variance inequalities to obtain a global one. 

To do the gluing, it does not seem possible to adapt the strategy used to prove Theorem \ref{t:mainjohn}, namely to apply a Boman chain-type argument in $\Ball_r$ (and then letting $r\rightarrow+\infty$). Therefore, in the sequel, we follow a different strategy, which is based on the construction and the spectral analysis of a graph. We construct a decomposition of $\R^d$ into convex sets, and each of them is seen as one vertex of an infinite graph, endowed with a weight corresponding to its mass. The edges of this graph are given by those pairs of convex sets which intersect, and each edge is endowed with a weight equal to the mass of the intersection of the two convex sets. Then, the gluing argument for variance inequalities is obtained through a Cheeger inequality in finite truncations of the weighted graph that we just described. Using the Cheeger inequality is a robust strategy, with which it is possible for instance to recover Theorem \ref{t:mainjohn} (at least in Lipschitz domains), and this is why we use it here; we briefly explain in Remark \ref{r:sanscheeger} an alternative argument to replace the part of the proof where the Cheeger inequality is used, but which is quite specific to the power law case considered in Theorem \ref{t:powlawdist} and does not seem to generalize well.

\subsection{Laplacians in weighted graphs and Cheeger inequality}\label{s:laplagraphs}
We gather in this section general facts regarding Laplacians in infinite weighted graphs and the associated Cheeger inequality. Our presentation follows closely \cite{cheegerunbounded}.
\subsubsection{Weighted graphs}
Let $V$ be a countable set equipped with the discrete topology. Elements of $V$ are called vertices. We assume that $V$ is endowed with a function $\pds:V\rightarrow (0,\infty)$ which can be turned into a Radon measure on $V$ of full support  by the formula $\pds(U)=\sum_{i\in U}\pds_i$ for $U\subseteq V$ (we denote by $\pds_i$ the value of $\pds$ at $i\in V$). Then let $w:V\times V\rightarrow [0,\infty)$ be a symmetric function vanishing on the diagonal and satisfying
$$
\forall i\in V, \quad \sum_{j\in V}w_{ij}<+\infty
$$
(here an in the sequel, $w_{ij}=w(i,j)$).
We denote by $E$ the set of all edges, i.e., the set of all $(i,j)\in V\times V$ such that $w_{ij}>0$. If $(i,j)\in E$, we say that $i$ and $j$ are neighbors. The weighted graph $G=(V,E,\delta,w)$ is said locally finite if each vertex has only finitely many neighbors. 

\subsubsection{Graph Laplacians}\label{s:graphLaplacian}
Denote by $C_c(V)$ the space of real valued functions on $V$ with finite support.  We set 
$$
\ell^2(V,\pds)=\left\{u:V\rightarrow \R \mid \sum_{i\in V} \pds_i u(i)^2<+\infty\right\}
$$
and endow it with the scalar product $\langle u,v\rangle_{\pds} =\sum_{i\in V} \pds_i u(i)v(i)$
and the norm $\|u\|_{\pds}=\langle u,u\rangle_\pds^{1/2}$. Let the quadratic form $\mathcal{Q}=\mathcal{Q}_w$ with domain $\mathcal{D}$ be given by
$$
\mathcal{Q}(u)=\frac12 \sum_{i,j\in V} w_{ij}(u(i)-u(j))^2, \quad \mathcal{D}=\left\{u\in \ell^2(V,\delta) \mid \mathcal{Q}(u)<\infty\right\}.
$$ 
The corresponding positive selfadjoint operator $L$ acts as 
\begin{equation}\label{e:Laplacian}
Lu(i)=\frac{1}{\pds_i}\sum_{j\in V}w_{ij}(u(i)-u(j))
\end{equation}
(see \cite[Theorem 1.12]{kellerbook}). In the cases that will be considered in this paper, 
$$
\sum_{i\in V} \delta_i<+\infty
$$ 
and there exists $C>0$ such that 
\begin{equation}\label{e:gammaplusgrandquen}
\forall i\in V, \quad C\pds_i\geq \sum_{j\in V} w_{ij}.
\end{equation}
As a consequence, $L$ is a bounded operator in $\ell^2(V,\delta)$. Indeed,
$$
|Lu(i)|^2 \leq \delta_i^{-2} \sum_{j\in V} w_{ij}|u(i)-u(j)|^2 \sum_{j\in V} w_{ij} \leq C\delta_i^{-1}\sum_{j\in V} w_{ij}|u(i)-u(j)|^2
$$
hence
$$
\|Lu\|_\delta^2=\sum_{i\in V} |Lu(i)|^2 \delta_i \leq 2C\sum_{i,j\in V}w_{ij}(|u(i)|^2+|u(j)|^2)=4C\sum_{i,j\in V}w_{ij}|u(i)|^2\leq 4C\|u\|_{\delta}^2.
$$ We notice that the function $\mathbf{1}\in\mathcal{D}$ equal to $1$ on all $V$ is in the kernel of $L$. We set
$$
\lambda_2(L)=\inf\{ \mathcal{Q}(u) \mid  \|u\|_\pds=1, \langle u,\mathbf{1}\rangle_\pds=0\}.
$$
The Cheeger inequality is a lower bound on $\lambda_2(L)$ in terms of a constant measuring how well the graph is connected, called the isoperimetric constant of $G$.

\subsubsection{Isoperimetric constant and Cheeger inequality}

For $U\subset V$ we denote by
$$
\text{vol}(U)=\sum_{i\in U}\pds_i
$$
its volume and by
\begin{equation}\label{e:sizebdry}
|\partial U|=\sum_{i\in U, j\notin U} w_{ij}
\end{equation}
the size of its boundary. Finally the isoperimetric constant of $G$ is
\begin{equation}\label{e:defiso}
h=\inf_{\substack{U\subset V\\ 0<\vol(U)\leq \frac12\vol(V)}} \frac{|\partial U|}{\vol(U)}.
\end{equation}

The following Cheeger inequality will be instrumental. 
\begin{proposition} \label{p:cheeger}
If \eqref{e:gammaplusgrandquen} holds, then $\lambda_2(L)\geq h^2/2C$ where $C$ is the constant in \eqref{e:gammaplusgrandquen}.
\end{proposition}
For completeness, we provide in Appendix \ref{a:cheeger} an elementary proof assuming that $V$ is finite, which is sufficient for the present paper since we apply Proposition \ref{p:cheeger} only to finite graphs (in Lemma \ref{l:inf}). A proof of Proposition \ref{p:cheeger} for infinite graphs may be found in \cite[Theorem 3.5]{mokhtari}. A more general Cheeger inequality, applying as well to infinite graphs which do not verify \eqref{e:gammaplusgrandquen} has been proved in \cite[Theorem 13.4]{kellerbook}.

\subsection{Gluing variance inequalities}\label{s:gluing}
\begin{definition}\label{d:typeABD}
Let $\mnum,\mrat<+\infty$. An absolutely continuous probability measure $\rho(x)dx$ on a bounded set $\X$ is said to be of type $(\mnum,\mrat)$ with respect to a finite family $(Q_i)_{i\in V}$ of convex sets such that
\begin{equation}\label{e:Omega}
\spt(\rho)=\bigcup_{i\in V} Q_i
\end{equation}
if the following properties are verified:
\begin{enumerate}[(i)]
\item Each $Q_i$ intersects at most $\mnum$ of the sets $Q_j$'s, including itself.
\item For any $i\in V$,
\begin{equation} \label{e:unbddrho}
\frac{\sup_{x\in Q_i} \rho_{|Q_i}(x)}{\inf_{x\in Q_i} \rho_{|Q_i}(x)}\leq \mrat.
\end{equation}
\end{enumerate}
We simply say that $\rho$ is of type $(\mnum,\mrat)$ if it is of type $(\mnum,\mrat)$ with respect to some finite family $(Q_i)_{i\in V}$ satisfying the above conditions.
\end{definition}
To any probability measure of type $(\mnum,\mrat)$ it is possible to associate an undirected weighted graph as follows.
\begin{definition}[Weighted graph associated to a measure of type $(\mnum,\mrat)$]\label{d:weigraph}
To an absolutely continuous probability measure $\rho$ of type $(\mnum,\mrat)$ with respect to a finite family $(Q_i)_{i\in V}$ we associate a weighted graph $G=(V,E,\delta,w)$ as follows. Its vertices are given by the set $V$, and each $i\in V$ has weight 
$$
\pds_i=\rho(Q_i).
$$
The edges $E$ of $G$ are the couples $(i,j)$ for which  $w_{ij}>0$ (and in particular $i\neq j$). Each edge $(i,j)$ is endowed with a weight
$$
w_{ij}=\rho(Q_i\cap Q_j).
$$
The positive graph Laplacian $L$ on $\mathcal{D}\subset \ell^2(V,\pds)$ acts as \eqref{e:Laplacian}. The condition \eqref{e:gammaplusgrandquen} is verified with $C:=\mnum$.
\end{definition}

This subsection is devoted to the proof of the following result, which is an analogue of Theorem \ref{t:stability-compact} for measures $\rho$ of type $(\mnum,\mrat)$. Our proof follows and improves the strategy of proof of \cite[Proposition B.2]{barycenters}. Recall that the dual Brenier potentials $\psi_\mu, \psi_\nu$ have been introduced in \eqref{e:dualbrenier}.
\begin{theorem}\label{l:varineq}
Let $\rho$ be an absolutely continuous probability measure on a bounded set $\X\subset\R^d$ of type $(\mnum,\mrat)$, whose graph Laplacian associated to the decomposition \eqref{e:Omega} is denoted by $L$. We assume that $\lambda_2(L)>0$. Let $\mathcal{Y}\subset\R^d$ be compact. Then, for any probability measures $\mu,\nu$ supported in $\mathcal{Y}$, 
\begin{equation}\label{e:varineq}
\Var_\rho(\phi_\mu-\phi_\nu)\leq e   R_{\mathcal{Y}} \diam(\X) \mnum^2\mrat\left(1+\frac{2\mnum}{\lambda_2(L)}\right)\sca{\psi_\mu-\psi_\nu}{\nu-\mu}.
\end{equation}
\end{theorem}

Notice that since the set $\X$ is bounded, $\psi_\mu$ and $\psi_\nu$ are continuous functions in $\mathcal{Y}$. A key step in the proof is the following lemma, which plays the role of Lemma \ref{l:gluevarjohn} in the present case:
\begin{lemma}\label{l:gluevargauss}
For any $i\in V$, let $\rho_i=\frac{1}{\rho(Q_i)}\rho_{|Q_i}$. Let $f:\R^d\rightarrow \R$ be a continuous function.  Then
\begin{equation}\label{e:varrhof}
\Var_\rho(f)\leq \mnum\left(1+\frac{2\mnum}{\lambda_2(L)}\right)\sum_{i\in V}\pds_i\Var_{\rho_i}(f).
\end{equation}
\end{lemma}
\begin{proof}[Proof of Lemma \ref{l:gluevargauss}]
We set $S=\sum_{j\in V}\pds_j$. We first show that 
\begin{equation}\label{e:varineqpropB2}
\text{Var}_\rho(f)\leq S\sum_{i\in V} \pds_i\text{Var}_{\rho_i}(f)+\frac12\sum_{i,j\in V}(m_i-m_j)^2\pds_i\pds_j
\end{equation}
where $m_i=\int_{Q_i}fd\rho_i$. The proof goes as follows:
\begin{align*}
\Var_\rho(f)&=\frac12 \int_{\R^d\times\R^d}(f(x)-f(y))^2d\rho(x)d\rho(y)\\
&\leq \frac12 \sum_{i,j\in V} \int_{Q_i\times Q_j} (f(x)-f(y))^2d\rho(x)d\rho(y) \\
&=\frac12\sum_{i,j\in V}\int_{Q_i\times Q_j} (f(x)-m_i+m_i-m_j+m_j-f(y))^2d\rho(x)d\rho(y)\\
&=S\sum_{i\in V}\int_{Q_i}(f(x)-m_i)^2d\rho(x)+\frac12\sum_{i,j\in V}(m_i-m_j)^2\rho(Q_i)\rho(Q_j)
\end{align*}
which is exactly \eqref{e:varineqpropB2}.

Let $m=(m_i)_{i\in V}$. We denote the weighted mean of the $m_i$ by $\widetilde{m}=S^{-1}\sum_{i\in V} \pds_im_i$. The computations below show that $m-\widetilde{m}\in \ell^2(V,\pds)$ and $\mathcal{Q}(m-\widetilde{m})<+\infty$ (of course for this we need to assume that the right-hand side in \eqref{e:varrhof} is finite, which we do in the sequel). Moreover, $\langle m-\widetilde{m},\mathbf{1}\rangle_\pds=0$. We compute
\begin{align}
\frac{1}{2}\sum_{i,j\in V}(m_i-m_j)^2\pds_i\pds_j=S\|m-\widetilde{m}\|_\pds^2\nonumber&\leq \frac{S}{\lambda_2(L)}\langle m-\widetilde{m},L(m-\widetilde{m})\rangle_\pds\\
&=\frac{S}{2\lambda_2(L)}\sum_{i,j\in V}w_{ij}(m_i-m_j)^2.\label{e:mimj}
\end{align}

For any $i\neq j$ such that $w_{ij}>0$, denoting $m_{i\cap j}=\frac{1}{\rho(Q_i\cap Q_j)}\int_{Q_i\cap Q_j} fd\rho$, one has 
$$
\frac12(m_i-m_j)^2\leq (m_{i\cap j}-m_i)^2+(m_{i\cap j}-m_j)^2.
$$
For such $i,j\in V$,
\begin{align*}
(m_{i\cap j}-m_i)^2&=\left(\frac{1}{\rho(Q_i\cap Q_j)}\int_{Q_i\cap Q_j}(f-m_i)d\rho\right)^2\leq  \frac{1}{\rho(Q_i\cap Q_j)}\int_{Q_i\cap Q_j}(f-m_i)^2d\rho \\ 
&\leq  \frac{1}{\rho(Q_i\cap Q_j)}\int_{Q_i}(f-m_i)^2d\rho =\frac{\rho(Q_i)}{w_{ij}}\Var_{\rho_i}(f),
\end{align*}
and similarly for $(m_{i\cap j}-m_j)^2$, therefore we get
$$
\frac12 (m_i-m_j)^2\leq \frac{\pds_i}{w_{ij}}\text{Var}_{\rho_i}(f)+\frac{\pds_j}{w_{ij}}\text{Var}_{\rho_j}(f).
$$
Plugging into \eqref{e:mimj} yields
\begin{align*}
\frac12 \sum_{i,j\in V}(m_i-m_j)^2\pds_i\pds_j&\leq \frac{S}{\lambda_2(L)}\sum_{i\in V}\sum_{j \mid \rho(Q_i\cap Q_j)>0}(\pds_i\text{Var}_{\rho_i}(f)+\pds_j\text{Var}_{\rho_j}(f))\\
&\leq \frac{2\mnum S}{\lambda_2(L)}\sum_{i\in V} \pds_i\text{Var}_{\rho_i}(f).
\end{align*}
All in all and using that $S\leq \mnum$, we have proved \eqref{e:varrhof}.
\end{proof}

\begin{proof}[Proof of Theorem \ref{l:varineq}] 
Lemma \ref{l:gluevargauss} for $f=\phi_\mu-\phi_\nu$  gives 
\begin{equation}\label{e:ibonikis}
\Var_{\rho}(\phi_\mu-\phi_\nu)\leq \mnum\left(1+\frac{2\mnum}{\lambda_2(L)}\right) \sum_{i\in V} \pds_i\Var_{\rho_i}(\phi_\mu-\phi_\nu).
\end{equation}
For any $i\in V$, we apply Theorem \ref{t:stability-compact} to $\rho_i$ (supported in $Q_i$) with $\sigma$ being the normalized Lebesgue measure on $Q_i$. The functions $\psi_\mu$ and $\psi_\nu$ are both continuous in $\mathcal{Y}$. Since $\psi_\mu^*=\phi_\mu$ and $\psi_\nu^*=\phi_\nu$, we get 
\begin{equation}\label{e:pliplo}
\text{Var}_{\rho_i}(\phi_\mu-\phi_\nu)\leq e  R_{\mathcal{Y}} \diam(\X)\left(\frac{M_{\rho_i}}{m_{\rho_i}}\right)\sca{\psi_\mu-\psi_\nu}{(\nabla \phi_\nu)_\# \rho_i-(\nabla \phi_\mu)_\# \rho_i}
\end{equation}
where $M_{\rho_i}$ and $m_{\rho_i}$ are defined according to \eqref{e:supetinf}.
Finally we observe that
\begin{equation}\label{e:plusptiquA}
\begin{aligned}
\sum_{i\in V}\pds_i\sca{\psi_\mu-\psi_\nu}{(\nabla \phi_\nu)_\# \rho_i-(\nabla \phi_\mu)_\# \rho_i}&=\sum_{i\in V}\sca{\psi_\mu-\psi_\nu}{(\nabla \phi_\nu)_\# \rho_{|Q_i}-(\nabla \phi_\mu)_\# \rho_{|Q_i}}\\
&\leq A\sca{\psi_\mu-\psi_\nu}{(\nabla \phi_\nu)_\# \rho-(\nabla \phi_\mu)_\# \rho}.
\end{aligned}
\end{equation}
since the above sum contains only non-negative terms (due to \eqref{e:pliplo}) and each point $x\in\X$ belongs to at most $\mnum$ sets $Q_i$. Combining \eqref{e:unbddrho}, \eqref{e:ibonikis}, \eqref{e:pliplo} and \eqref{e:plusptiquA}, we get \eqref{e:varineq}.
\end{proof}

\subsection{Estimates for generalized Cauchy distributions}\label{s:powerlaws}
In the present section, assuming that $\rho$ is of the form
\begin{equation}\label{e:powerlaw}
\rho(x)=c(x)\langle x\rangle^{-\beta}, \qquad m\leq c(x)\leq M,
\end{equation}
we construct a decomposition of $\R^d$ into a family $(Q_i)_{i\in V}$ of convex sets satisfying Definition \ref{d:typeABD}, and we check all properties that will be needed to apply Theorem \ref{l:varineq}. As already mentioned in Section \ref{s:warmup}, it will be necessary to consider truncations $\rho_r$ of $\rho$ in order to apply Theorem \ref{l:varineq}, and therefore we consider for each $r$ a subfamily of $(Q_i)_{i\in V}$ which covers the support of $\rho_r$. 

\subsubsection{Decomposition of $\R^d$ and construction of a family of weighted graphs} \label{s:decompoRd}
Let $\rho$ be of the form \eqref{e:powerlaw}. We construct a locally finite family $(Q_j)_{j\in V}$ of subsets of $\R^d$ (depending on $\ppow$, omitted in the notation), indexed by $V=\N\times \{-1,1\}^d$. We refer the reader to Figure \ref{fig:powerlaw} for an illustration. 

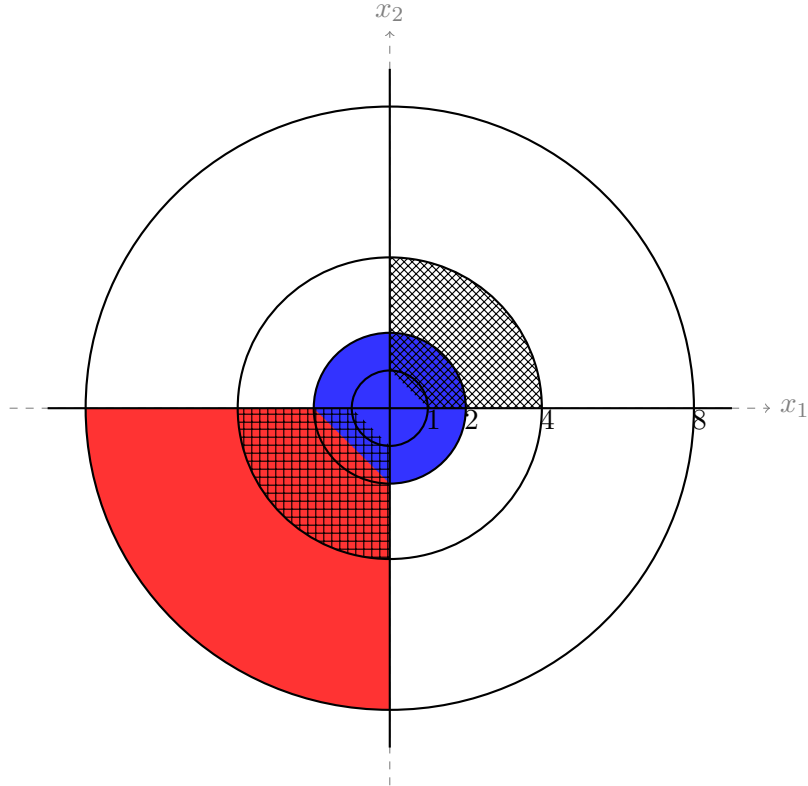
\begin{figure}[h]
\centering
\begin{tikzpicture}[scale=0.5]
        \draw[dashed, gray, ->] (-10, 0) -- (10, 0) node[anchor=west] {$x_1$};
    \draw[dashed, gray, ->] (0, -10) -- (0, 10) node[anchor=south] {$x_2$};
        \fill[blue!80,  opacity=0.5] 
        (0,0) -- (2,0)
        arc[start angle=0, end angle=360, radius=2] -- (0,0);
        \fill[red!80,  opacity=0.5] 
        (0,-2) -- (0,-8)
        arc[start angle=270, end angle=180, radius=8] -- (-2,0) -- (0,-2);
        \fill[pattern=grid, pattern color=black]
        (0,-1) -- (0,-4)
        arc[start angle=270, end angle=180, radius=4] -- (-1,0) -- (0,-1);
        \fill[pattern=crosshatch, pattern color=black]
        (1,0) -- (4,0)
        arc[start angle=0, end angle=90, radius=4] -- (0,1) -- (1,0);

    \draw[thick] (0,0) circle (1);
    \draw[thick] (0,0) circle (2);
    \draw[thick] (0,0) circle (4);
    \draw[thick] (0,0) circle (8);

    \draw[thick] (-9, 0) -- (9, 0);  
    \draw[thick] (0, -9) -- (0, 9);  

    \node at (1.15,-0.3) {$1$};
    \node at (2.15,-0.3) {$2$};
    \node at (4.15,-0.3)  {$4$};
    \node at (8.15,-0.3)  {$8$};
\end{tikzpicture}
\caption{The decomposition of $\R^d$ for $d=2$. Four domains $Q_{(J,\sigma)}$ are drawn: in blue, in red, with a grid and with crosshatches.}
\label{fig:powerlaw}
\end{figure}

We set $\Cor_0=B(0,2)$ and
\begin{equation}\label{e:Ink}
\Cor_{J}=B(0,2^{J+1})\setminus B(0,2^{J-1})
\end{equation}
for $J\in\N^*$. There holds
$$
\R^d=\bigcup_{J\in\N}\Cor_{J}.
$$
This is not a partition since some of the $\Cor_{J}$ overlap, but any point in $\R^d$ belongs to at most two of these sets. We consider the $2^d$ orthants
$$
\H_\sigma= \{x\in\R^d \mid \forall i\in[d], \ \sigma_ix_i\geq 0\}
$$
for $\sigma=(\sigma_1,\ldots,\sigma_d)\in\{-1,1\}^d$. We denote by ${\rm Conv}$ the convex-hull, and we finally define for $(J,\sigma)\in\N^*\times\{-1,1\}^d$
\begin{equation}\label{e:Conv}
Q_{(J,\sigma)}={\rm Conv}(\Cor_J\cap \H_\sigma)
\end{equation}
(which is included in the convex set $\H_\sigma$) and $Q_{(0,\sigma)}=\mathcal{C}_0=B(0,2)$ for any $\sigma\in\{-1,1\}^d$.

We consider as in Definition \ref{d:weigraph} the graph $G=(V,E,\delta,w)$ associated to the sets $Q_{(J,\sigma)}$. 
We could prove various properties of the graph $G$, for instance that it has a positive spectral gap. But since we need later to perform a truncation of $\R^d$ in order to prove Theorem \ref{t:logconcave} (as explained in Section \ref{s:warmup}), we instead define a family of truncated graph depending on a parameter $r$, and prove properties of this family which are ``uniform in $r$".

We define this family of truncated graphs (i.e., subgraphs of $G$) as follows, see Figure \ref{fig:graph} for an illustration. For any $r\geq 1$ of the form $r=2^n$ for some $n\in\N^*$, $G_r=(V_r,E_r,\delta_r,w_r)$ is the restriction of $G=(V,E,\delta,w)$ to the set of vertices $(J,\sigma)\in V$ such that $Q_{(J,\sigma)}\subset \Ball_r$. The weights on the vertices $V_r$ and edges $E_r$ are unchanged compared to those on $V,E$. Due to \eqref{e:Ink} and \eqref{e:Conv}, for any $r=2^n$,
$$
\rho_r=\frac{\rho_{|\Ball_r}}{\rho(\Ball_r)} \quad \text{verifies} \quad \spt(\rho_r)=\bigcup_{(J,\sigma)\in V_r} Q_{(J,\sigma)}.
$$

\begin{figure}[h]
\centering

\begin{tikzpicture}[scale=0.8]
    \draw[fill=black] (0,0) circle[radius=0.1];
    \draw[fill=black] (1, 1) circle[radius=0.1];
    \draw[fill=black] (2, 2) circle[radius=0.1];
    \draw[fill=black] (3, 3) circle[radius=0.1];
    \draw[fill=black] (-1,-1) circle[radius=0.1];
    \draw[fill=black] (-2,-2) circle[radius=0.1];
    \draw[fill=black] (-3,-3) circle[radius=0.1];
    \draw[fill=black] (1,-1) circle[radius=0.1];
    \draw[fill=black] (2,-2) circle[radius=0.1];
    \draw[fill=black] (3,-3) circle[radius=0.1];
    \draw[fill=black] (-1, 1) circle[radius=0.1];
    \draw[fill=black] (-2, 2) circle[radius=0.1];
    \draw[fill=black] (-3,3) circle[radius=0.1];

    \draw (-3,-3) -- (3,3);
    \draw (-3,3) -- (3,-3);
\draw[-] (0,0) .. controls (1.2,0.6) .. (2,2);
\draw[-] (0,0) .. controls (1.2,-0.6) .. (2,-2);
\draw[-] (0,0) .. controls (-1.2,0.6) .. (-2,2);
\draw[-] (0,0) .. controls (-1.2,-0.6) .. (-2,-2);

\draw[-] (1,1) .. controls (1.8,2.6) .. (3,3);
\draw[-] (1,-1) .. controls (1.8,-2.6) .. (3,-3);
\draw[-] (-1,1) .. controls (-1.8,2.6) .. (-3,3);
\draw[-] (-1,-1) .. controls (-1.8,-2.6) .. (-3,-3);


\node at (2.1,1) {$Q_{(0,(1,1))}$};
\node at (3.1, 2) {$Q_{(1,(1,1))}$}; 
\node at (4, 3) {$Q_{(2,(1,1))}$}; 
\node at (2.2, -1) {$Q_{(0,(1,-1))}$};
\node at (3.2, -2) {$Q_{(1,(1,-1))}$}; 
\node at (4.2, -3) {$Q_{(2,(1,-1))}$}; 

\end{tikzpicture}
\caption{The graph $G_r$ for $r=2^3$ and $d=2$. Only a few labels are written.} 
\label{fig:graph}
\end{figure}
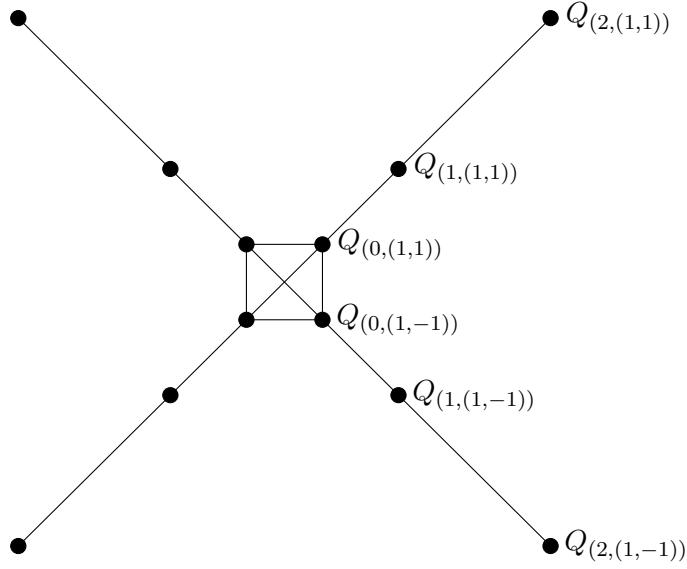

\subsubsection{Uniform type}\label{s:jpp3}

In Section \ref{s:conclulogconcave} we apply Theorem \ref{l:varineq} to $\rho_r$ for $r$ large enough. For this, we prove here a uniform upper bound (independent of $r$) on the constants $\mnum, \mrat$ appearing in the right-hand side of \eqref{e:varineq}:
\begin{lemma}\label{l:type}
There exist $\mnum, \mrat<+\infty$ such that for any $r=2^n$, $n\in\N^*$, $\rho_r$ is of type $(\mnum,\mrat)$ with respect to $(Q_{(J,\sigma)})_{(J,\sigma)\in V_r}$.
\end{lemma}
\begin{proof} 
Let $(J,\sigma)\in V$ and assume $J\neq 0$ (for $J=0$, the number of intersections with other sets is automatically bounded and \eqref{e:unbddrho} may be checked by hand). We observe that 
$$
Q_{(J,\sigma)}\subset \Bigl\{ x\in\R^d\mid \sum_{i=1}^d \sigma_i x_i\geq 2^{J-1}\Bigr\}\subset \R^d\setminus B(0,d^{-1/2}2^{J-1})
$$
hence
\begin{equation}\label{e:deboule}
Q_{(J,\sigma)}\subset B(0,2^{J+1})\setminus B(0,d^{-1/2}2^{J-1}).
\end{equation}
If $Q_{(J,\sigma)}$ and $Q_{(J',\sigma')}$ intersect for some $(J',\sigma')\in V$, then  due to \eqref{e:deboule}, $2^{J'+1}\geq d^{-1/2}2^{J-1}$ and $2^{J+1}\geq d^{-1/2}2^{J'-1}$. The number of such $J'$ is bounded above by $5+\log_2(d)$, hence the number of $Q_{(J',\sigma')}$ intersecting $Q_{(J,\sigma)}$ is at most $2^d(5+\log_2(d))$, which is independent of $(J,\sigma)\in V$ and $r$. Besides, \eqref{e:deboule} also implies that if $x,y\in Q_{(J,\sigma)}$, then
$$
\frac{\rho(x)}{\rho(y)}\leq \frac{M}{m}\frac{(d^{-1/2}2^{J-1})^{-\beta}}{(1+2^{2(J+1)})^{-\beta/2}}\leq \frac{M}{m}\frac{(d^{-1/2}2^{J-1})^{-\beta}}{2^{-\beta(J+2)}}\leq \frac{M}{m}(8\sqrt{d})^\beta
$$
which concludes the proof.
\end{proof}

\subsubsection{Uniform lower bound on the isoperimetric constant}\label{s:uniflobo}
We set $r_n=2^n$ for $n\in\N^*$ and denote by $h_{r_n}$ the isoperimetric constant of $G_{r_n}$ (see \eqref{e:defiso}).
\begin{lemma}\label{l:isomgaussian}
There holds
$\liminf_{n\rightarrow +\infty} h_{r_n}>0$.
\end{lemma}
\begin{proof}[Proof of Lemma \ref{l:isomgaussian}]
For $n\in\N^*$, the graph $G_{r_n}$ is finite, connected and non-empty.
Since $G_{r_n}$ is defined as a subgraph of $G$ with same weights, the volume of a set $U\subset V_{r_n}$ is the same as the volume of $U$ considered as a subset of $V$, and is equal to ${\vol}(U)=\sum_{i\in U}\rho(Q_i)$. Let $n_0\in\N^*$ such that ${\vol}(V_{r_{n_0}})\geq\frac34 \vol(V)$. We denote by $w_{\min}>0$ the minimum of the weights of edges between elements of $V_{r_{n_0}}$.
We consider $n\geq n_0$. For $U\subset V$ we denote by ${\rm Im}(U)$ its image in $\R^d$, i.e., 
$$
{\rm Im}(U)=\bigcup_{i\in U} Q_i.
$$
Let us consider $U\subset V_{r_n}$ with $0<{\rm vol}(U)\leq \frac12 {\rm vol}(V_{r_n})$. We seek for a lower bound on $|\partial U|/{\rm vol}(U)$.

Note that  $V_{r_{n_0}}$ can naturally be considered as a subset of $V_{r_n}$, itself a subset of $V$, and that the previous inequality yields ${\rm vol}(U)\leq \frac12 {\rm vol}(V)<{\rm vol}(V_{r_{n_0}})$. We deduce that there exists a vertex which is in $V_{r_{n_0}}$ but not in $U\subset V_{r_n}$. 

If $0\in {\rm Im}(U)$, pick $(J,\sigma)$ a vertex which is in $ V_{r_{n_0}}$ but not in $U$. Necessarily there exists $0\leq J'< J$ with $(J',\sigma)\in U$ and $(J'+1,\sigma)\notin U$. The edge between the vertices $(J',\sigma)$ and $(J'+1,\sigma)$ joins $U$ to $U^c$. Its weight is bounded below by $w_{\min}$. Since $\vol(U)$ is bounded above by $\vol(V)$ which is finite, this proves that $|\partial U|/\vol(U)$ is bounded below independently of $U$ (and $n\geq n_0$) if $0\in{\rm Im}(U)$.

Assume now $0\notin {\rm Im}(U)$. Denote by $\mathcal{S}$ the set of $\sigma\in\{-1,1\}^d$ such that there exists $J\in\N$ for which $(J,\sigma)\in U$. For $\sigma\in\mathcal{S}$, we denote by $J(\sigma)\in\N^*$ the smallest integer $J$ such that $(J,\sigma)\in U$. By definition, the edge between $(J(\sigma),\sigma)$ and $(J(\sigma)-1,\sigma)$ joins $U$ to $U^c$. We set $r_\sigma=2^{J(\sigma)-1}$.
Due to \eqref{e:Ink} and \eqref{e:Conv},
$$
\rho(Q_{(J(\sigma),\sigma)}\cap Q_{(J(\sigma)-1,\sigma)})\geq C_\rho m r_\sigma^{d-\beta}.
$$
Hence
\begin{equation}\label{e:volbdryA}
|\partial U|\geq C_\rho\sum_{\sigma\in\mathcal{S}}r_\sigma^{d-\beta}.
\end{equation}
To find an upper bound on $\vol(U)$ we notice that ${\rm Im}(U)\subset \bigcup_{\sigma\in\mathcal{S}} \H_\sigma \cap B(0,r_\sigma)^c$ due to the definition of $J(\sigma)$. Since each point in ${\rm Im}(U)$ belongs to at most $\mnum$ sets $Q_j$, we obtain
\begin{equation}\label{e:ptitvol}
\vol(U)\leq \mnum \rho({\rm Im}(U)) \leq  \mnum M C_\rho'\sum_{\sigma\in\mathcal{S}}\int_{r_\sigma}^{+\infty}s^{d-1-\beta}ds\leq C_\rho'' \sum_{\sigma\in\mathcal{S}}r_\sigma^{d-\beta}.
\end{equation}
The key fact is that the bounds \eqref{e:volbdryA} and \eqref{e:ptitvol} are independent of $r$.
Hence $|\partial U|/\vol(U)$ is bounded below independently of $U$ and $r$. This finishes the proof of Lemma \ref{l:isomgaussian}.
\end{proof}

\subsubsection{Uniform spectral gap} \label{s:jpp2}
It can be proved that $\lambda_2(L)>0$ where $L$ is the graph Laplacian (see Definition \ref{d:weigraph}) associated to the graph $G$ introduced in Section \ref{s:decompoRd}, but this is not what we will need in Section \ref{s:conclulogconcave}. Instead, we prove in this section a lower bound on $\lambda_2(L_{r_n})$ which is uniform in $n$, where $L_{r_n}$ is the graph Laplacian associated to the weighted graph $(V_{r_n},E_{r_n},\delta_{r_n},w_{r_n})$ introduced in Section \ref{s:decompoRd} (and $r_n=2^n$).

\begin{lemma}\label{l:inf}
There holds $\liminf_{n\rightarrow+\infty}\lambda_2(L_{r_n})>0$.
\end{lemma}
\begin{proof}[Proof of Lemma \ref{l:inf}]
Due to Proposition \ref{p:cheeger}, there holds $\lambda_2(L_{r_n})\geq \frac{h_{r_n}^2}{2\mnum }$ and $A$ is uniform in $n$ according to Lemma \ref{l:type}. Lemma \ref{l:inf} then follows from Lemma \ref{l:isomgaussian}.
\end{proof}

\begin{remark}\label{r:sanscheeger}
It is possible to prove Lemma \ref{l:inf} directly using the explicit form of the graphs and elementary (but fairly long) computations in the one-dimensional branches (see Figure \ref{fig:graph}).
\end{remark}

\subsection{Conclusion of the proof of Theorem \ref{t:powlawdist}}\label{s:conclulogconcave}
In the sequel, $r=2^n$ for some large $n\in\N$ which is fixed later. We denote by $\phi_{\mu,r}$ and $\phi_{\nu,r}$ the restriction of $\phi_\mu$ and $\phi_\nu$ to $\Res_r$ (extended by $+\infty$ outside $\Res_r$). We set $\psi_{\mu,r}=\phi_{\mu,r}^*$ and $\psi_{\nu,r}=\phi_{\nu,r}^*$; then $\psi_{\mu,r}^*=\phi_{\mu,r}$ and $\psi_{\nu,r}^*=\phi_{\nu,r}$. Recall also the notation $\rho_r$, $\mu_r$ and $\nu_r$ from \eqref{e:rhormurnur}. We apply Theorem \ref{l:varineq} together with Lemma \ref{l:type} and Lemma  \ref{l:inf}. This gives
\begin{equation}\label{e:varineqrhos}
\Var_{\rho_r}(\phi_{\mu,r}-\phi_{\nu,r})\leq C_{\rho,\mathcal{Y}}r\sca{\psi_{\mu,r}-\psi_{\nu,r}}{\nu_r-\mu_r}
\end{equation}
where $C_{\rho,\mathcal{Y}}$ does not depend on $r$. The arguments from \eqref{e:varineqrhosgau} to \eqref{e:phimuphinugau} work here too (they do not rely on any specific property of $\rho$), and we obtain exactly \eqref{e:phimuphinugau}, i.e.,
\begin{equation}\label{e:phimuphinu}
\|\phi_\mu-\phi_\nu\|_{L^2(\rho)}^2\leq C_{\rho,\mathcal{Y}}\left(r^2W_1(\mu,\nu)+r^2m_0(r)+m_1(r)^2+m_2(r)\right)
\end{equation}
where the truncated moments $m_0(r)$, $m_1(r)$, $m_2(r)$ are defined in \eqref{e:truncmoments}. Notice that $m_0(r)$, $m_1(r)$ and $m_2(r)$ are finite due to the assumption that $\beta>d+2$. Also, $\|\phi_\mu-\phi_\nu\|_{L^2(\rho)}$ is upper bounded by a constant which only depends on $\rho$ (for the same reason as in Section \ref{s:vrailogconc} below \eqref{e:upperboundmoment2gau}), thus we may assume in the sequel that $W_1(\mu,\nu)\leq 1/2$ in order to prove \eqref{e:yeah}.

Now recall that $\rho(x)=c(x)\langle x\rangle^{-\beta}$ with $\beta>d+2$ and $0<m\leq c(x)\leq M<+\infty$. Hence for $\ell\in\{0,1,2\}$ 
\begin{equation}\label{e:upperboundmoment2}
m_\ell(r)\leq C_{\rho,\ell} r^{d+\ell-\beta}.
\end{equation}
We plug into \eqref{e:phimuphinu}. Since $r^{d+2-\beta}\geq r^{2d+2-2\beta}$, we get 
$$
\|\phi_\mu-\phi_\nu\|_{L^2(\rho)}^2\leq C_{\rho,\mathcal{Y}}\left(r^2W_1(\mu,\nu)+r^{d+2-\beta}\right).
$$
We choose $r=2^n$ with $n=\lceil -(\beta-d)^{-1}\log_2(W_1(\mu,\nu))\rceil$ where $\lceil x\rceil=\inf\{j\in\N \mid j\geq x\}$. We obtain \eqref{e:yeah}.
Concerning \eqref{e:powlawmap} we use Proposition \ref{p:GNdelmer} for $K=\Res_r$ ($r=2^n$ with $n$ to be fixed later) and recall that $\phi_\mu,\phi_\nu$ are $R_{\mathcal{Y}}$-Lipschitz. We get
\begin{equation}\label{e:balch}
\|T_\mu-T_\nu\|^2_{L^2(\lambda,\Ball_r)}\leq C_\rho  r^{2(d-1)/3}R_{\mathcal{Y}}^{4/3}\|\phi_\mu-\phi_\nu\|^{2/3}_{L^2(\lambda,\Ball_r)}.
\end{equation} 
This is written for the Lebesgue measure $\lambda$, but since $m\langle x\rangle^{-\beta}\leq\rho(x)\leq M$ we deduce that
$$
\|T_\mu-T_\nu\|^2_{L^2(\rho,\Ball_r)}\leq C_{\rho,\mathcal{Y}}  r^{2(d-1)/3}r^{\beta/3}\|\phi_\mu-\phi_\nu\|^{2/3}_{L^2(\rho)}.
$$
Since $\|T_\mu\|_{L^\infty}\leq R_{\YSp}$ and $\|T_{\nu}\|_{L^\infty}\leq R_{\YSp}$, using \eqref{e:upperboundmoment2} for $\ell=0$ we get
$$
\|T_\mu-T_\nu\|^2_{L^2(\rho)}\leq  C_{\rho,\mathcal{Y}}  r^{2(d-1)/3}r^{\beta/3}\|\phi_\mu-\phi_\nu\|^{2/3}_{L^2(\rho)}+C_{\rho,\mathcal{Y}} r^{d-\beta}.
$$
We optimize over $r$, i.e, we choose $r=2^n$ with $n=\lceil -(2\beta-1-\frac{d}{2})^{-1} \log_2(\|\phi_\mu-\phi_\nu\|_{L^2(\rho)})\rceil$. We obtain
$$
\|T_\mu-T_\nu\|_{L^2(\rho)}\leq C_{\rho,\mathcal{Y}}\|\phi_\mu-\phi_\nu\|^{(\beta-d)/(4\beta-d-2)}_{L^2(\rho)}.
$$
Combining with \eqref{e:yeah} we get \eqref{e:powlawmap}.

Let us finally prove the sharpness of \eqref{e:yeah}, for an arbitrary $\rho$ satisfying the assumptions of Theorem \ref{t:powlawdist} for some $\beta>d+2$. We consider $\phi_r(x)=\max(|x|-r,0)-c_r$ for $r>0$, where $c_r$ is chosen in a way that $\int_{\R^d}\phi_r d\rho = 0$. Then, writing $f(r) \asymp g(r)$ if $f(r)/g(r)$ is bounded above and below uniformly as $r\rightarrow +\infty$, we have $\|\phi_{2r}-\phi_r\|_{L^2(\rho)}\asymp r^{\frac12(-\beta+d+2)}$. Also, noticing that $(\nabla\phi_r)_{\#}\rho$ is the linear combination of a Dirac mass in $0$ with mass $\rho(\mathcal{B}_r)$ and a uniform measure on the unit sphere with total mass $1-\rho(\mathcal{B}_r)$, we obtain 
$$
W_1((\nabla\phi_r)_{\#}\rho,(\nabla\phi_{2r})_{\#}\rho)=\rho(\mathcal{B}_{2r})-\rho(\mathcal{B}_r)\asymp r^{-\beta+d},
$$
hence $\|\phi_{2r}-\phi_r\|_{L^2(\rho)}\asymp W_1((\nabla\phi_r)_{\#}\rho,(\nabla\phi_{2r})_{\#}\rho)^\theta$. This shows the sharpness of the exponent $\theta$ in \eqref{e:yeah}. \qed

\bigskip 

The same test functions as above may be used to prove sharpness of \eqref{e:stabpotlogconc}, up to the log-factor, for $\rho(x)=c_d e^{-|x|^2/2}$:
\begin{proposition} \label{p:sharpexpo}
If $\rho(x)=c_d e^{-|x|^2/2}$ is the standard Gaussian, then \eqref{e:stabpotlogconc} is sharp up to the log-factor.
\end{proposition}
\begin{proof}
Consider as above the test function $\phi_r(x)=\max(|x|-r,0)-c_r$ for $r>0$. Setting $r'=r+r^{-1}$ for $r$ large, we have $\|\phi_{r'}-\phi_r\|_{L^2(\rho)}\asymp r^{\frac12(d-4)}e^{-r^2/4}$ and  
$$
W_1((\nabla\phi_{r'})_{\#}\rho,(\nabla\phi_{r})_{\#}\rho)\asymp r^{d-2}e^{-r^2/2}
$$
hence $\|\phi_{r'}-\phi_r\|_{L^2(\rho)}\asymp W_1^{\frac12}|\log W_1|^{-1}$ where $W_1$ stands for $W_1((\nabla\phi_{r'})_{\#}\rho,(\nabla\phi_{r})_{\#}\rho)$.
\end{proof}

\begin{remark}[Stability estimates for exponentially decaying densities]\label{r:exppowdist}
The proof strategy developed in this section may be adapted to prove stability estimates analogous to \eqref{e:yeah} and \eqref{e:powlawmap} in the case where $\rho(x)=c(x)e^{-\kappa |x|^{-\alpha}}$ for $\alpha,\kappa>0$ and $0<m\leq c(x)\leq M<+\infty$. The main difference lies in the definition of the cells $\Cor_J$: instead of \eqref{e:Ink}, we have to define the cells in a way that $\rho$ varies by not more than a multiplicative factor which is uniform over all cells. This can be made via a radial decomposition similar in spirit to \eqref{e:Ink}. At the same time, in order for each cell to be convex and to intersect only a bounded number of other cells, an angular decomposition needs to be performed. The details are rather lengthy and we shall not pursue this here.
\end{remark}

\subsection{Weighted Poincaré inequalities} \label{s:weightpoincare}
As a corollary of our proof technique, we recover weighted Poincaré inequalities for generalized Cauchy distributions similar to those established in the paper \cite{bobkovledoux}:
\begin{proposition} Let $\rho(x)=c(x)\langle x\rangle^{-\beta}$ be a probability measure on $\R^d$, with $\beta>d+2$ and $0<m\leq c(x)\leq M<+\infty$. Then there exists $C>0$ such that 
\begin{equation}\label{e:weightpoincare2}
\Var_\rho(f)\leq C\int_{\R^d} |\nabla f(x)|^2\langle x\rangle^2d\rho(x)
\end{equation} 
for any function $f$ for which the right-hand side is finite.
\end{proposition}
\begin{proof}
We notice that the $Q_j$ constructed in Section \ref{s:decompoRd} (and given by \eqref{e:Conv}) are obtained by homotheties and isometries of a single convex shape, namely the convex hull of the intersection of an orthant with an annulus whose outer radius is $4$ times its inradius. As a consequence, the Poincaré-Wirtinger inequality ${\rm Var}_{\rho_j}(f)\leq C_j\int_{Q_j}|\nabla f|^2d\rho_j$ holds, with a constant $C_j$ proportional to ${\rm diam}(Q_j)^2$. Due to the properties of the decomposition of $\R^d$ constructed in Section \ref{s:decompoRd} this implies 
$$
{\rm Var}_{\rho_j}(f)\leq C\int_{Q_j}|\nabla f(x)|^2 \langle x\rangle^2 d\rho_j(x)
$$ 
with a constant $C$ which is now independent of $j$ (and $f$). Applying Lemma \ref{l:gluevargauss} we obtain \eqref{e:weightpoincare2}.
\end{proof} 
The results in \cite{bobkovledoux} derive such inequalities with other methods and address the dependence of $C$ in the parameters $\beta,d$. 

\subsection{Comparison between gluing methods}\label{s:strategy} 
It is worth comparing the two methods used in this work to glue together variance inequalities. The  first one, based on the Boman chain condition, does not seem to generalize to unbounded supports, see comments in Section \ref{s:warmup}. On the contrary, the method based on spectral graph theory used in Section \ref{s:proofexception} could as well be used to prove stability for $\rho$ supported in a bounded domain $\X$, by constructing the weighted graph associated to the Boman decomposition of $\X$ (with one vertex per cube and one edge per pair of intersecting cubes). However, proving that the Cheeger constant of this graph is $>0$ quickly becomes a pain unless assuming some stronger regularity on $\partial\X$, for instance that $\X$ is a Lipschitz domain. This approach yields no improvement over Theorem \ref{t:mainjohn} but it could be leveraged to find other cases than Theorems \ref{t:explosebord} and \ref{t:hallin} where stability estimates can be proved even though $\rho$ is not assumed to be bounded away from $0$ and $+\infty$. We leave this for future work.

Let us finally illustrate on a simple example an interesting feature shared by both approaches. Let us consider a bounded domain $\X$ with dumbbell shape, i.e. two dsjoint balls joint by a thin tube of length $1$ and radius $\varepsilon\ll1$. We endow $\X$ with the Lebesgue measure, normalized to be a probability measure. Then $\B$ in Definition \ref{d:boman} is of order $\varepsilon^{-1}\gg1$, and $\D$ in \eqref{e:ctdoubling} is consequently $\gtrsim \varepsilon^{-d}$. If one now considers the graph associated to the Boman cube decomposition of $\X$, then one can check that $\lambda_2\lesssim \varepsilon^{d+1}$. In both cases we see that the constants in the variance inequalities \eqref{e:varineqgeneral} and \eqref{e:varineq} blow-up as $\varepsilon\rightarrow0$. Of course, this is related to the worsening of the Poincaré constant of $\X$ as $\varepsilon\rightarrow 0$.

\appendix

\section{Proof of Lemma \ref{l:maximal}}  \label{a:decompprop}\label{a:maximal}
For $x\in\X$ we denote by $\mathcal{F}_x$ the set of $Q\in\mathcal{F}$ for which $x\in Q$, and for $g\in L^1(\rho)$ we set
$$
Mg(x)=\sup_{Q\in \mathcal{F}_x} \frac{1}{\rho(BQ)}\int_{BQ}|g(y)|d\rho(y).
$$
We first prove that 
\begin{equation}\label{e:hardy}
\|Mg\|_{L^2(\rho)}\leq 2^{3/2}\D^{1/2}\|g\|_{L^2(\rho)}.
\end{equation}
For this it is sufficient to prove that there exists $\D>0$ such that 
\begin{equation}\label{e:alphag}
\forall \alpha>0, \quad \rho(S_\alpha)\leq \D\alpha^{-1}\|g\|_{L^1(\rho)}
\end{equation}
where $S_\alpha=\{x\in \X \mid Mg(x)>\alpha\}$. Indeed, \eqref{e:hardy} follows immediately from \eqref{e:alphag}, together with the $L^\infty(\rho)$-boundedness of $M$ and the Marcinkiewicz interpolation theorem \cite[Theorem 9.1 in Chapter VIII.9]{dibenedetto}.

Let us prove \eqref{e:alphag}. For any $x\in S_\alpha$, we pick $Q_x\in\mathcal{F}_x$ such that $\int_{BQ_x} |g(y)|d\rho(y)\geq \alpha\rho(BQ_x)$. Using the usual Vitali covering argument, we find a countable subset $S_\alpha'\subset S_\alpha$ such that $BQ_x\cap BQ_{x'}=\emptyset$ for any distinct $x,x'\in S_\alpha'$, and
$$
S_\alpha \subset \bigcup_{x\in S_\alpha'} \ct Q_x.
$$
Hence, using assumption \eqref{e:ctdoubling}, $B\geq 1$ and the disjointness of the sets $BQ_x$, we obtain
\begin{align*}
\rho(S_\alpha)\leq \D \sum_{x\in S_\alpha'}\rho(Q_x)\leq \D \sum_{x\in S_\alpha'}\rho(BQ_x)&\leq \alpha^{-1}\D\sum_{x\in S_\alpha'}\int_{BQ_x} |g(y)|d\rho(y)\\
&\leq \alpha^{-1}\D\|g\|_{L^1(\rho)}
\end{align*}
which concludes the proof of \eqref{e:alphag} and \eqref{e:hardy}. 

We turn to the proof of Lemma \ref{l:maximal}. Recall that $a_Q\geq 0$ for any $Q\in\mathcal{F}$ -- this is the only property of the sequence $(a_Q)_{Q\in\mathcal{F}}$ that we use below. We first observe that for any $y\in Q$ there holds $Mg(y)\geq \frac{1}{\rho(BQ)}\int_{BQ}|g(x)|d\rho(x)$. Hence, 
\begin{align*}
\Bigl|\int_{\R^d}\sum_{Q\in\mathcal{F}} a_Q \chi_{BQ}(x)g(x)d\rho(x)\Bigr|
&\leq \sum_{Q\in\mathcal{F}} a_Q \rho(B Q) \frac{1}{\rho(BQ)}\int_{BQ}|g(x)|d\rho(x) \\
&\leq \sum_{Q\in\mathcal{F}}a_Q \frac{\rho(B Q)}{\rho(Q)} \int_{Q}Mg(y) d\rho(y)\\
&\leq \D \sum_{Q\in\mathcal{F}}a_Q \int_{Q}Mg(y) d\rho(y) \\
&= D\int_{\R^d}\sum_{Q\in\mathcal{F}} a_{Q}\chi_{Q}(y)Mg(y) d\rho(y),
\end{align*}
where we used again assumption \eqref{e:ctdoubling} to get the last inequality.
Combining with \eqref{e:hardy}, we get
\begin{align*}
\Bigl|\int_{\R^d}\sum_{Q\in\mathcal{F}} a_Q \chi_{BQ}(x)g(x)d\rho(x)\Bigr| 
& \leq \D \Bigl\|\sum_{Q\in\mathcal{F}} a_Q\chi_{Q}\Bigr\|_{L^2(\rho)}\|Mg\|_{L^2(\rho)}\\
&\leq (2\D)^{3/2} \Bigl\|\sum_{Q\in\mathcal{F}}a_Q\chi_{Q}\Bigr\|_{L^2(\rho)}\|g\|_{L^2(\rho)}
\end{align*}
which concludes the proof, by duality.

\section{Proof of Proposition \ref{p:cheeger} (Cheeger inequality)}\label{a:cheeger}
For the proof of Proposition \ref{p:cheeger}, we keep the notation introduced in Section \ref{s:laplagraphs}. We adapt the proof of \cite[Theorem 2.2]{chung}. For $u\in \ell^2(V,\pds)$ we set 
\begin{equation}\label{e:rayleigh}
R(u)= \frac{\mathcal{Q}(u)}{\|u\|_\pds^2}.
\end{equation}
Let $\varepsilon>0$ and let $u:V\rightarrow \R$ such that $R(u)\leq \lambda_2(L)+\varepsilon$ and $\langle u,\mathbf{1}\rangle_\pds=0$. Up to relabelling the vertices we may assume that 
$$
u(1)\geq \ldots\geq u(n)
$$
where $n=|V|$.
For $k\in[n]$ let $S_k=\{1,\ldots,k\}$ and define
$$
\alpha_G=\min_{k\in [n]} \frac{|\partial S_k|}{\min({\rm vol}(S_k),{\rm vol}(V\setminus S_k))}.
$$
Let $r\in[n]$ denote the largest integer such that $\text{vol}(S_r)\leq \text{vol}(G)/2$. Since $\sum_{i\in[n]}\pds_{i}u(i)=0$,
$$
\sum_{i\in[n]} \pds_i u(i)^2=\min_{c\in\R} \sum_{i\in[n]} \pds_i(u(i)-c)^2 \leq \sum_{i\in[n]} \pds_i(u(i)-u(r))^2.
$$
To lighten notation below, we omit the notation ``$\in[n]$" when the index of a sum runs over $[n]$. We define the positive part and the negative part of $u(i)-u(r)$, denoted by $u_+(i)$ and $u_-(i)$ respectively, as follows:
$$
u_+(i)=\begin{cases} u(i)-u(r) \text{  if  } u(i)\geq u(r) \\ 0 \text{  otherwise}\end{cases}  \qquad u_-(i)=\begin{cases} |u(i)-u(r)| \text{  if  } u(i)\leq u(r) \\ 0 \text{  otherwise}\end{cases}
$$
There holds
\begin{align*}
2R(u)&=\frac{\sum_{i,j}w_{ij}(u(i)-u(j))^2}{\sum_i \pds_i u(i)^2}\\
&\geq\frac{\sum_{i,j}w_{ij}(u(i)-u(j))^2}{\sum_i \pds_i (u(i)-u(r))^2}\\
&\geq \frac{\sum_{i,j}w_{ij}\left((u_+(i)-u_+(j))^2+(u_-(i)-u_-(j))^2\right)}{\sum_i \pds_i (u_+(i)^2+u_-(i)^2)}.
\end{align*}
Without loss of generality we have $R(u_+)\leq R(u_-)$ and therefore $\lambda_2(L)+\varepsilon\geq R(u_+)$ since $\frac{a+c}{b+d}\geq \min(\frac{a}{b},\frac{c}{d})$ (if we assume $\lambda_2(L)+\varepsilon\geq R(u_-)$ instead, the subsequent computations can be carried out in the same way). Then we have
\begin{align}
\lambda_2(L)+\varepsilon\geq R(u_+)&=\frac12 \frac{\sum_{i,j}w_{ij}(u_+(i)-u_+(j))^2}{\sum_i \pds_i u_+(i)^2}\cdot 1\nonumber\\
&=\frac12 \frac{\sum_{i,j}w_{ij}(u_+(i)-u_+(j))^2}{\sum_i \pds_i u_+(i)^2}\cdot\frac{\sum_{i,j}w_{ij}(u_+(i)+u_+(j))^2}{\sum_{i,j}w_{ij}(u_+(i)+u_+(j))^2}\label{e:double}\\
&\geq \frac{\left(\sum_{i,j}w_{ij}|u_+(i)^2-u_+(j)^2|\right)^2}{8C\left(\sum_i \pds_i u_+(i)^2\right)^2}.\label{e:triple}
\end{align}
To go from \eqref{e:double} to \eqref{e:triple} we use Cauchy-Schwarz for the numerator; for the denominator we use \eqref{e:gammaplusgrandquen}.
Let us show the following identity for the numerator of \eqref{e:triple}:
\begin{equation}\label{e:othnum}
\frac12 \sum_{i,j}w_{ij}|u_+(i)^2-u_+(j)^2|=\sum_{i}(u_+(i)^2-u_+(i+1)^2)|\partial S_i|
\end{equation}
where for $S\subset V$, we set $\partial S=\{(k,\ell)\in V\mid k\in S, \ell\notin S\}$ and $|\partial S|=\sum_{k\in S, \ell\notin S}w_{k\ell}$. We have
\begin{align*}
\sum_{i}(u_+(i)^2-u_+(i+1)^2)|\partial S_i|&=\sum_{i}\sum_{\substack{k\leq i \\ \ell>i}}(u_+(i)^2-u_+(i+1)^2)w_{k\ell}\\
&=\sum_{k<\ell}w_{k\ell}\sum_{i=k}^{\ell-1}u_+(i)^2-u_+(i+1)^2\\
&=\sum_{k<\ell}w_{k\ell}(u_+(k)^2-u_+(\ell)^2)
\end{align*}
which concludes the proof of \eqref{e:othnum} since $u_+(1)\geq \ldots \geq u_+(n)$.

We continue our computations: we set $\vol'(S)=\min(\vol(S),\vol(G)-\vol(S))$, we have
\begin{align}
\lambda_2(L)+\varepsilon&\geq \frac{\left(\sum_{i}(u_+(i)^2-u_+(i+1)^2)\alpha_G\vol'(S_i)\right)^2}{2C\left(\sum_i \pds_i u_+(i)^2\right)^2}\label{e:alphaGapp}\\
&=\frac{\alpha_G^2}{2C} \frac{\left(\sum_{i}u_+(i)^2(\vol'(S_i)-\vol'(S_{i-1}))\right)^2}{\left(\sum_i \pds_i u_+(i)^2\right)^2} \nonumber 
\end{align}
where \eqref{e:alphaGapp} follows from \eqref{e:triple}, \eqref{e:othnum} and the definition of $\alpha_G$. Then, we observe that the quotient in the right-hand side is equal to $1$. Indeed, since $u_+(i)= 0$ for $i\geq r$, and $\vol'(S_i)=\vol(S_i)$ for $i<r$ by definition of $r$, we have 
$$
\sum_{i\in[n]}u_+(i)^2(\vol'(S_i)-\vol'(S_{i-1}))=\sum_{i=1}^{r-1} u_+(i)^2(\vol(S_i)-\vol(S_{i-1}))=\sum_{i=1}^{r-1}u_+(i)^2 \delta_i =\sum_{i\in[n]} u_+(i)^2 \delta_i.
$$
We conclude that $\lambda_2(L)+\varepsilon\geq \alpha_G^2/2C$. This being true for any $\varepsilon>0$, we obtain $\lambda_2(L)\geq \alpha_G^2/2C$, which concludes the proof since  $\alpha_G\geq h$.

\end{document}